\documentclass[10pt]{article}

\usepackage{amsthm}
\usepackage[section]{algorithm}
\usepackage[title]{appendix}
\usepackage[affil-it]{authblk}
\usepackage{fancyhdr}
\usepackage{ifpdf}
\usepackage{xcolor}
\usepackage{hyperref}
\usepackage{amsfonts,amsmath,amssymb,amsfonts,amscd}
\usepackage{boldtensors}
\usepackage{booktabs}
\usepackage{float}
\usepackage{mathtools}
\usepackage{soul}
\usepackage{subcaption}
\usepackage{url}
\usepackage{array}
\ifpdf
\hypersetup{
    pdftitle={Structure-Aware Tensorial Model Reduction},
    pdfauthor={Arjun Vijaywargiya, Eric C. Cyr, Anthony Gruber}
}
\fi

\usepackage[noend]{algpseudocode}
\usepackage{minted}
\setminted{
    style=pastie,
    fontsize=\small,
    linenos,
    frame=none
}

\DeclareMathOperator*{\argmin}{arg\,min}

\renewcommand{\epsilon}{\varepsilon}

\newcommand{\nn}[1]{\left\|#1\right\|}

\let\originalleft\left
\let\originalright\right
\renewcommand{\left}{\mathopen{}\mathclose\bgroup\originalleft}
\renewcommand{\right}{\aftergroup\egroup\originalright}

\usepackage{comment}

\definecolor{BYUblue}{RGB}{0, 61, 165}
\definecolor{sandiablue}{RGB}{0, 50, 90}
\definecolor{sandiared}{RGB}{130, 36, 51}
\definecolor{burntorange}{cmyk}{0, 0.65, 1, .09}
\definecolor{NDgold}{RGB}{174, 145, 66}

\usepackage{ifthen}
\newcommand{\rev}[3][default]{%
    \ifthenelse{\equal{#3}{}}{}{#3}%
}
\usepackage{booktabs, makecell, tablefootnote}

\usepackage[nameinlink]{cleveref}

\definecolor{darkgreen}{rgb}{0.0, 0.5, 0.0}
\definecolor{darkblue}{rgb}{0.0, 0.0, 0.6}
\definecolor{darkred}{rgb}{0.5, 0.0, 0.0}
\hypersetup{
    colorlinks=true,
    linkcolor=darkgreen,    
    citecolor=darkblue,     
    urlcolor=darkred,       
}

\crefformat{equation}{#2(#1)#3}
\numberwithin{equation}{section}

\theoremstyle{plain}
\newtheorem{theorem}{Theorem}
\newtheorem{proposition}{Proposition}
\newtheorem{lemma}{Lemma}
\newtheorem{corollary}{Corollary}

\theoremstyle{remark}
\newtheorem{remark}{Remark}

\theoremstyle{definition}
\newtheorem{definition}{Definition}

\theoremstyle{assumption}
\newtheorem{assumption}{Assumption}

\numberwithin{theorem}{section}
\numberwithin{proposition}{section}
\numberwithin{lemma}{section}
\numberwithin{corollary}{section}
\numberwithin{remark}{section}
\numberwithin{definition}{section}

\crefname{theorem}{Theorem}{Theorems}
\Crefname{theorem}{Theorem}{Theorems}
\crefname{proposition}{Proposition}{Propositions}
\Crefname{proposition}{Proposition}{Propositions}
\crefname{lemma}{Lemma}{Lemmata}
\Crefname{lemma}{Lemma}{Lemmata}
\crefname{corollary}{Corollary}{Corollaries}
\Crefname{corollary}{Corollary}{Corollaries}
\crefname{algorithm}{Algorithm}{Algorithms}
\Crefname{algorithm}{Algorithm}{Algorithms}
\crefname{appendix}{Appendix}{Appendices}
\Crefname{appendix}{Appendix}{Appendices}

\renewcommand{\arraystretch}{1.2}

\setminted{
    style=pastie,
    fontsize=\normalsize,
    linenos,
    frame=none
}

\begin{document} 

\pdfinfo{
   /Author (Arjun Vijaywargiya, Eric C. Cyr, Anthony Gruber)
   /Title (Structure-Aware Tensorial Model Reduction)
   /Keywords (parametric model reduction, tensor methods, Hamiltonian systems, structure preservation, scientific machine learning)
}

\title{Structure-Aware Tensorial Model Reduction}

\author[1]{Arjun~Vijaywargiya}
\author[2]{Eric C. Cyr}
\author[2]{Anthony Gruber\thanks{Corresponding author. E-mail: \href{mailto:adgrube@sandia.gov}{adgrube@sandia.gov}.}}

\affil[1]{\normalsize Oden Institute for Computational Engineering and Sciences, University of Texas at Austin}
\affil[2]{\normalsize Scientific Machine Learning, Center for Computing Research, Sandia National Laboratories}

\date{}

\maketitle

\vspace{-.25in}

\begin{abstract}
\noindent

This work investigates a two-stage method for constructing projection-based reduced-order models (ROMs) of parameterized partial differential equations (PDEs).  Based on established tensorial ROM methodology, the proposed approach reduces dimensionality offline by encoding solution snapshots using a multi-linear Tucker factorization, so that a reduced basis which varies nonlinearly with PDE parameters can be rapidly constructed online and used in a Galerkin ROM.  Two novel extensions of this strategy, tailored to the cases of structured PDEs and sparse parameter sampling, are presented: the construction of reduced bases orthonormalized with respect to a general discrete inner product, and the interpolation of encoded states via radial basis functions.  Basic representation and ROM error estimates are presented demonstrating the validity of these modifications, and the approach is challenged on \rev[anthony]{}{state-linear} examples where monolithic-basis ROMs are known to struggle, including a realistic instance of Maxwell's equations in 3D.  Results suggest that the proposed nonlinear basis ROM can effectively mitigate linear restrictions on Kolmogorov $n$-width while improving upon previous tensorial ROM technology, particularly in the highly nonlinear and data-limited regimes characteristic of practical use cases.

\end{abstract}




\section{Introduction}\label{sec:intro}

Consider a ``full-order model'' (FOM) of the form
\begin{equation}\label{eq:FOM}
    \dot{~x}(t,~\mu) = ~f\big(t,~x(t,~\mu),~\mu\big), \qquad ~x(0,~\mu) = ~x_0(~\mu),
\end{equation}
where $~x\in\mathbb{R}^N$ is the system state with initial condition $~x_0$ and $~\mu\in\mathbb{R}^p$ are system parameters.  Here, the model \eqref{eq:FOM} is assumed to arise via the semi-discretization of a parameterized system of first-order PDEs, so that $N\gg 1$ is some large number of ``coefficients'' or ``grid points''.  Standard projection-based model reduction relies on the following key assumption: since the solutions to \cref{eq:FOM} inherit regularity from the governing PDE,
the intrinsic dimensionality of the system \eqref{eq:FOM} is much less than $N$.  Therefore, the FOM should be well-approximated with a linear trial space of dimension $n\ll N$.

This low-dimensionality assumption has led to a wide variety of reduced-basis methods \cite{Benner2015,hesthaven2022reduced,parish2025residual} 
sharing the following basic principles.  First, the state variable $~x\approx ~U\hat{~x}$ is approximated using a linear trial space ${\rm im}\,~U$ spanned by the columns of a reduced-basis matrix $~U\in\mathbb{R}^{N\times n}$.  Inserting this approximation into the FOM \eqref{eq:FOM} and projecting onto a (potentially identical) subspace ${\rm im}\,~V$ for $~V\in\mathbb{R}^{N\times m}$ ($m\ll N$) yields a Petrov-Galerkin reduced-order model (ROM) in terms of the coefficients $\hat{~x}\in\mathbb{R}^n$,
\begin{equation}\label{eq:ROM}
    ~V^\intercal~U\dot{\hat{~x}}(t,~\mu) = ~V^\intercal~f\big(t,~U\hat{~x}(t,~\mu), ~\mu \big), \qquad \hat{~x}(0,~\mu) = (~V^\intercal~U)^{-1}~V^\intercal~x_0(~\mu).
\end{equation}
Provided that the right-hand side of \eqref{eq:ROM} is pre-computable or effectively approximated in a pre-computable fashion, the ROM \eqref{eq:ROM} provides a fast approximation to the dynamics of the desired FOM via the simulation of $m\ll N$ ordinary differential equations.  However, observe that many variations on this procedure are possible.  For example, the operations of time integration and Petrov-Galerkin projection are generally not commutative, and exchanging the order leads to an entirely different system with different dynamical properties.  This has led to a large amount of interest into ``discretize-then-optimize'' ROMs somewhat distinct from the ``optimize-then-discretize'' flavor presented in \eqref{eq:ROM} \cite{carlberg2017galerkin,quarteroni2014reduced,benner2014model,berman2023randomized}.

The utility of the ROM depends strongly on the construction of the reduced bases $~U$ and $~V$ defining the trial and test spaces for the Petrov-Galerkin projection. Many options are available at present (c.f. \cite{hesthaven2022reduced,Benner2015,Ghattas_Willcox_2021}), each with benefits and drawbacks.
Specializing to the Galerkin case $~U=~V$ for simplicity, many common and well-studied approaches are data-driven rather than analytical.  Arguably, the most popular such technique is the Proper Orthogonal Decomposition (POD), which aims to maximize the variance in a set of snapshot data consisting of discrete PDE trajectories
\begin{equation*}
    ~X\in\mathbb{R}^{N\times T\times P}, \qquad [~X]^i_{\alpha s} = x^i(t_\alpha, ~\mu_s),
\end{equation*}
where $x^i(t_{\alpha},~\mu_s)\in\mathbb{R}$ denotes the $i^{\rm th}$ component of the state $~x$ evaluated at time $t_{\alpha}$ and parameteric instance $~\mu_s$.  Considering the column-wise matricization $~X_{(1)} = \mathrm{cvec}_{23}~X \in \mathbb{R}^{N\times TP}$ of these snapshots, where $[{\rm cvec}_{23}~X]^i_{(s-1)T+\alpha} = [~X]^i_{\alpha s}$, POD minimizes the projection error $|~P_U^\perp~X_{(1)}|^2 \coloneqq |(~I-~U~U^\intercal)~X_{(1)}|^2$ onto the linear subspace ${\rm im}\,~U$ in the Frobenius norm.  This is a classical problem whose solution $~U\in\mathbb{R}^{N\times n}$ is efficiently given using the thin SVD $~X~W = ~U~\Sigma$.  Moreover, the reduced basis $~U$ constructed in this way comes with a guarantee of monotonic error decay for dynamics captured by the training data (c.f. \eqref{eq:proj-error-mono}).  For this reason, POD-ROMs have cemented themselves as versatile and trustworthy tools for a wide variety of applications \cite{lieu2005pod,tang2025applications,azam2013investigation,cardoso2009development,florez2017applications,le2017accelerating,peters2021mode}.

Despite the utility of the POD and linear reduced-basis methods in general, they face an insurmountable drawback in many cases of interest due to their limited approximability.  For a set of semidiscrete PDE solutions 
\[\mathcal{F} = \{~x(t,~\mu)\mid ~x\,\,\rm{solves}\,\,\eqref{eq:FOM}, t\in [0,T], ~\mu\in\Omega \},\]
whose parameters lie in some compact set $\Omega$, their worst-case approximation error is quantified by the so-called Kolmogorov n-width \cite{kolmogorov1936beste,fisher1980width,unger2019kolmogorov}:
\[d_n(\mathcal{F}) = \inf_{~U\in\mathbb{R}^{N\times n}} \sup_{~x\in\mathcal{F}}\,\inf_{\hat{~x}\in\mathbb{R}^n} \big|~x - ~U\hat{~x}\big|_{\mathcal{F}},\]
where $|\cdot|_{\mathcal{F}}$ denotes an appropriate Hilbert space norm.  This notion provides a fundamental lower bound on the error incurred by solutions to the ROM \eqref{eq:ROM} when $~U=~V$ is linear, and is known to decay algebraically like $n^{-1/2}$ for wave-like problems involving transport or discontinuities \cite{greif2019decay}.  This means that halving the error in a linear subspace ROM can require quadrupling its dimension, making models such as \eqref{eq:ROM} prohibitively costly in practical outer-loop scenarios.

A great effort has been made to mitigate this difficulty through the development of projection-based ROMs with nonlinear reduced bases \cite{lee2020model,barnett2022quadmanifold,geelen2023quadmanifold,schwerdtner2024greedy,diaz2025kernel}.  These methods forego optimality guarantees on basis quality in exchange for greater approximation power, often yielding ROMs with significantly lower errors in the small-basis regime.  However, most of these methods focus on introducing nonlinearity with respect to the state $~x$, not the parameters $~\mu$, and also face significant drawbacks not seen in the linear case.  Most notably, these nonlinear methods offer limited to no assurance that increasing the basis size will actually decrease the practical approximation error. 
Optimization errors that occur during nonlinear basis construction (i.e., the training process) typically cause the resulting ROM to saturate at some basis size $n$, beyond which the ROM error stops decreasing and there is a crossover point $n'\ll N$ where the linear POD ROM achieves lower errors.  
Moreover, these methods are comparatively slow both offline and online: the use of a nonlinear basis converts even linear FOMs into nonlinear ROMs.  When the trial (and/or test) space is not polynomial in the state $~x$, this means that hyper-reduction is required to achieve a scalable ROM, potentially mitigating any benefits of using a small basis size.  

On the other hand, many PDEs of interest in practical applications (e.g., Maxwell's equations considered in \Cref{sec:numerics}) are state-linear but vary nonlinearly with parameters.  These problems may also have a slowly decaying Kolmogorov $n$-width due to their parametric variation, but, importantly, this can be mitigated with a state-linear reduced basis.  It turns out that, by incorporating nonlinear parametric variation in the basis $~U=~U(~\mu)$, analogues of the linear subspace ROMs \eqref{eq:ROM} can be constructed that carry over benefits such as rigorous error guarantees and monotonic improvement with basis dimension.  The remainder of this work will discuss a particular approach to this based on the pioneering work \cite{mamonov2022interpolatory}, as well as various extensions which improve its applicability to the problems considered here.  

The work \cite{mamonov2022interpolatory} provides a new take on projection-based model reduction using ideas from the tensor literature.  Here, the key idea is to build the reduced basis $~U = ~U(~\mu)$ with a two-stage process that preserves the characteristic offline/offline cost splitting as well as provable guarantees on the method's error.  First, an offline tensor factorization (CP, Tucker, or tensor-train) is performed on the snapshot data $~X$, leading to low-dimensional multilinear approximation $\tilde{~X}\approx~X$ expressed in terms of orthonormal bases for the different tensor dimensions and a relatively small ``core tensor'' of coefficients living in some latent space.  Notice that this provides a large dimension reduction, but is not enough for an effective model reduction on its own: the orthonormal bases produced by tensor factorization do not possess a global ordering, and therefore no conclusions can be drawn about the relative importance (i.e., contribution to data variance) of each basis vector.  However, this problem is surmountable.  Given a new parameter instance $~\mu$ online, this ``reduced database'' $\tilde{~X}$ generated by tensor factorization is queried at $~\mu$ via an interpolation or local least-squares procedure, and a reduced basis $~U(~\mu)$ is built (at the latent-space level!) from the SVD of the resulting ``interpolated'' coefficients and the orthonormal bases computed offline during tensor factorization.  Importantly, the necessary interpolation and local SVD procedures executed online scale only with the size of the core tensor and not the size of the original dimensions in $~X$ (c.f. \Cref{alg:tucker_basis}), making this pipeline computationally feasible and relatively efficient.  As an added benefit, it was shown in $\cite{mamonov2022interpolatory}$ that the resulting reduced basis satisfies an approximation error bound involving the Tucker reconstruction error, the parametric interpolation error, and the error in the parameter-local SVD.  All together, this strategy provides a fast, interpretable, and theoretically grounded strategy for introducing parametric nonlinearity into projection-based ROMs.  

With this said, the tensorial ROM strategy in \cite{mamonov2022interpolatory} is limited in a few noteworthy ways.  First, it is applicable only to the construction of reduced bases which are Euclidean-orthonormal, i.e., those satisfying $~U^\intercal~U=~I$.  This precludes or makes cumbersome the use of certain classes of ROMs which are structure-preserving or otherwise optimal in energy-norm (e.g., \cite{vijaywargiya2025tensor,parish2023impact,afkham2018symplectic}), since there is no way to guarantee the $~M$-orthogonality condition $~U^\intercal~M~U=~I$ for a given discrete inner product or ``mass matrix'' $~M\in\mathbb{R}^{N\times N}$.  Additionally, the interpolation procedures presented in \cite{mamonov2022interpolatory}, including standard Lagrange interpolation and a nearest-neighbor least-squares fit, are functional but not entirely adequate in the presence of sparse or high-dimensional data.  Lagrange interpolants are expensive in high dimensions, while the local least-squares fit is piecewise-linear and globally non-smooth.  It is shown in the experiments of \Cref{sec:numerics} that both of these can create issues in practical scenarios.
Finally, despite a useful bound on approximation error,  guarantees on the tensorial ROM error are presently limited to the case of time-dependent parabolic problems \cite{mamonov2025priori}, hindering the overall trustworthiness of the method outside of this setting.

The present work aims to address the aforementioned issues with the tensorial ROM.  First, an extension of the basis construction is presented that yields an $~M$-orthonormalized reduced basis $~U=~U(~\mu)$ for any symmetric and positive definite choice of $~M$.  Next, an interpolation procedure using local Gaussian radial basis functions (GRBFs) is introduced, yielding a piecewise-smooth response surface with minimal overhead in the high-dimensional case.  It is shown that the approximation error estimate of \cite{mamonov2022interpolatory} also carries over to this setting, and an additional \textit{a priori} bound on ROM error is proven for systems of the form \eqref{eq:FOM} under mild assumptions.  Finally, the numerical performance of this modified strategy is investigated and shown to improve upon both the monolithic basis approach and the previous tensorial ROM.


To summarize, the main contribution of this work is a performant extension of the tensorial ROM via more general reduced bases and a modified snapshot interpolation procedure.  More precisely, this involves:
\begin{itemize}
    \item A novel tensorial ROM basis construction algorithm incorporating orthonormality constraints in nonstandard inner products.   
    \item A novel snapshot interpolation scheme using Gaussian radial basis functions, designed to improve performance in the presence of sparse parametric sampling and highly nonlinear parametric dependence.  
    \item Bounds on appropriately weighted projection error and ROM error providing theoretical support for the proposed strategy.
    \item Experiments demonstrating improved accuracy and stability on systems with gradient flow and Hamiltonian structure, including a complex 3D Maxwell case from electromagnetics.
    \item Example code reproducing the numerical results. 
\end{itemize}

The remainder of this work is structured as follows. \Cref{sec:method} reviews the tensorial ROM as introduced in \cite{mamonov2022interpolatory}, and \Cref{sec:ours} introduces the proposed modifications involving non-Euclidean orthonormality and GRBF interpolation.  \Cref{sec:analysis} then proves bounds on the representation error of the reduced basis and the resulting error in the ROM.  \Cref{sec:numerics} investigates the performance of the proposed approach on various problems in parametric PDEs with gradient or Hamiltonian structure.  Finally, \Cref{sec:conclusion} summarizes the work and discusses future directions.  \rev[anthony]{}{Supplemental sensitivity studies and comparisons are included in \Cref{app:sens,app:QB}.}

\section{Tensorial ROM}\label{sec:method}

The first goal is to describe the tensorial ROM of \cite{mamonov2022interpolatory} and the proposed modifications.  While this technology is generally applicable to a variety of tensor decompositions, the exposition here will be restricted to the Tucker decomposition and the Higher-Order SVD (HOSVD) which approximates it, where the main ideas are effectively exercised.  
\subsection{The Higher-Order SVD}
Consider the FOM \eqref{eq:FOM}, defined in terms of a state variable $~{x} = ~{x}(t,~{\mu})\in\mathbb{R}^N$ depending on time and a vector of parameters $~{\mu}\in\mathbb{R}^p$.  Given a tensor of snapshots drawn at various time and parameter instances, its HOSVD factorization satisfies the following quasi-optimality result \cite{delathauwer2000multilinear}.

\begin{theorem}\label{thm:tucker}
    Consider a snapshot tensor $~X\in\mathbb{R}^{N\times T\times P}$ and fixed Tucker ranks $\bar N < N$, $\bar T < T$, and $\bar P < P$.  The HOSVD reconstruction
    \begin{align*}
        \tilde{~X} = \sum_{i=1}^{\bar{N}}\sum_{\alpha=1}^{\bar{T}}\sum_{s=1}^{\bar{P}} [~{C}]_{i\alpha s}~{w}^i\otimes~{\tau}^\alpha\otimes~{\nu}^s \approx ~X,
    \end{align*}
    satisfies the quasi-optimal error bound 
    \begin{align*}
        \big|~X-\tilde{~X}\big|_F \leq \sqrt{3}\big|~X-~X_{\rm opt}\big|_F,
    \end{align*} 
    where $~X_{\rm opt}$ is the best approximation to $~X$ among all tensors with the given ranks.  Here, $~C\in\mathbb{R}^{\bar N\times \bar T\times \bar P}$ is known as the ``core tensor'' and $\{~{w}^i\}_{i=1}^{\bar{N}},\{~{\tau}^\alpha\}_{\alpha=1}^{\bar{T}},\{~{\nu}^s\}_{s=1}^{\bar{P}}$ are unordered orthonormal sets of vectors.  
\end{theorem}

\Cref{thm:tucker} provides the link between the HOSVD and its ``lower order'' namesake. On the other hand, note that the orthonormal bases it produces are globally unordered.  While each mode is individually ordered according to 
the SVD of the corresponding unfolding (c.f. \Cref{alg:hosvd}), there is no order relationship between the bases of different modes.  Since projection-based model reduction relies heavily on a small-and-ordered subset of basis vectors that are selected based on their information content, this creates additional work for the tensorial ROM approach mentioned in \Cref{sec:intro}.

To understand the computation of the HOSVD factorization, recall the following standard definition for tensor unfolding.

\begin{definition}[Mode-$k$ Unfolding \cite{Kolda2009TensorDA}]\label{def:mode-k}
    Let $~X\in\mathbb{R}^{I_1\times I_2\times...\times I_N}$ be a tensor.  The mode-$k$ unfolding $~X_{(k)}\in\mathbb{R}^{I_k\times I_1I_2...\widehat{I_k}...I_N}$ (hat indicates removal) is the matrix identifying element $(i_1,i_2,...,i_N)$ of $~X$ with element $(i_k,j)$ of $~X_{(k)}$, where
    \begin{align*}
        j = \sum_{\substack{n=1 \\ n\neq k}}^N\left[(i_n-1)\prod_{\substack{m=1 \\ m\neq k}}^{n-1}I_m\right].
    \end{align*}
\end{definition}
Observe that \Cref{def:mode-k} simply matricizes the given tensor so that $k$-th axis appears first and the remaining axes are unfolded ``column-wise'' (i.e., with Fortran ordering).  This leads to a simple SVD-based algorithm for computing the HOSVD.

\begin{algorithm}
\caption{Higher-Order SVD (HOSVD)}
\label{alg:hosvd}
\begin{algorithmic}[1]
\State \textbf{Input:} Snapshot tensor $~X \in \mathbb{R}^{N \times T \times P}$, target dimensions $(\bar{N}, \bar{T}, \bar{P})$.

\State Compute the top $\bar{N}$ left singular vectors of $~X_{(1)}$ to form $~W \in \mathbb{R}^{N \times \bar{N}}$.
\State Compute the top $\bar T$ left singular vectors of $~X_{(2)}$ to form $~T \in \mathbb{R}^{T\times \bar T}$.
\State Compute the top $\bar P$ left singular vectors of $~X_{(3)}$ to form $~S \in \mathbb{R}^{P \times \bar P}$.

\State Compute $[~C]_{i\alpha s} = \sum_{j,\beta,r} [~W]_{ji}[~X]_{j\beta r}[~T]_{\beta\alpha}[~S]_{rs}$ to form the core $~C \in \mathbb{R}^{\bar N \times \bar T \times \bar P}$

\State \textbf{Output:} Core $~C$ and factor matrices $~W, ~T, ~S$ defining $\tilde{~X}\approx~X$.
\end{algorithmic}
\end{algorithm}

\Cref{alg:hosvd} shows that the HOSVD of a 3-tensor $~X$ can be computed at the cost of a single SVD along each mode. Due to this structure, it is not difficult to show (see, e.g., \cite{delathauwer2000multilinear}) that the global error in the HOSVD is bounded by the sum of the local truncation errors incurred by each SVD, i.e.,
\begin{align}
    \big|~X-\tilde{~X}\big|_F^2 \leq \sum_{i=\bar{N}+1}^N\sigma_{1,i}^2 + \sum_{i=\bar{T}+1}^T\sigma_{2,i}^2 + \sum_{i=\bar{P}+1}^P\sigma_{3,i}^2 = \Delta^2, 
\end{align}
where $\sigma_{j,i}$ denotes the $i^{\rm th}$ singular value of the unfolding $~X_{(j)}$.  
This guarantees the alternative error bound in terms of the (small) truncation parameter $\varepsilon = \Delta / |~X|^2_F$,
\begin{equation}\label{eq:tucker_error}
    \big|~X-\tilde{~X}\big|_F \leq \varepsilon|~X|_F.
\end{equation}
This will be useful in the error estimation of Section~\ref{sec:analysis}.

\subsection{Reduced basis construction}

Remarkably, it is enough to manipulate the small core tensor $~{C}$ and factor matrices $~W,~T,~S$ in the HOSVD decomposition to build an effective ROM for \eqref{eq:FOM} whose reduced basis is tailored to any desired parameter value $~\mu$.
To see this, 
suppose there is a ``generalized index'' vector $~e=~e(~\mu)\in\mathbb{R}^{P}$ such that $\tilde{~X}_{\mu}\coloneqq \tilde{~X}~e(~\mu)\in\mathbb{R}^{N\times T}$ is a reasonable approximation to the inaccessible snapshot matrix at a test parameter $~\mu\in\mathbb{R}^p$.
It follows that the large matrix $\tilde{~X}_{\mu}$ can be expressed in terms of the previously computed HOSVD data and the vector $~e$
as
\begin{equation*}
    \tilde{~X}_{\mu} = \tilde{~X}~e = \sum_{i,\alpha,s} [~C]_{i\alpha s}(~e^\intercal~\nu^s)~{w}^i\otimes~{\tau}^\alpha = ~W(~C~S^\intercal~e)~T^\intercal \eqqcolon ~W~C_{\mu}~T^\intercal.
\end{equation*}
To see the advantage of this, notice that the thin SVD of the local snapshot approximation given by $\tilde{~X}_{\mu}~V_{\mu} = ~U_{\mu}~\Sigma_{\mu}$ is directly computable from the (small) local core matrix $~C_{\mu}\in\mathbb{R}^{\bar N\times \bar T}$ and these basis matrices: if $~C_{\mu}~V_c = ~U_c~\Sigma_c$ is the thin SVD of the core matrix, then
\begin{equation*}
    \tilde{~X}_{\mu}(~T~V_c) = ~W~C_{\mu}~V_c = (~W~U_c)~\Sigma_c.
\end{equation*}
This directly implies the thin SVD of $\tilde{~X}_{\mu}$, given in terms of the column-orthonormal matrices $~U_{\mu} = ~W~U_c$ and $~V_{\mu} = ~T~V_c$ along with the matrix $~\Sigma_{\mu} = ~\Sigma_c$ of singular values.  Therefore, building a local reduced basis $~U_{\mu}$ can be accomplished with the sequence of steps outlined in \Cref{alg:tucker_basis}.

\begin{algorithm}
\caption{HOSVD-Based Reduced Basis Generation}
\label{alg:tucker_basis}
\begin{algorithmic}[1]
\State \textbf{Input:} Snapshot tensor $~X \in \mathbb{R}^{N \times T \times P}$, Tucker ranks $(\bar{N}, \bar T, \bar P)$, generalized index $~e(~\mu)\in\mathbb{R}^{P}$, ROM dimension $n\ll N$.
\Statex \vspace{0.em} \hrulefill \ \textbf{Offline Stage} \ \hrulefill \vspace{0.2em}
\State Perform HOSVD \Cref{alg:hosvd} on $~X$, yielding $\tilde{~X}$ in terms of $~C\in\mathbb{R}^{\bar N\times \bar T\times \bar P}$ and column-orthonormal $~W\in \mathbb{R}^{N\times \bar N}$, $~T\in\mathbb{R}^{T\times\bar T}$, $~S\in\mathbb{R}^{P\times\bar P}$.
\Statex \vspace{0.em} \hrulefill \ \textbf{Online Stage} \ \hrulefill \vspace{0.2em}
\State Compute the thin SVD of $~C_{\mu} = ~C~S^\intercal~e(~\mu) \in \mathbb{R}^{\bar N\times\bar T}$ at test parameter $~\mu\in\mathbb{R}^p$, yielding $~U_c\in \mathbb{R}^{\bar{N}\times n}$, $~\Sigma_{c}\in\mathbb{R}^{n\times n}$, $~V_c\in\mathbb{R}^{\bar T\times n}$. 
\State \textbf{Output:} $~U_{\mu} = ~W~U_{c}$, $~\Sigma_{\mu} = ~\Sigma_c$, and $~V_{\mu} = ~T~V_c$ satisfying $\tilde{~X}_{\mu}~V_{\mu} = ~U_{\mu}~\Sigma_{\mu}$ (where $\tilde{~X}_{\mu} = ~W~C_{\mu}~T^\intercal$).
\end{algorithmic}
\end{algorithm}

Once the relatively expensive HOSVD factorization has been completed, the remaining steps of \Cref{alg:tucker_basis} are fast to execute for a given parameter instance $~\mu$.  In particular, lines 3 and 4 involve operations scaling with the Tucker ranks $\bar N,\bar T,\bar P$, which are much smaller than the original snapshot tensor dimensions.  Therefore, the Galerkin ROM analogous to \eqref{eq:ROM} becomes 
\begin{equation}\label{eq:MO-ROM}
    \dot{\hat{~x}}(t,~\mu) = ~U^\intercal_{\mu}~f\big( t, ~U_{\mu}\hat{~x}(t,~\mu),~\mu \big), \qquad \hat{~x}(0,~\mu) = ~U_{\mu}^\intercal~x_0(~\mu).
\end{equation}
Experiments in \cite{mamonov2022interpolatory,mamonov2024tensorial} demonstrated significant advantages of the ROM \eqref{eq:MO-ROM} over the corresponding monolithic basis formulation, particularly when PDE solutions vary nonlinearly with parameters. 
Additionally, note that \eqref{eq:MO-ROM} still retains the online speedup of the monolithic ROM \eqref{eq:ROM}: polynomial operators appearing in $~f$ can be pre-factored offline using the basis $~W$ coming from the HOSVD.  For example, consider the linear ROM $\dot{~x} = ~A~x$, so that the approximation $\tilde{~x}=~U_{\mu}\hat{~x}$ implies $\dot{\hat{~x}} = ~U_{\mu}^\intercal~A~U_{\mu}\hat{~x}$.  While directly computing the reduced linear operator requires an operation which scales with the FOM dimension $N$, it is also true that
\begin{equation*}
    \hat{~A}_{\mu} \coloneqq ~U_{\mu}^\intercal~A~U_{\mu} = ~U_c^\intercal(~W^\intercal~A~W)~U_c \eqqcolon ~U_c^\intercal\hat{~A}~U_c,
\end{equation*}
where $\hat{~A} = ~W^\intercal~A~W\in\mathbb{R}^{\bar N \times \bar N}$ can be pre-computed offline.  This makes the tensorial ROM highly efficient in practice, since the online adaptation scales only with the size of the Tucker ranks.

\subsection{Parameter Interpolation}
To complete the tensorial ROM pipeline described in \cite{mamonov2022interpolatory}, it remains to discuss how to construct the generalized index vector $~e(~\mu)\in\mathbb{R}^{P}$ used (formally) in generating the approximate snapshots $~X_{\mu}\approx ~X(~\mu)$.  There are a number of options for accomplishing this, based on considerations such as the dimensionality of the data and the structure of the parameter sampling.  We briefly review the two options suggested in \cite{mamonov2022interpolatory}.

\vspace{1pc}

\noindent \textbf{Barycentric Interpolation.}  A straightforward option, practical for low-dimensional parameter spaces, is given by the (piecewise-linear) barycentric interpolant of the training data.  Let $C \subset \mathbb R^p$ denote the convex hull of the given parameters $\{~\mu_s\}_{s=1}^{P}$ and consider the Delaunay triangulation of the set $\{~\mu_s\}$, each cell of which is a $p$-simplex with $p+1$ affinely independent vertices.  The barycentric coordinates
$(\lambda_0, \ldots, \lambda_{p})$ of a query parameter $~\mu$ are the solution to the affine system
$$
\begin{bmatrix}
~\mu_{s_0} & ~\mu_{s_1}  & \cdots & ~\mu_{s_p}\\
1 & 1 & \cdots & 1
\end{bmatrix}
\begin{bmatrix}
\lambda_0 \\ \lambda_1 \\ \vdots \\\lambda_p
\end{bmatrix} = 
\begin{bmatrix}
    ~\mu \\ 1
\end{bmatrix},
$$
where $\{~\mu_{s_0}, \ldots, ~\mu_{s_p}\}$ are the vertices of the (unique) simplex containing $~\mu$. In this case, a generalized index vector $~e(~\mu)$ can be defined as 
\begin{equation}\label{eq:gen-index-barycentric}
    ~e(~\mu) = \sum_{j=0}^p \lambda_j~e_{s_j}, 
\end{equation}
where $\{~e_s\}_{s=1}^{P}$ is the standard basis for $\mathbb{R}^{P}$.  Observe that all entries of the interpolation vector $~e(~\mu)$ are nonnegative and sum to one, leading to a strict interpolation of the snapshots in the tensor $~X$: if $~\mu=~\mu_s$ for some $1\leq i\leq N_s$, then $~e(~\mu) = ~e_s$ and $~X_{\mu} = ~X_{\mu_s}$.  However, this barycentric procedure quickly becomes expensive as the size of $p$ grows, since the cost of constructing the Delaunay triangulation scales with $\mathcal O\big(N_s^{\lceil p/2 \rceil}\big).$

\vspace{1pc}

\noindent \textbf{Mamonov/Olshanskii Generic Interpolation}
Another option for computing the generalized index $~e(~\mu)$ uses a parametric interpolant constructed via a weighted local least-squares procedure.  Given the matrix  $~Z\in\mathbb{R}^{p\times P}$ of training parameter samples from before, arranged so that $[~Z]_{is} = [~\mu_s]_i$,
along with a query point $~\mu\in\mathbb{R}^p$, consider the $K>p$ nearest neighbors $\{~\mu_{s_1},...,~\mu_{s_K}\}$ of $~\mu$.   Define the inverse distance matrix $~D = \mathrm{Diag}(~d)^{-1}\in\mathbb{R}^{K\times K}$ in terms of element-wise distances $[~d]_j = |~\mu - ~\mu_{s_j}|$, along with the local augmented parameter matrix $\bar{~Z}\in\mathbb{R}^{(p+1)\times K}$ whose $j^{th}$ column is $(~\mu_{s_j}^\intercal, 1)^\intercal$.  Then, solving the weighted minimum-norm problem
\begin{equation*}
    \argmin_{\bar{~a}=~D~a}\big|\bar{~Z}\bar{~a}-~c\big|^2,
\end{equation*}
for $~c = (~\mu^\intercal, 1)^\intercal$ yields a vector $\bar{~a} = ~D(\bar{~Z}~D)^\dagger~c\in\mathbb{R}^K$ of interpolation weights.  Supposing that the columns of $\bar{~Z}$ are ordered from $s_1,s_2,...,s_K$, this leads to a generalized index vector $~e(~\mu)$ defined by 
\begin{equation}\label{eq:gen-index-MO}
    ~e(~\mu) = \sum_{j=1}^K [\bar{~a}]_j~e_{s_j}.    
\end{equation}
Note that the last row of $\bar{~Z}$ and last entry of $~c$ ensure that the entries of $\bar{~a}$ sum to one, while the weighting enforced by $~D$ places more emphasis on parameters close to $~\mu$.  Moreover, no structure on the parameter domain $C\subset\mathbb{R}^P$ is necessary for computing \eqref{eq:gen-index-MO}, making this approach scalable and reasonably performant in high dimensions.

\section{Proposed Extensions}\label{sec:ours}
The tensorial ROM discussed in \Cref{sec:method} is simple to use and effective in a wide variety of parametric PDE scenarios.  However, it is limited in a couple of key ways.  For one, present technology requires that the reduced basis $~U_{\mu}\in\mathbb{R}^{N\times n}$ be Euclidean-orthonormal, leading to incompatibility with (or awkward use of) certain classes of structure-preserving ROMs.  To see this, consider for simplicity the standard heat equation $\dot{x} = \Delta x$ in terms of a function $x\in C^2(\mathcal{M})$ on some closed domain $\mathcal{M}$.  Important properties of this PDE, such as the monotonic energy decay (i.e., energy-stability) necessary for the stability of its discretizations, are guaranteed by its interpretation as an $L^2$-gradient flow, i.e., 
\begin{equation*}
    \dot{x} = -\mathrm{grad}_{L^2}\mathcal{D}(x), \qquad \mathcal{D}(x) = \int_\mathcal{M} |\nabla x|^2\,dV,
\end{equation*}
where $\mathcal{D}$ is known as the Dirichlet energy functional.  Expressing $x = \sum_i x_i\varphi_i$ in terms of a finite basis $\{\varphi_i\}_{i=1}^N \subset H^1$ of functions and testing against an arbitrary $y\in H^1$ yields the standard finite element discretization 
\begin{equation}\label{eq:heatFOM}
    (\dot{x},y)_{L^2} = -(\nabla x,\nabla y)_{L^2} \quad \Longleftrightarrow \quad ~M\dot{~x} = -~K~x, 
\end{equation}
where $~x\in\mathbb{R}^N$ is a vector of coefficients, $[~M]_{ij} = (\varphi_i,\varphi_j)_{L^2}$ is the mass matrix of the scheme, and $[~K]_{ij} = (\nabla\varphi_i,\nabla\varphi_j)_{L^2}$ is its stiffness matrix.  Importantly, the FOM \eqref{eq:heatFOM} retains the gradient flow structure of the continuous PDE, since 
\begin{equation}\label{eq:heatGF}
    \dot{~x} = -~M^{-1}~K~x = -~M^{-1}~D^\intercal~M~D~x = -~D^*~D~x = -\mathrm{grad}_M \big(|~D~x|_M^2\big),
\end{equation}
in terms of the $~M$-weighted norm $|~x|^2_M = ~x^\intercal~M~x$, the combinatorial gradient $~D$ (c.f. \cite{knill2013dirac}), and its $~M$-adjoint $~D^*=~M^{-1}~D^\intercal~M$.  On the other hand, making the ROM approximation $\tilde{~x}=~U\hat{~x} \approx ~x$ in terms of a Euclidean-orthonormal reduced basis $~U\in\mathbb{R}^{N\times n}$ and applying Galerkin projection leads to two possible and inequivalent reduced-order equations.  Depending on the FOM expression \eqref{eq:heatGF} or \eqref{eq:heatROM}, the ROM for the reduced state $\hat{~x}\in\mathbb{R}^{n}$ becomes either
\begin{equation}
    \dot{\hat{~x}} = -\widehat{~M^{-1}~K}\hat{~x}, \qquad {\rm or} \qquad \hat{~M}\dot{\hat{~x}} = -\hat{~K}\hat{~x},
\end{equation}
where $\widehat{~M^{-1}~K} = ~U^\intercal~M^{-1}~K~U$, $\hat{~M}=~U^\intercal~M~U$, and $\hat{~K}=~U^\intercal~K~U$ are reduced-order operators.  It turns out that the second equation admits the interpretation of a discrete gradient flow while the first does not: the latter equations are just $\dot{\hat{~x}} = -\mathrm{grad}_{\hat{M}}(|~D~U\hat{~x}|_M^2)$ in the reduced $\hat{~M}$-inner product, while the former cannot be expressed as $\dot{\hat{~x}} = -\mathrm{grad}_{\bar{M}}\hat{\mathcal{D}}(\hat{~x})$ for any functional $\hat{\mathcal{D}}$ of the reduced state and any symmetric and positive definite $\bar{~M}$.  Consequently, only one such ROM is energy-stable for every choice of reduced basis, and the user must know the distinction: directly projecting a structured FOM is not enough to produce a structure-preserving ROM.  
The first extension presented here addresses exactly this concern. 


\subsection{Weighted Orthonormal Tensorial Basis}

In the usual setting of Galerkin ROMs, there are established techniques for computing weighted orthonormal reduced bases $~U\in\mathbb{R}^{N\times n}$ \cite{daescu2008dual,reyes2020projection}, which simplify structural considerations and eliminate conditioning issues with reduced mass matrices. In particular, it is common for the reduced mass matrix $\hat{~M}$ to be poorly conditioned especially at large basis sizes $n$.  This can lead to significantly degraded performance, making it desirable to have a strategy for ensuring that the reduced basis $~U$ is mass-weighted, i.e., $~U^\intercal~M~U=~I$.  This becomes even more true for models with more complicated geometric structure, e.g. Hamiltonian models \cite{vijaywargiya2025tensor}, where $~M$-orthonormality can avoid a proliferation of $\hat{~M}$ factors and their difficult-to-compute inverses.  

The first contribution of this work is a similar procedure for constructing weighted orthonormal reduced bases $~U_{\mu}\in\mathbb{R}^{N\times n}$, nonlinearly calibrated to a particular parameter instance  $~\mu\in\mathbb{R}^p$, within the tensorial ROM framework.  First, we present a result showing the existence of an $~M$-weighted Tucker decomposition for any symmetric and positive definite $~M\in\mathbb{R}^{N\times N}$.  Note the following Lemma.

\begin{lemma}\label{lem:M-opt-vs-E-opt}
    Let $~X\in\mathbb{R}^{N\times T\times P}$ be a snapshot tensor, $~r = (\bar N, \bar T, \bar P)$ be a collection of Tucker ranks, $~M\in\mathbb{R}^{N\times N}$ be symmetric and positive definite, and $|~X|^2_M = \sum_{i,j,\alpha,s}[~M]_{ij}[~X]_{i\alpha s}[~X]_{j\alpha s}$ denote the $~M$-weighted Frobenius norm.  Consider the best rank-$~r$ reconstructions of $~X$ in the Frobenius norm and its $~M$-weighted counterpart:
    \begin{align*}
        ~X^F_{\rm opt} = \argmin_{{\rm rank}(\bar{~X})=~r} \big|~X-\bar{~X}\big|_F, \quad ~X^M_{\rm opt} = \argmin_{{\rm rank}(\bar{~X})=~r} \big|~X-\bar{~X}\big|_M.
    \end{align*}
    If $~M=~R^\intercal~R$ is a Cholesky factorization, then $(~R~X)^F_{\rm opt} = ~R~X^M_{\rm opt}$.

\end{lemma}
\begin{proof}
    By the minimality of $~X^M_{\rm opt}$ and for any rank-$~r$ tensor $\bar{~Y}=~R\bar{~X}$, it follows that 
    \begin{align*}
        \big|~X-~X^M_{\rm opt}\big|_M = \big|~R~X-~R~X^M_{\rm opt}\big|_F \leq \big|~R~X-~R\bar{~X}\big|_F = \big|~R~X-\bar{~Y}\big|_F.
    \end{align*}
    Since $~R$ has full rank, minimizing over $\bar{~Y}$ yields $|~R~X-~R~X^M_{\rm opt}|_F \leq |~R~X-(~R~X)^F_{\rm opt}|_F$.  On the other hand, minimality of $(~R~X)^F_{\rm opt}$ guarantees the similar relationship
    \begin{align*}
        \big|~R~X-(~R~X)^F_{\rm opt}\big|_F \leq \big|~R~X-\bar{~Y}\big|_F = \big|~R~X-~R\bar{~X}\big|_F = \big|~X-\bar{~X}\big|_M,
    \end{align*}
    and minimizing over $\bar{~X}$ establishes the reverse inequality.  Therefore, $|~R~X-~R~X^M_{\rm opt}|_F = |~R~X-(~R~X)^F_{\rm opt}|_F$ and $(~R~X)^F_{\rm opt}=~R~X^M_{\rm opt}$ as desired.
\end{proof}

With this result in hand, the existence of an $~M$-weighted HOSVD decomposition can now be proven.



\begin{theorem}\label{thm:M-tucker}
    Let $~X\in\mathbb{R}^{N\times T\times P}$ be a tensor and let $~M\in\mathbb{R}^{N\times N}$ be a symmetric and positive definite matrix.
    Given fixed Tucker ranks $\bar{N}<N$, $\bar{T}< T$, and $\bar P < P$, there exists a HOSVD reconstruction  
    \begin{align*}
        \tilde{~{X}} = \sum_{i=1}^{\bar{N}}\sum_{\alpha=1}^{\bar{T}}\sum_{s=1}^{\bar{P}} [~{C}]_{i\alpha s}~{w}^i\otimes~{\tau}^\alpha\otimes~{\nu}^s \approx ~X,
    \end{align*}
    which satisfies the quasi-optimal error bound 
    \begin{align*}
        \big|~X-\tilde{~X}\big|_M \leq \sqrt{3}\big|~X-~X_{\rm opt}\big|_M,
    \end{align*}
    where $~X_{\rm opt}$ is the best Tucker rank-$(\bar{N},\bar{T},\bar{P})$ approximation to $~X$ in the $~M$-weighted Frobenius norm $|~X|^2_M = \sum_{i,j,\alpha,s}[~M]_{ij}[~X]_{i\alpha s}[~X]_{j\alpha s}$.
\end{theorem}
\begin{proof}
    Let $~M=~R^\intercal~R$ be the Cholesky factorization of $~M$.  Consider the Tucker decomposition of the weighted snapshots $~R~X$,
    \begin{align*}
        \tilde{~{X}}_R = \sum_{i=1}^{\bar{N}}\sum_{\alpha=1}^{\bar{T}}\sum_{s=1}^{\bar{P}} [~{C}]_{i\alpha s}\tilde{~{w}}^i\otimes~{\tau}^\alpha\otimes~{\nu}^s \approx ~R~X,
    \end{align*}
    defined in terms of a Euclidean-orthonormal basis $\tilde{~W}\in\mathbb{R}^{N\times\bar N}$ (containing $\{\tilde{~w}_i\}$) for the first dimension.  Solving the system $~R~W=\tilde{~W}$ for $~W$ yields an $~M$-orthonormal change-of-basis along with the decomposition 
    \begin{align*}
        \tilde{~{X}} = \sum_{i=1}^{\bar{N}}\sum_{\alpha=1}^{\bar{T}}\sum_{s=1}^{\bar{P}} [~{C}]_{i\alpha s}~{w}^i\otimes~{\tau}^\alpha\otimes~{\nu}^s = ~R^{-\intercal}\tilde{~X}_{R}.
    \end{align*}
    We claim that $\tilde{~X}$ is the desired decomposition.  To see this, recall that $~R~X_{\rm opt} = (~R~X)_{opt}$ by \Cref{lem:M-opt-vs-E-opt}.
    Therefore, it follows that
    \begin{align*}
        \big|~X-\tilde{~X}\big|_M &= \big|~R~X-\tilde{~X}_R\big|_F \leq \sqrt{3}\big|~R~X-(~R~X)_{opt}\big|_F \\ 
        &= \sqrt{3}\big| ~R~X-~R~X_{\rm opt}\big|_F = \sqrt{3}\big|~X-~X_{\rm opt}\big|_M,
    \end{align*}
    establishing the desired quasi-optimal error bound.
\end{proof}

\begin{remark}
    Repeating the argument leading to \eqref{eq:tucker_error} guarantees the reformulated error bound 
    \begin{equation}\label{eq:M-tucker_error}
        \big|~X-\tilde{~X}\big|_M \leq \varepsilon|~X|_M,
    \end{equation}
    where $\varepsilon = \Delta_M / |~X|_M$ and $\Delta_M$ denotes the sum of modal truncations of $~R~X$.
\end{remark}


\Cref{thm:M-tucker} shows that an $~M$-orthonormal HOSVD can be obtained by first applying \Cref{alg:hosvd} to the weighted snapshots $~R~X$, and then re-weighting the resulting basis $\tilde{~W}\in\mathbb{R}^{N\times\bar N}$ by $~R^{-1}$.  It is straightforward to check that the resulting sequence of steps in \Cref{alg:tucker_basis} preserves $~M$-orthonormality, yielding a reduced basis $~U_{\mu}\in\mathbb{R}^{N\times n}$ satisfying $~U_{\mu}^\intercal~M~U_{\mu}=~I$.  In particular, observe that $~U_{\mu} = ~W~U_c$ remains $~M$-orthonormal when $~U_c\in\mathbb{R}^{\bar{N}\times n}$ is Euclidean-orthonormal, since $~U_{\mu}^\intercal~M~U_{\mu} = ~U_c^\intercal~W^\intercal~M~W~U_c = ~U_c^\intercal~U_c = ~I$.  For convenience, the steps of this construction are summarized in \Cref{alg:M_tucker_basis}.

\begin{algorithm}
\caption{Weighted HOSVD-Based Reduced Basis Generation}
\label{alg:M_tucker_basis}
\begin{algorithmic}[1]
\State \textbf{Input:} Snapshot tensor $~X \in \mathbb{R}^{N \times T \times P}$, SPD matrix $~M\in\mathbb{R}^{N\times N}$, Tucker ranks $(\bar{N}, \bar T, \bar P)$, generalized index $~e(~\mu)\in\mathbb{R}^{P}$, ROM dimension $n\ll N$.
\Statex \vspace{0.em} \hrulefill \ \textbf{Offline Stage} \ \hrulefill \vspace{0.2em}
\State Compute the Cholesky decomposition $~M=~R^\intercal~R$.
\State Perform HOSVD \Cref{alg:hosvd} on $~R~X$, yielding $\tilde{~X}_R$ in terms of $~C\in\mathbb{R}^{\bar N\times \bar T\times \bar P}$ and column-orthonormal $\tilde{~W}\in \mathbb{R}^{N\times \bar N}$, $~T\in\mathbb{R}^{T\times\bar T}$, $~S\in\mathbb{R}^{P\times\bar P}$.
\State Solve $~R~W = \tilde{~W}$ for the $~M$-orthonormal basis $~W\in\mathbb{R}^{N\times \bar{N}}$.
\State Compute $\tilde{~X} = \sum_{i,\alpha,s}[~C]_{i\alpha s}~w^i\otimes~\tau^\alpha\otimes~\nu^s$ in terms of $~C,~W,~T,~S$.
\Statex \vspace{0.2em} \hrulefill \ \textbf{Online Stage} \ \hrulefill \vspace{0.2em}
\State Compute the thin SVD of $~C_{\mu} = ~C~S^\intercal~e(~\mu) \in \mathbb{R}^{\bar N\times\bar T}$ at test parameter $~\mu\in\mathbb{R}^p$, yielding $~U_c\in \mathbb{R}^{\bar{N}\times n}$, $~\Sigma_{c}\in\mathbb{R}^{n\times n}$, $~V_c\in\mathbb{R}^{\bar T\times n}$.
\State \textbf{Output:} $~U_{\mu} = ~W~U_{c}$, $~\Sigma_{\mu} = ~\Sigma_c$, and $~V_{\mu} = ~T~V_c$ satisfying $\tilde{~X}_{\mu}~V_{\mu} = ~U_{\mu}~\Sigma_{\mu}$ (where $\tilde{~X}_{\mu} = ~W~C_{\mu}~T^\intercal$).
\end{algorithmic}
\end{algorithm}


\subsection{Interpolation with Gaussian Radial Basis Functions}

\Cref{alg:M_tucker_basis} enables the use of non-Euclidean orthonormalized reduced bases in the tensorized ROM, greatly simplifying application of the structure-preserving techniques used in the experiments (c.f. \Cref{sec:numerics}).  However, this is not enough to achieve acceptable performance in all cases of practical interest, particularly when parameter sampling is high-dimensional, unstructured, and sparse.  To further improve the tensorial ROM in these instances, we now present an alternative to the interpolation techniques discussed in \Cref{sec:method}.  Employing local (or global) Gaussian radial basis functions (GRBFs), the proposed approach produces smooth local (or global) interpolants which better capture curvature information across the sparsely sampled parameter space.  Moreover, these interpolants can be stably and efficiently constructed even when the number of neighbors $K \ll p$ is much smaller than the dimension of the parameter space.  

Note that the case of sparse sampling across the parameter space is not artificial.  The primary goal of a ROM is often to enable many-query analysis of expensive simulations, where snapshot collection is inherently limited.  This can substantially weaken the distance-weighted least-squares strategy of \cite{mamonov2022interpolatory}, which may pull the interpolation artificially toward the mean when all training samples are far away from the query parameter.  In addition, the weighted least-squares strategy cannot effectively detect local extrema or other curvature information across the parameter space, leading to decreased accuracy when this information is important.  Conversely, a GRBF-based interpolation produces a non-polynomial profile that varies smoothly throughout parameter space, automatically incorporating parametric nonlinearities with high regularity.

To describe this procedure in detail, consider the collection 
$~Q \in \mathbb{R}^{p\times P}$ of training samples and a query point $~\mu \in\mathbb{R}^p$ as before. A local GRBF-based weight vector $~e(~\mu)$ can be constructed by first selecting $K$ nearest neighbors $\{~\mu_{s_1},\dots,~\mu_{s_K}\}$ of $~\mu$ and a radial Gaussian kernel $\phi_{\varepsilon}(r) = \exp[-(\varepsilon r)^2]$ with shape parameter $\varepsilon>0$. Given these data, a local kernel matrix $~K_{\varepsilon}\in\mathbb{R}^{K\times K}$ and evaluation vector $~k_{\varepsilon}(~\mu)\in\mathbb{R}^K$ are defined by 
\begin{align*}
    [~K_{\varepsilon}]_{jk} = \phi_{\varepsilon}\big(|~\mu_{s_j}-~\mu_{s_k}|\big), \qquad [~k_{\varepsilon}(~\mu)]_j = \phi_{\varepsilon}\big(|~\mu-~\mu_{s_j}|\big).
\end{align*}
The local GRBF weight vector $~e_{\text{loc}}(~\mu)\in\mathbb{R}^K$ associated with the interpolant is then obtained as $~e_{\text{loc}}(~\mu) = ~K_{\varepsilon}^{-1}~k_{\varepsilon}(~\mu)$. Finally, this local weight vector can be embedded into a global weight vector $~e(~\mu)\in\mathbb{R}^P$ by defining
\begin{equation}\label{eq:gen-index-RBF}
[~e(~\mu)]_j = \begin{cases}
    [~e_{\text{loc}}(~\mu)]_k & j = s_k \in \{s_1,\dots,s_K\}, \\
    0 & {\rm otherwise}.
\end{cases}
\end{equation}
It follows that any scalar quantity sampled at the training parameters is interpolated at $~\mu$ as a weighted sum using $~e(~\mu)$, similar to the previous constructions \eqref{eq:gen-index-barycentric} and \eqref{eq:gen-index-MO}.  On the other hand, \eqref{eq:gen-index-RBF} has the advantage of (locally) increased regularity, offering potential gains in representational power.  It will be shown in \Cref{sec:numerics} that this leads to substantially better performance in cases of practical interest. 


\subsection{Algorithmic Complexity}

\rev[anthony]{}{
To understand the effect of the proposed extensions on computational cost, it is useful to examine the algorithmic complexity of the proposed tensorial ROM compared to the MO and monolithic baselines. To that end, \Cref{tab:time-complexity} presents the asymptotic time-complexities of the involved operations, while \Cref{tab:memory-complexity} presents the corresponding asymptotic memory-complexities.
}

\rev[anthony]{}{
First, observe that the monolithic approach is always less complex for a fixed $n$: the offline cost is a single SVD of complexity $\mathcal{O}(N^2TP)$ (assuming $N<TP)$ plus operator pre-computation, and the only online operations are load vector projection and time integration.  Conversely, the tensorial approaches require a HOSVD of complexity roughly $\mathcal{O}(N^2TP + T^2PN + P^2NT)$ along with $\mathcal{O}(NTP\bar{N} + \bar{N}TP\bar{T}+\bar{N}\bar{T}P\bar{P})$ core formation offline, followed by interpolation, local SVD computation, and basis assembly online, in addition to the usual load vector projection and time integration.  With that said, all additional online operations in the tensorial approach (besides basis assembly) scale with the Tucker ranks $\bar{N},\bar{T},\bar{P}$ and not the full state dimension $N$, making them relatively efficient in practice.  Moreover, it will be seen in \Cref{sec:numerics} that the tensorial approach enables the use of much smaller basis sizes $n$ for a given error target, substantially mitigating this increase in online cost.  
}

\rev[anthony]{}{
On the interpolation side, it is useful to compare the proposed GRBF strategy to the procedure described in \cite{mamonov2022interpolatory}. Observe that the offline cost of building and factorizing the global RBF kernel matrix $~K_\varepsilon$ is $\mathcal{O}(pP^2 + P^3)$, linear in the parameter dimension $p$ and cubic in the training-set size $P$, while the online cost of evaluating the RBF weight vector is $\mathcal{O}(pP + P^2)$.  For large $P$, a local interpolant may be preferred: with a brute-force neighbor search it incurs no offline cost, and the online cost of building the weight vector $~e(~\mu)$ is $\mathcal{O}(K^3 + pK^2 + pP)$, where $K$ is the (typically small) number of neighbors.  Conversely, the online cost of the MO interpolant depends \emph{quadratically} on $p$, as $\mathcal{O}(p^2 K + pP)$.  Since the global or local GRBF interpolant scales only linearly in $p$, a substantial cost reduction is expected in scenarios with high-dimensional parametric dependence. Finally, note that both interpolants can benefit from a prebuilt neighbor-search tree that reduces the $\mathcal{O}(pP)$ search term to $\mathcal{O}(p\log P)$ in regimes where such structures are effective, at an additional offline cost of $\mathcal{O}(pP\log P)$.
}

\begingroup
\renewcommand{\thefootnote}{\fnsymbol{footnote}}
\setcounter{footnote}{0}
\begin{table}[htbp!]
\centering
\renewcommand{\arraystretch}{1.3}
\setlength{\tabcolsep}{4pt} 

\resizebox{\textwidth}{!}{%
\begin{tabular}{@{}llccc@{}}
\toprule
 Stage & Step & Monolithic & Tensorial-MO & Tensorial-RBF \\
\midrule
\midrule
\textit{Offline} & Cholesky $~M=~R^\top ~R$\tablefootnote{Dense $~M$; for sparse $~M$ with a fill-reducing ordering Cholesky scales as $\mathcal O(N^{3/2})$ (2D) / $\mathcal O(N^{2})$ (3D) and $~R~X$, $~R~W=\tilde{~W}$ scale as $\mathcal O(\bar N\mathrm{nnz}(~R))$.}
  & --- & $\mathcal O(N^{3})$ & $\mathcal O(N^{3})$ \\
 & Weighting $~R~X$
  & --- & $\mathcal O(N^{2}TP)$ & $\mathcal O(N^{2}TP)$ \\
 & SVD of $~X_{(1)}$
  & $\mathcal O\big(N(TP) \min(N,TP)\big)$ & $\mathcal O\big(N(TP)\min(N,TP)\big)$ & $\mathcal O\big(N(TP) \min(N,TP)\big)$ \\
  \\
  & SVD of $~X_{(2)}$
  & --- & $\mathcal O\big(T(NP)\min(T,NP)\big)$ & $\mathcal O\big(T(NP) \min(T,NP)\big)$ \\
  \\
 & SVD of $~X_{(3)}$
  & --- & $\mathcal O\big(P(NT) \min(P,NT)\big)$ & $\mathcal O\big(P(NT)\min(P,NT)\big)$ \\
  \\
 & Core $~C$ formation
  & --- & \makecell[l]{$\mathcal O(NTP\bar N + \bar N TP \bar T$ \\ $  + \bar N \bar T P \bar P)$} & \makecell[l]{$\mathcal O(NTP\bar N + \bar N TP \bar T$ \\ $+ \bar N \bar T P \bar P)$} \\ \\
 & \makecell[l]{Solve $~R~W=\tilde {~W}$}
  & --- & $\mathcal O(N^{2}\bar N)$ & $\mathcal O(N^{2}\bar N)$ \\
 &
 \makecell[l]{Building and factorizing\\  kernel $~K_\varepsilon$}
  & --- & --- & $\mathcal O(pP^{2}+P^{3})$ \\
 & \makecell[l]{Reduced operator \\ $~W^\top\!~A~W$ / $~U^\top\!~A~U$}\tablefootnote{Dense $~A$; sparse operators with $\mathcal O(N)$ nonzeros give $\mathcal O(N\bar N^{2})$ / $\mathcal O(Nn^{2})$, with dense-inverse operators applied in factored form.}
  & $\mathcal O(N^{2}n)$ & $\mathcal O(N^{2}\bar N)$ & $\mathcal O(N^{2}\bar N)$ \\
\midrule
\midrule
\textit{Online} & \makecell[l]{Generalized index \\ vector $~e(~\mu)$}\tablefootnote{Brute-force neighbor search; a prebuilt $k$-$d$ tree gives $\mathcal O(p\log P)$ per query at $\mathcal O(pP\log P)$ offline.}
  & --- & $\mathcal O(p^{2}K+pP)$ & $\mathcal O(pP+P^{2})$ \\
 & Form $~C_\mu=~C~S^\top ~e$
  & ---
  & \makecell[l]{$\mathcal O(\bar N\bar T\bar P$ $+\,P\bar P)$}
  & \makecell[l]{$\mathcal O(\bar N\bar T\bar P$ $+\,P\bar P)$} \\
 & Local SVD of $~C_\mu$
  & --- & $\mathcal O\big(\bar N\bar T\min(\bar N,\bar T)\big)$ & $\mathcal O\big(\bar N\bar T\min(\bar N,\bar T)\big)$ \\
 & \makecell[l]{Basis assembly \\ $~U_\mu=~W~U_c$}
  & --- & $\mathcal O(N\bar N n)$ & $\mathcal O(N\bar N n)$ \\
 & \makecell[l]{Operator projection \\ $~U_c^\top\hat {~A} ~U_c$}
  & ---
  & \makecell[l]{$\mathcal O(\bar N^{2}n$ $+\,\bar N n^{2})$}
  & \makecell[l]{$\mathcal O(\bar N^{2}n$ $+\,\bar N n^{2})$} \\
 & \makecell[l]{Spatial load projection \\ $\hat{~f}_{sp}=~U^\top ~f_{sp}(~\mu)$}\tablefootnote{Includes $\mathcal O(N)$ assembly of $~f_{sp}(~\mu)$ before projection.}
  & $\mathcal O(Nn)$ & $\mathcal O(Nn)$ & $\mathcal O(Nn)$ \\
 & ROM integration\tablefootnote{Linear ROM, separable forcing; the $n^{3}$ term is the one-time implicit iteration-matrix factorization, absent for explicit schemes ($\mathcal O(Tn^{2})$).}
  & $\mathcal O(Tn^{2}+n^{3})$ & $\mathcal O(Tn^{2}+n^{3})$ & $\mathcal O(Tn^{2}+n^{3})$ \\
\bottomrule
\end{tabular}%
}
\caption{Time-complexity of the different steps in the \textit{offline} and \textit{online} stages associated with the monolithic and tensorial ROMs. The tensorial ROMs introduce additional online overhead; apart from the basis assembly $~U_\mu=~W~U_c$, each added step scales with the Tucker ranks $\bar N\le N$, $\bar T\le T$, $\bar P\le P$ rather than the full dimension $N$.}
\label{tab:time-complexity}
\end{table}
\endgroup

\begin{table}[htbp!]
\centering
\renewcommand{\arraystretch}{1.3}
\setlength{\tabcolsep}{6pt}
\begin{tabular}{@{}lccc@{}}
\toprule
Object & Monolithic & Tensorial-MO & Tensorial-RBF \\
\midrule
\midrule
\makecell[l]{Tucker factors $~W,~T,~S$ \\ and core $~C$}
  & --- 
  & $\mathcal O(N\bar N{+}T\bar T{+}P\bar P{+}\bar N\bar T\bar P)$
  & $\mathcal O(N\bar N{+}T\bar T{+}P\bar P{+}\bar N\bar T\bar P)$ \\
Basis $~U$
  & $\mathcal O(Nn)$ & --- & --- \\
\makecell[l]{Reduced operator \\ $~U^\top\!~A~U$ / $~W^\top\!~A~W$}
  & $\mathcal O(n^{2})$ & $\mathcal O(\bar N^{2})$ & $\mathcal O(\bar N^{2})$ \\
Parameter set $~Z\in\mathbb{R}^{p\times P}$
  & --- & $\mathcal O(pP)$ & $\mathcal O(pP)$ \\
Kernel matrix $~K_\varepsilon\in\mathbb{R}^{P\times P}$
  & --- & --- & $\mathcal O(P^{2})$ \\
\bottomrule
\end{tabular}
\caption{Memory-complexity required at query time by the monolithic and
tensorial ROMs.}
\label{tab:memory-complexity}
\end{table}

\section{Error estimation}\label{sec:analysis}

Before exercising the improvements outlined in \Cref{sec:ours}, it is important to verify that these changes to the tensorial ROM do not negatively impact its theoretical assurances.  To that end, we now present a representation error estimate, along with a bound on the error in the Galerkin ROM \eqref{eq:MO-ROM},
where $~U_{\mu}$ is the output of \Cref{alg:M_tucker_basis} with generalized index \eqref{eq:gen-index-RBF}.  First, recall the snapshot tensor $~X\in\mathbb{R}^{N\times T\times P}$ and its mode-one matricization $~X_{(1)}\in\mathbb{R}^{N\times TP}$.  Given a Cholesky factorization $~M = ~R^\intercal~R$ for some mass matrix $~M$, the thin SVD $~R~X_{(1)}\bar{~V} =\bar{~U}~\Sigma$ defines an $~M$-orthonormal monolithic basis $~U$ through $~R~U=\bar{~U}$.  It follows that the in-sample projection error satisfies 
\begin{equation}\label{eq:proj-error-mono}
    \big|~P^\perp_{U}~X\big|^2_{M} = \big|~P^\perp_{U}~X_{(1)}\big|^2_{M} = \sum_{i=n+1}^{\bar{R}} \sigma_i^2,
\end{equation}
where $\bar R$ is the rank of the matrix $~X_{(1)}$, $\{\sigma_i\}$ are the entries of ${~\Sigma}$, and $~P_{U}^\perp = ~I-~U~U^*$ denotes $~M$-orthogonal projection onto the complement of ${\rm im}\,~U$. Here, $~U^* = ~U^\intercal~M$ denotes the $~M$-adjoint of $~U$, defined by the equality ($\hat{~x}\in\mathbb{R}^n$ and $~y\in\mathbb{R}^N$)
\[\langle~U\hat{~x},~y\rangle_{M} = \langle\hat{~x},~U^*~y\rangle_M.\]
This shows that the approximation error of the basis $~U$ is controllable in $~M$-norm: by choosing an appropriately large $n$, the projection error \eqref{eq:proj-error-mono} can be made arbitrarily small.   


The first goal is to provide a similar guarantee for the local projection error in the basis $~U_{\mu}$ used in the tensorial ROM \eqref{eq:MO-ROM}.  While this is straightforward when $~\mu$ is contained in the training set, the more general (and useful) case requires the following mild assumptions on the snapshot interpolation scheme, also employed in \cite{mamonov2022interpolatory}.  Let $\Omega\subset\mathbb{R}^P$ denote a compact parameter domain.  
\begin{assumption}\label{assump:stability}
    The interpolation process $~\mu\mapsto~e(~\mu)$ is stable for all $~\mu\in\Omega$, i.e., there exists a constant $C_e>0$ independent of $~\mu$ such that 
    \[\sup_{~\mu\in\Omega} \big|~e(~\mu)\big|\leq C_e. \]
\end{assumption}
\begin{assumption}\label{assump:bounded-error}
    The error in the interpolation process is bounded, i.e., there exists a constant $\delta>0$ (depending on $\Omega$) such that 
    \[\sup_{~\mu\in\Omega} \big|g(~\mu) - ~e(~\mu)^\intercal~g(~S)\big| \leq |g|_{C^p(\Omega)}\delta,\]
    for any function $g\in C^p(\Omega)$. 
\end{assumption}


It will also be useful to have a ``dictionary'' converting between the geometry induced by $~M$ and the usual Euclidean geometry, established by the following result.

\begin{proposition}\label{prop:M-isom}
    Consider the Cholesky factorization $~M=~R^\intercal~R$.  The map $~A\mapsto ~R~A~R^{-1}$ provides an isometric isomorphism between the general linear group $GL(N, \mathbb{R})$ equipped with the $~M$-spectral norm, 
    \begin{equation*}
        |~A|_{M,2} = \sup_{|~x|_M\neq 0} \frac{|~A~x|_M}{|~x|_M},
    \end{equation*}
    and the same group equipped with the (Euclidean) spectral norm $|~A|_2$.
\end{proposition}
\begin{proof}
    The homomorphism property follows immediately: for $~A,~B\in GL(N,\mathbb{R})$,
    \[(~R~A~R^{-1})(~R~B~R^{-1}) = ~R~A(~R^{-1}~R)~B~R^{-1} = ~R(~A~B)~R^{-1}.\]
    Similarly, injectivity and surjectivity follow directly from the from the full-rankedness of $~R$.  To establish the isometry property, observe that, 
    \[|~A|_{M,2} = \sup_{|~x|_M\neq 0} \frac{|~A~x|_M}{|~x|_M} = \sup_{|~R~x|\neq 0} \frac{|~R~A~x|}{|~R~x|} = \sup_{|~y|\neq 0} \frac{|~R~A~R^{-1}~y|}{|~y|} = |~R~A~R^{-1}|_2, \]
    where $~y=~R~x$.
\end{proof}


\Cref{prop:M-isom} provides a direct link between the $~M$-norm and its Euclidean counterpart, leading to analogues of basic inequalities such as that below.

\begin{corollary}\label{cor:M-CS}
    Let $~A,~B\in\mathbb{R}^{N\times N}$ and $~b\in\mathbb{R}^N.$
    The following Cauchy-Schwarz inequalities hold: $|~A~b|_M \leq |~A|_{M,2}|~b|_M$ and $|~A~B|_M\leq |~A|_{M,2}|~B|_M$.
\end{corollary}
\begin{proof}
For the first inequality,  observe that 
\[|~A~b|_M = \big|(~R~A~R^{-1})(~R~b)\big|\leq \big|~R~A~R^{-1}\big|_2|~R~b| = |~A|_{M,2}|~b|_M,\]
where the inequality is due to the usual (Euclidean) Cauchy-Schwarz.  The second now follows quickly: letting $~b_j\in\mathbb{R}^N$ denote the columns of $~B$,
\[|~A~B|_M = \sum_{j=1}^N |~A~b_j|_M \leq |~A|_{M,2}\sum_{j=1}^N |~b_j|_M = |~A|_{M,2}|~B|_M. \qedhere \]
\end{proof}

With these notions in place, it follows that the representation error of the HOSVD basis from \Cref{alg:M_tucker_basis} can be bounded in an interpretable fashion. 


\begin{proposition}\label{prop:representation-error}
    Under the assumptions above, the representation error of $~U_{\mu}$ at any $~\mu\in\Omega$ is bounded as 
    \[\big|~P_{U_{\mu}}^\perp~X(~\mu)\big|_M^2 \leq C_e^2\varepsilon^2|~X|_M^2 + \delta^2|~X|_{M\times C^p(\Omega)}^2 + \sum_{i=n+1}^R\sigma_{\mu,i}^2,\]
    in terms of the error in the Tucker factorization (first term), parameter interpolation (second term), and local SVD (third term).  Here, $|~X|_{M\times C^p(\Omega)} = \sum_{i,j,\alpha}[~M]_{ij}|x^i(t_{\alpha},\cdot)|_{C^p(\Omega)}|x^j(t_{\alpha},\cdot)|_{C^p(\Omega)}$ and $\{\sigma_{\mu,i}\}$ denote the entries of $~\Sigma_{\mu}$.
\end{proposition}
\begin{proof}
    The argument is analogous to the proof of \cite[Theorem 4.1]{mamonov2022interpolatory}.  Consider the $~M$-orthonormal Tucker factorization 
    \begin{align*}
        \tilde{~{X}} = \sum_{i=1}^{\bar{N}}\sum_{\alpha=1}^{\bar{N}_t}\sum_{s=1}^{\bar{N}_s} [~{C}]_{i\alpha s}~{w}^i\otimes~{\tau}^\alpha\otimes~{\nu}^s \approx ~X,
    \end{align*}
    which defines the core tensor $~C\in\mathbb{R}^{\bar N\times\bar T\times\bar P}$ and the basis matrices $~W\in \mathbb{R}^{N\times \bar N}$, $~T\in\mathbb{R}^{T\times\bar T}$, $~S\in\mathbb{R}^{P\times\bar P}$ satisfying $~W^\intercal~M~W=~I$, $~T^\intercal~T=~I$, and $~S^\intercal~S=~I$.  There are two cases to consider:

    \paragraph{Case 1: $~\mu$ in snapshot set}
    If $~\mu$ corresponds to one of the snapshots in the training set, it is straightforward to extract the local snapshot matrix $~X_{\mu}\in\mathbb{R}^{N\times T}$ and its approximation $\tilde{~X}_{\mu}\in\mathbb{R}^{N\times T}$: if $~\mu = ~\mu_i$ is the $i^{\rm th}$ parameter sample, simply form $~X_{\mu} = ~X~e_i$ and $\tilde{~X}_{\mu} = \tilde{~X}~e_i$ where $~e_i = (0,...0,1,0,...,0) \in \mathbb{R}^{N_s}$ is the unit vector with $1$ in the $i^{th}$ position.  Defining the local $~M$-orthonormal basis $~U_{\mu}$ through $\tilde{~X}_{\mu}~V_{\mu} = ~U_{\mu}~\Sigma_{\mu}$ using \Cref{alg:M_tucker_basis}, it follows that the projection error satisfies 
\begin{equation*}
\begin{split}
    \big|~P^\perp_{U_{\mu}}~X_{\mu}\big|_M^2 &\leq \big|~P^\perp_{U_{\mu}}\big(~X_{\mu}-\tilde{~X}_{\mu}\big)\big|_M^2 + \big|~P^\perp_{U_{\mu}}\tilde{~X}_{\mu}\big|_M^2 \leq \big|~P^\perp_{U_{\mu}}\big|_{M,2}^2\big|(~X-\tilde{~X})~e_i\big|_M^2 + \big|~P^\perp_{U_{\mu}}\tilde{~X}_{\mu}\big|^2_M \\
    &\leq \big|~X-\tilde{~X}\big|_M^2 + \sum_{i=n+1}^{R}\sigma_{\mu,i}^2 \leq \varepsilon^2|~X|_M^2 + \sum_{i=n+1}^{R}\sigma_{\mu,i}^2,
\end{split}
\end{equation*}
where $R$ is the rank of $\tilde{~X}_{\mu}$.  Here, the first inequality is the triangle inequality, the second is Cauchy-Schwarz \Cref{cor:M-CS}, and the third uses that $~P_{U_\mu}^\perp$ is $~M$-orthogonal projection along with the energy criterion \eqref{eq:proj-error-mono}. 

\paragraph{Case 2: $~\mu$ not in snapshot set}
In the case that $~\mu$ is not contained in the training snapshots, obtaining the corresponding error bound 
requires more effort.  
First, let $~X(~\mu)$ denote the inaccessible high-fidelity solution at $~\mu$, and let $~X_{\mu},\tilde{~X}_{\mu}$ denote the local snapshot matrices constructed as in the previous case, with $~e_i$ replaced by the generalized index vector $~e(~\mu)$.  By the previous argument, it follows that  
\begin{equation*}
    \big|~P_{U_{\mu}}^\perp~X(~\mu)\big|_M^2 \leq \big|~P_{U_{\mu}}^\perp\big(~X(~\mu)-\tilde{~X}_{\mu}\big)\big|_M^2 + \big|~P_{U_{\mu}}^\perp\tilde{~X}_\mu\big|_{M}  \leq \big|~X(~\mu)-\tilde{~X}_{\mu}\big|^2_M + \sum_{i=n+1}^R \sigma_{\mu,i}^2,
\end{equation*}
since the basis $~U_{\mu}$ is constructed from the SVD of $\tilde{~X}_{\mu}$.  The first term can be split in two:
\begin{equation*}
    \big|~X(~\mu)-\tilde{~X}_{\mu}\big|^2_M \leq \big|~X(~\mu)-~X_{\mu}\big|^2_M + \big|~X_{\mu}-\tilde{~X}_{\mu}\big|_M^2.
\end{equation*}
The second term is controllable in view \Cref{thm:M-tucker} and the stability \Cref{assump:stability},
\begin{equation*}
    \big|~X_{\mu}-\tilde{~X}_{\mu}\big|_M^2 = \big|(~X-\tilde{~X})~e(~\mu)\big|_M^2 \leq C_e^2\big|~X-\tilde{~X}\big|_M^2 \leq C_e^2\varepsilon^2|~X|^2_M.
\end{equation*}
Similarly, the first term is controllable via the interpolation \Cref{assump:bounded-error},
\begin{align*}
    \big|~X(~\mu)-~X_{\mu}\big|^2_M &= \sum_{i,j,\alpha}[~M]_{ij}\big(x^i(t_{\alpha},~\mu) - ~e(~\mu)^\intercal x^i(t_{\alpha},~S)\big)\big(x^j(t_{\alpha},~\mu) - ~e(~\mu)^\intercal x^j(t_{\alpha},~S)\big) \\
    &\leq \sum_{i,j,\alpha}[~M]_{ij}\big|x^i(t_{\alpha},~\mu) - ~e(~\mu)^\intercal x^i(t_{\alpha},~S)\big|\big|x^j(t_{\alpha},~\mu) - ~e(~\mu)^\intercal x^j(t_{\alpha},~S)\big| \\
    &\leq \delta^2 \sum_{i,j,\alpha}[~M]_{ij}\big|x^i(t_{\alpha},\cdot)\big|_{C^p(\Omega)}\big|x^j(t_{\alpha},\cdot)\big|_{C^p(\Omega)} = \delta^2\nn{~X}^2_{M\times C^p(\Omega)}
\end{align*}
Putting this together yields the desired error bound
\begin{equation*}
    \big|~P_{U_{\mu}}^\perp~X(~\mu)\big|_M^2 \leq C_e^2\varepsilon^2|~X|^2_M + \delta^2|~X|^2_{M\times C^p(\Omega)} + \sum_{i=n+1}^R \sigma_{\mu,i}^2. \qedhere
\end{equation*}
\end{proof}

\Cref{prop:representation-error} shows that there are three components to the representation error of the local basis $~U_{\mu}$.  Namely, reconstruction error from the HOSVD, interpolation error from the parameter sampling, and truncation error from the local SVD.  These must be balanced for an effective model reduction,  but controlling representation error alone is not 
enough for a satisfactory ROM.  Previous work \cite{mamonov2025priori} has shown that the tensorial ROM \eqref{eq:MO-ROM} with basis \Cref{alg:tucker_basis} has controllable error in the special case of time-dependent parabolic PDEs.  The next goal is to establish a similar guarantee for the $~M$-orthonormal basis \Cref{alg:M_tucker_basis} given any Lipschitz-continuous right-hand side $~f$.  First, note the following assumption.


\begin{assumption}\label{assump:lipschitz-f}
    The function $~f:\mathbb{R}^N\to\mathbb{R}^N$ defining the FOM $\dot{~x}=~f(~x)$ is $~M$-Lipschitz continuous on the domain $K$ of $~x$, i.e., there exists a constant $L_f>0$ such that 
    \[|~f(~x)-~f(~y)|_M \leq L_f|~x-~y|_M, \qquad \forall\,~x,~y\in K.\]
    Equivalently, the function $~R~f$ is Lipschitz continuous where $~M=~R^\intercal~R$.
\end{assumption}

\begin{remark}
    \Cref{assump:lipschitz-f} is clearly satisfied by any Lipschitz continuous $~f$, since $$|~R~f(~x)-~R~f(~y)| \leq |~R|_{2}|~f(~x)-~f(~y)| = \lambda_{\rm max}(~M)^{1/2}|~f(~x)-~f(~y)|.$$
\end{remark}

A natural norm measuring the ROM error is the $~M$-weighted $L^2$ error in time, 
\[\nn{~x(~\mu)}_M^2 = \int_0^\tau |~x(t,~\mu)|^2_M\,dt, \]
where $\tau>0$ denotes the final time under consideration.  With this in hand, there is the following estimate on the error in the proposed ROM.  

\begin{theorem}\label{thm:ROM-error}
    Consider an autonomous FOM $\dot{~x}(~\mu) = ~f(~x(~\mu))$ with $~f$ satisfying \Cref{assump:lipschitz-f} and a corresponding Galerkin ROM $\dot{\hat{~x}}(~\mu) = ~U_{\mu}^*~f(~U_{\mu}\hat{~x}(~\mu))$ defined by the local $~M$-orthonormal basis $~U_{\mu}$.  Then, the ROM error at any $~\mu\in\Omega$ satisfies 
    \[\nn{~x(~\mu) - ~U_{\mu}\hat{~x}(~\mu)}_M \leq \big(1 + L_f C(\tau)\sqrt{\tau}\big)\big\|~P_{U_{\mu}}^\perp~x(~\mu)\big\|_M,\]
    in terms of constants $L_f$ and $C(\tau)$, along with the representation error $~P^\perp_{U_{\mu}}~x(~\mu)$.
\end{theorem}
\begin{proof}
    First, write $~x=~x(~\mu)$ and observe that the error decomposes:
    \begin{equation*}
        ~x - ~U_{\mu}\hat{~x} = (~x - ~P_{U_{\mu}}~x) + (~P_{U_{\mu}}~x - ~U_{\mu}\hat{~x}) = ~P_{U_{\mu}}^\perp~x + ~y,
    \end{equation*}
    in terms of the $~M$-orthogonal projection $~P_{U_{\mu}}~x = ~U~U^*~x$.  
    The $~M$-norm of the first term is bounded by the work above.  The second term can be bounded with Gronwall.  More precisely, the
    $~M$-norm of the derivative $\dot{~y}$ satisfies
    \[|\dot{~y}|_{M} = \big|~P_{U_\mu}\dot{~x} - ~U_{\mu}\dot{\hat{~x}}\big|_M \leq \big|~P_{U_\mu}\big|_M \big|~f(~x)-~f(~U_{\mu}\hat{~x})\big|_M \leq \big|~f(~x)-~f(~U_{\mu}\hat{~x})\big|_M, \]
    where \Cref{cor:M-CS} was used in the first inequality, and the second equality used that $~P_{U_\mu}$ is an $~M$-orthogonal projection matrix.
    Applying the triangle inequality along with \Cref{assump:lipschitz-f} and the fact that the time derivative of the norm $\partial_t|~y|_M = (2|~y|_M)^{-1}\langle\dot{~y},~y\rangle_M \leq |\dot{~y}|_M$ is less than the norm of the time derivative yields a first-order differential equation for $|~y|_M$:
    \[\partial_t|~y|_M \leq |\dot{~y}|_M \leq \big| ~f(~x)-~f(~P_{U_{\mu}}~x) \big|_M + \big| ~f(~P_{U_{\mu}}~x)-~f(~U_{\mu}\hat{~x}) \big|_M \leq L_{f}\big(\big|~P_{U_\mu}^\perp~x\big|_M + |~y|_M\big). \]
    Using positivity of the integrating factor $e^{-L_{f}t}$, it follows that 
    \[\partial_t\big(e^{-L_{f}t}|~y|_M\big) = e^{-L_{f}t}\big(\partial_t|~y|_M - L_{f}|~y|_M\big) \leq e^{-L_{f}t}L_{f}\big|~P_{U_{\mu}}^\perp~x\big|_M, \]
    and therefore integration in view of the condition $~y(0)=~0$ implies
    \begin{align*}
        |~y|_M &\leq L_{f}\int_{0}^t e^{L_{f}(t-s)}\big|~P_{U_{\mu}}^\perp~x(s)\big|_M\, ds \leq L_{f}\int_0^\tau e^{L_{f}(\tau-s)}\big|~P_{U_{\mu}}^\perp~x(s)\big|_M\, ds \\
        &\leq L_{f}\left(\int_{0}^\tau e^{2L_{f}(\tau-s)}\, ds\right)^{\frac{1}{2}}\left(\int_{0}^\tau \big|~P_{U_{\mu}}^\perp~x(s)\big|_M^2\, ds\right)^{\frac{1}{2}} \leq C L_f\big\|~P_{U_{\mu}}^\perp~x\big\|_M,
    \end{align*}
    where $C = \|e^{2L_{f}(\tau-\cdot)}\| = \sqrt{(e^{2L_{f}\tau}-1)/(2L_{f})} $ is a constant of integration.
    Integrating again in time then yields 
    \[\|~y\|_M \leq L_fC(\tau)\sqrt{\tau} \big\|~P_{U_{\mu}}^\perp~x\big\|_M, \]
    and therefore the desired ROM error is bounded as
    \[ \|~x(~\mu) - ~U_{\mu}\hat{~x}(~\mu)\|_M \leq \big\|~P_{U_{\mu}}^\perp~x\big\|_M + \|~y\|_M \leq \big(1 + L_fC(\tau)\sqrt{\tau}\big)\big\|~P_{U_{\mu}}^\perp~x\big\|_M. \qedhere \]
\end{proof}

\Cref{thm:ROM-error} ensures that the modifications to the tensorial ROM \eqref{eq:MO-ROM} proposed in \Cref{sec:ours} produce a model with controlled error provided the system under consideration is first-order and the right-hand side $~f$ is well behaved.  Moreover, this guarantee persists regardless of the character of the PDE being discretized.  \rev[anthony]{}{This provides an \emph{a priori} justification of the method, although the constants involved may be difficult to compute in practice.}  Note that similar error bounds have appeared in the ROM literature in many places, e.g., \cite{chaturantabut2012state,GONG2017780,Gruber:2023,gruber2024vc,geng2024gradopinf}, although no previous tensorial ROM strategy has supplied such a bound.

\section{Numerical results}\label{sec:numerics}

Now that the structure-aware tensorial ROM has been introduced and its theoretical properties have been examined, it remains to compare its performance to previous model reduction strategies.  This section evaluates the proposed methodology on three parametric PDE problems of increasing complexity: a 2D scalar heat equation, a 2D scalar wave equation, and a full 3D Maxwell system. In addition, each benchmark problem compares three distinct approaches to constructing the reduced bases needed for ROM deployment:

\begin{enumerate}
    \item \textbf{Monolithic SVD}: a single, parameter-independent
        $~M$-orthonormal basis $~U\in\mathbb{R}^{N\times n}$ obtained from the truncated SVD of the matricized snapshot tensor $~X_{(1)}\in\mathbb{R}^{N\times TP}$.  This is the ``default'' model reduction strategy commonly employed for parametric problems involving mass matrices. 

    \item \textbf{HOSVD with Mamonov/Olshanskii (MO) interpolation}: the structure-aware $~M$-orthonormal basis $~U_{\mu}\in\mathbb{R}^{N\times n}$ generated by \Cref{alg:M_tucker_basis} with the distance-weighted least-squares interpolation scheme \cref{eq:gen-index-MO} from \cite{mamonov2022interpolatory}.  This represents a modest improvement on the current state-of-the-art in tensorial ROM strategies, where the structure-aware basis generation technique introduced in this work is combined with previous technology.

    \item \textbf{HOSVD with Gaussian Radial Basis Function interpolation}: the structure-aware $~M$-orthonormal basis $~U_{\mu}\in\mathbb{R}^{N\times n}$ generated by \Cref{alg:M_tucker_basis} with the proposed GRBF interpolation of \Cref{sec:ours}.  This is the full structure-aware tensorial ROM strategy proposed in this work, with the choice of a global GRBF interpolant to promote smoothness across the parameter space.  
\end{enumerate}

Since the Monolithic SVD is parameter-independent, it produces a single, global basis from
the full snapshot data and does not adapt to any query parameter $~\mu\in\mathbb{R}^p$.  Therefore, its performance is a reference for the ROM accuracy achievable without basis adaptivity.  Conversely, both HOSVD strategies are nonlinear across parameter space, enabling improved performance at somewhat increased computational overhead.  Note that all experiments report errors with respect to the $~M$-weighted Frobenius norm which bounds the $L^1([0,\tau],L^2(\Omega))$ inner product, defined for a discrete trajectory $~Q\in\mathbb{R}^{N\times T}$ as $|~Q|_M = \sqrt{\sum_{\alpha} ~q_\alpha^\intercal ~M ~q_\alpha}$, where $~q_\alpha$ denotes the $\alpha$-th column of $~Q$. Code for reproducing the numerical experiments is available in the repository: \url{https://github.com/arjunveejay/TuckerROMs}.

\begin{remark}\label{rem:GRBF-proof}
    Note that the proposed HOSVD with global GRBF interpolation strategy does fulfill the theoretical \Cref{assump:stability}.  In particular, $|~e(~\mu)| \leq |~K_\varepsilon^{-1}||~k_{\varepsilon}(~\mu)| \leq \sqrt{P}|~K_{\varepsilon}^{-1}| < \infty$ since each of the (maximum of $P$) entries of $~k_{\epsilon}(~\mu)<1$ is less than one and $~K_\varepsilon$ is a positive definite kernel (for distinct $~\mu_i$).  Similarly, the GRBF interpolant satisfies an analogue of \Cref{assump:bounded-error} in the native-space norm of the Gaussian kernel, i.e., $\sup_{~\mu\in\Omega} \big|g(~\mu) - ~e(~\mu)^\intercal~g(~S)\big| \leq |g|_{\mathcal{N}_{\phi}(\Omega)}\delta,$ where $\delta$ decays exponentially with the fill distance; cf. \cite[Thm.~11.22]{Wendland2004}.
    These assumptions are sufficient for the error bound in \Cref{thm:ROM-error} with the second term on the right-hand side replaced by $\delta^2|~X|^2_{M\times \mathcal{N}_{\phi}(\Omega)}.$
\end{remark}

\subsection{Heat equation}\label{subsec:heat}

As a first demonstration, consider the forced heat equation with uniform diffusion coefficient:
\begin{equation}\label{eq:heat}
    \left\{\begin{aligned}
        &\frac{\partial}{\partial t}y(~ x,t)
    = \Delta y(~ x,t) + f(~ \mu; ~ x, t), \qquad t\in[0,\tau], \quad  &&~x \in \Omega, \\
    & y(~ x,t) = 0, \hspace{4.05cm} t\in[0,\tau], \quad && ~ x \in \partial \Omega, \\
    & y(~{x},0) = 0, \quad && ~ x \in \Omega,
    \end{aligned}\right.
\end{equation}
where $\Omega = [0,2\pi]^2$ is a two-dimensional domain with homogeneous Dirichlet boundary conditions.  Here,
$f$ is a prescribed, time-varying forcing term parameterized by $~\mu = [\mu_1, \mu_2, \mu_3]^\intercal$,
\begin{align*}
    f(~\mu; ~ x, t) =  \mu_1 \exp{\left(-\frac{(x-\mu_2)^2+(y-\mu_3)^2}{2\sigma^2}\right)}
    \sin\!\left(\frac{x}{2}\right)\sin\!\left(\frac{y}{2}\right) e^{-t}.
\end{align*}
The spatial component of this forcing is the product of two terms: a fixed-width Gaussian ($\sigma=0.4$), centered at $(\mu_2, \mu_3)$ with amplitude $\mu_1$, along with a sinusoidal term ensuring compatibility with the boundary conditions.  Its temporal component is simple exponential decay guaranteeing a reduction in forcing over time.  

The FOM for \cref{eq:heat} is constructed through the continuous Galerkin finite-element method.  Given a triangulation 
$\Omega_h$ of the domain, consider the $H^1$-conforming finite-element space
\begin{align*}
    V_h := \{ v_h \in H^1(\Omega) \; : \; v_h |_{K} \in \mathrm P_1(K), \; \forall K \in \Omega_h \; ; \; v_h|_{\partial \Omega} = 0 \},
\end{align*}
where $\mathrm P_1(K)$ is the space of linear polynomials on element $K$.  The weak Galerkin formulation of \cref{eq:heat} reads: find $q_h \in V_h$ such that
\begin{align}\label{eq:weakheat}
    (\dot q_h, v_h)_{\Omega_h} = -(\nabla q_h, \nabla v_h)_{\Omega_h} + (f(~\mu), v_h)_{\Omega_h} \qquad \forall\, v_h \in V_h,
\end{align}
where $(\cdot, \cdot)_{\Omega_h}$ denotes the $L^2$ inner product on $\Omega_h$.  Expanding in the nodal basis $\{\phi_i\}$ of $V_h$ yields the matrix-vector system
\begin{align}\label{eq:matvecheat}
    ~M \dot{~q}(t) = ~A ~q (t) + ~f(~\mu; t), \qquad ~q_0 = ~q(0) = ~0.
\end{align}
where $~q\in\mathbb{R}^N$ collects the nodal degrees of freedom, $~M$ is the mass matrix, $~A$ is the (negative) stiffness matrix, and $~f$ denotes the load vector with components
\begin{align*}
    f_i(~\mu, t) = (f(~\mu; \cdot, t), \phi_i)_{\Omega_h}.
\end{align*}
The system \eqref{eq:matvecheat} provides the data source used to generate the snapshots $[~X]_{i\alpha s} = [~q(t_\alpha,~\mu_s)]_i$ that serve as training data for the reduced bases.

The corresponding Galerkin ROM is formed via straightforward projection.  Given a reduced basis $~U\in\mathbb{R}^{N\times n}$, this means substituting the approximation $\tilde{~q} = ~U \hat{~q} \approx ~q$ into \cref{eq:matvecheat}, where $\hat{~q}\in\mathbb{R}^n$ is a vector of coefficients. Testing the resulting expression against $~U$ yields the system
\begin{align}\label{eq:heatROMfull}
    ~U^\intercal ~M ~U \dot{\hat{~q}}(t) = ~U^\intercal ~A ~U \hat{~q}(t) + ~U^\intercal ~f(~\mu, t), \qquad \hat{~q}_0 = ~0.
\end{align}
When the basis $~U$ is $~M$-orthonormal, the reduced mass matrix $~U^\intercal ~M ~U = ~I$ simplifies to the identity, and \cref{eq:heatROMfull} simplifies to
\begin{align}\label{eq:heatROM}
    \dot{\hat{~q}}(t) = \hat{~A}\, \hat{~q}(t) + \hat{~f}(~\mu, t), \qquad \hat{~q}_0 = ~0,
\end{align}
where $\hat{~A} = ~U^\intercal ~A ~U$ and $\hat{~f} = ~U^\intercal ~f$.  
In the absence of forcing, this guarantees that the ROM \eqref{eq:heatROM} is a discrete gradient flow, as discussed in \Cref{sec:ours}.  This is the reduced system that is integrated in all experiments below.

\subsubsection{Experimental details}
The FEM discretization \eqref{eq:matvecheat} yields a heat system with  
$N=1082$ degrees of freedom that 
is integrated in time with the implicit Euler method. The interval of integration is $[0, \pi]$ and the step-size is $\Delta t = \pi/1200$, producing $T = 1201$ snapshots per parameter instance.  The amplitude $\mu_1$ is sampled uniformly from $(0,1)$ and the center $(\mu_2, \mu_3)$ uniformly from $\Omega=[0,2\pi]^2$, for a total of $P = 200$ samples split into 160 training and 40 testing instances.

The training snapshots form the tensor $~X \in\mathbb{R}^{1082\times 1201 \times 160}$.  Applying \Cref{alg:M_tucker_basis} with Tucker ranks $(\bar N, \bar T, \bar P) = (120, 120, 120)$ yields a relative $~M$-weighted representation error of $3.25 \times 10^{-4}$. For a given test parameter $~\mu$, the generalized index vector $~e(~\mu)$ is constructed using both the MO interpolant with \rev[anthony]{$15$}{$10$} nearest neighbors and the RBF interpolant with shape parameter \rev[anthony]{$\varepsilon=1$}{$\varepsilon = 1.2$}; \rev[anthony]{}{these hyper-parameters are chosen using the leave-one-out cross-validation (LOOCV) procedure described in \Cref{app:interpolation}}. A rank-$n$ basis $~U_\mu$ is then extracted from the local core matrix $~C_\mu$ as described in \Cref{alg:M_tucker_basis}. The monolithic baseline is obtained from the thin SVD of the column-wise matricization $~X_{(1)}\in\mathbb{R}^{N\times T P}$ (appropriately weighted), retaining the leading $~M$-orthonormal left singular vectors. 

\begin{figure}[ht]
  \centering
  \includegraphics[width=0.6\linewidth]{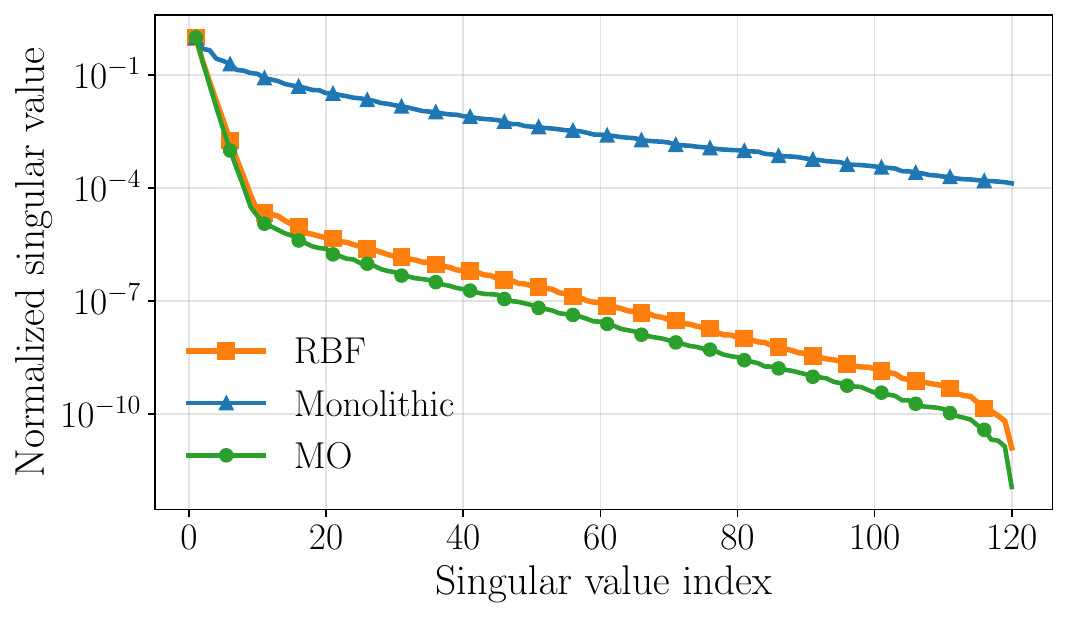}
  \caption{\rev[anthony]{}{Normalized singular value decay for the three bases $~U = ~U_{\mu}$ compared in the heat equation experiment, with the tensorial HOSVD bases constructed at a randomly selected test parameter. 
  Observe that the HOSVD bases exhibit significantly faster singular value decay due to their local adaptation.}}\label{fig:sv_heat}
\end{figure}

\begin{figure}[ht]
    \centering
    \includegraphics[width=\linewidth]{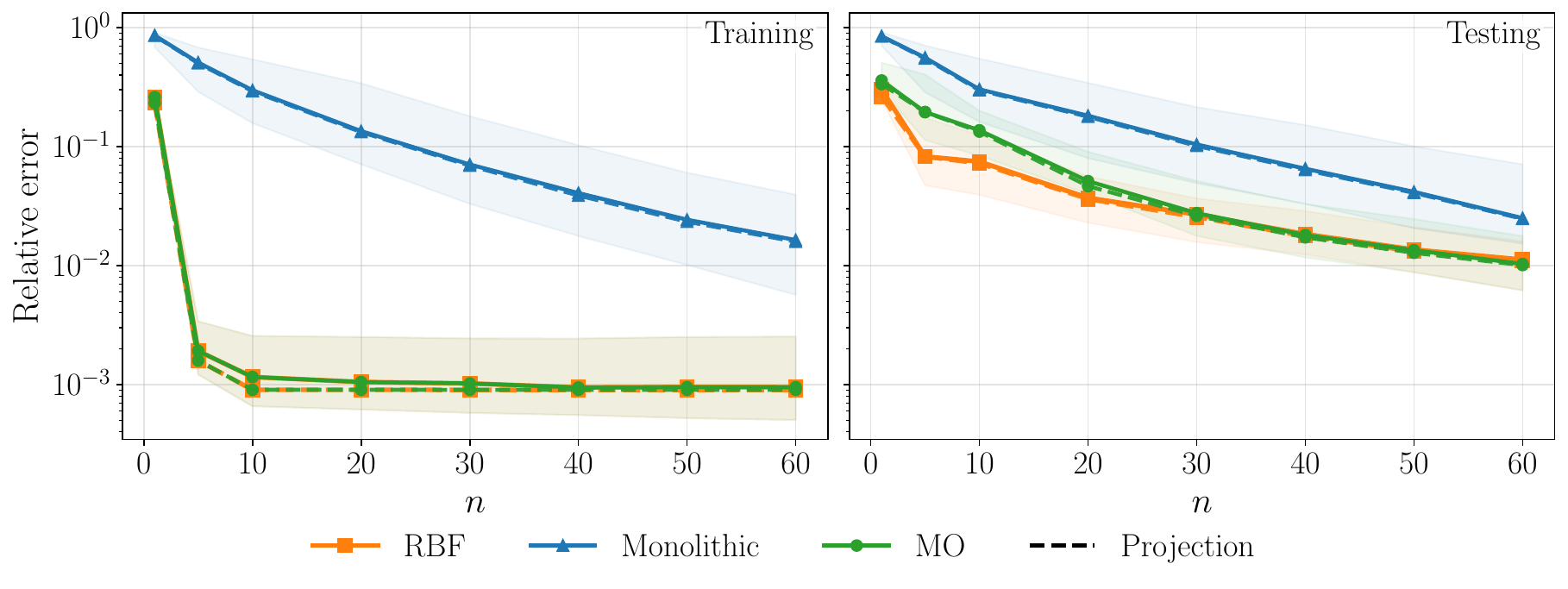}
    \caption{\rev[anthony]{}{Relative $~M$-weighted error in the ROM solutions to the heat system as a function of reduced basis dimension $n$, considered over all training (left) and testing (right) parameters.  For each method, solid lines denote the median ROM error across parameter instances while dashed lines denote the median projection error.  Shaded bands indicate the interquartile range of the ROM error.}} \label{fig:errors_heat}
\end{figure}

\begin{figure}[ht]
  \centering

  \includegraphics[width=0.95\linewidth]{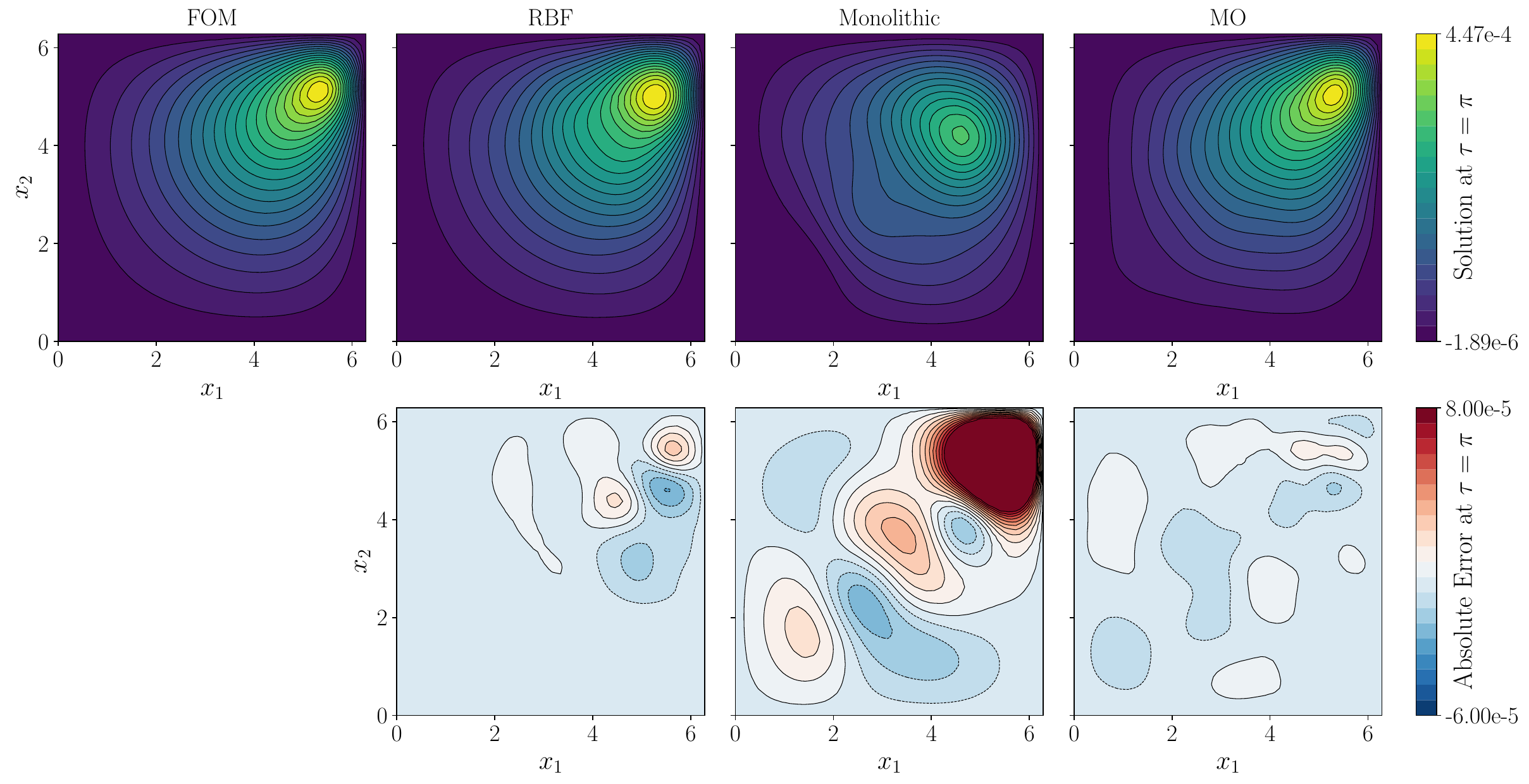}

  \caption{\rev[anthony]{}{Top row: FOM and ROM solutions to the heat system at terminal time $\tau=\pi$ for a sample test parameter with reduced state dimension $n=10$.  Bottom row: pointwise signed error in each ROM solution. While the RBF and MO solutions meaningfully approximate the FOM solution, the monolithic ROM exhibits significant distortion, dramatically underpredicting the maximum value of the solution.}}\label{fig:solution_heat}
\end{figure}

\Cref{fig:sv_heat} shows the normalized singular value decay of the three bases for a representative test parameter.  Both Tucker-based bases exhibit significantly faster decay than the monolithic SVD basis, with their normalized singular values falling below $10^{-5}$ within the first 20 modes.  The slower decay of the monolithic basis reflects the need to simultaneously represent solutions across all training parameters with a single global basis. Conversely, the tensorial bases locally and nonlinearly adapt to each query parameter, naturally circumventing the slow Kolmogorov $n$-width decay in parameter space.

\Cref{fig:errors_heat} reports the median relative $~M$-weighted ROM error (solid lines) and median projection error (dashed lines) of each approach as a function of the reduced basis dimension $n$, with shaded bands indicating the interquartile range of the ROM errors.  Consistent with \Cref{thm:ROM-error}, the projection errors represent the best approximations achievable by the given reduced basis independent of time integration, and bound the ROM errors from below.  Conversely, the ROM errors include the additional effect of integrating the reduced system \cref{eq:heatROM}.  For all three methods, this ROM error closely tracks the corresponding projection error, indicating that time integration introduces negligible additional error, as expected for uniformly elliptic problems. 


The error trends in \Cref{fig:errors_heat} are similarly informative.  On the training set (left panel), the errors for both tensor-based approaches are identical, dropping sharply to the order of $0.1\%$ by $n=10$ and plateauing thereafter, while the monolithic SVD error decays steadily but remains above $1\%$ even at $n=60$.  Since querying the tensorial ROMs on the training set requires no interpolation, the identical performance of RBF and MO in training is expected and shows that the weighted HOSVD strategy in \Cref{alg:M_tucker_basis} is highly effective in compressing dynamically relevant information.  On the testing set (right panel), the monolithic error curve exhibits similar behavior, but the tensorial ROM errors are noticeably higher than in training, settling near $1\%$ at $n=60$.  This train-test gap reflects the interpolation error inherent in constructing the generalized index vector $~e(~\mu)$ for unseen parameters, c.f. \Cref{prop:representation-error}.  Nevertheless, the tensorial ROMs still outperform the monolithic baseline across all values of $n$, with the RBF interpolant achieving the lowest median error for $n<30$ and the MO interpolant thereafter.

A representative visualization is provided by \Cref{fig:solution_heat}, which compares the FOM and ROM solutions at terminal time $\tau=\pi$ for $n=10$ and a representative test parameter.  Observe that the RBF and MO solutions approximate the FOM contours faithfully, with relative errors of \rev[anthony]{$11.5\%$}{9.1\%} and \rev[anthony]{$14.9\%$}{6.5\%}, respectively.  The monolithic ROM, by contrast, produces an over-smoothed and visibly distorted solution with spurious nonconvexity in its contours and a relative error of $71.1\%$. It is remarkable that most of this error is concentrated near the forcing center, where the monolithic ROM fails to resolve finer-scale details of the solution.  Moreover, while both tensor-based strategies can resolve this, their distributions of error are quite different.



\subsection{Wave equation}\label{subsec:wave}

Another, more challenging example is provided by the forced wave equation with uniform wave speed:
\begin{equation}\label{eq:wave}
    \left\{\begin{aligned}
        &\frac{\partial^2}{\partial t^2}y(~ x,t)
    = \Delta y(~ x,t) + f(~ \mu, ~ x, t), \qquad t\in[0,\tau], \quad  &&~ x \in \Omega, \\
    &  y(~ x,t) = 0, \hspace{4.2cm} t\in[0,\tau], \quad && ~ x \in \partial \Omega, \\
    & y(~{x},0) = 0, \quad \frac{\partial}{\partial t} y(~ x,0) = 0, \quad && ~ x \in \Omega,
    \end{aligned}\right.
\end{equation}
where $\Omega = [0, 2\pi]^2$ is a two-dimensional domain with homogeneous Dirichlet boundary conditions.  Here,
$f$ is a prescribed, time-varying parametric forcing term given by
\begin{align*}
    f(~ \mu, ~ x, t) = \exp\left({-\frac{(x-\mu_2)^2+(y-\mu_3)^2}{2\sigma^2}}\right)\sin\!\left(\frac{x}{2}\right)\sin\!\left(\frac{y}{2}\right)\cos(\mu_1 t),
\end{align*}
in terms of a parameter vector $~\mu = [\mu_1 \; \mu_2 \; \mu_3]^\intercal$. The main difference between this forcing function and that from the previous heat equation case is the presence of periodicity in time.  This forcing does not decay and its frequency is parameterized by $\mu_1$.


Following \cite{vijaywargiya2025tensor}, a Hamiltonian structure-preserving FOM for \cref{eq:wave} is constructed by adapting the mixed finite-element scheme of \cite{SANCHEZ2021113843}, written in first-order form using canonical variables $q := y$ and $p := \partial_t y$.  Given a triangulation $\Omega_h = \{K_i\}_{i=1}^{N_E}$ of the spatial domain, define the finite-element spaces
\begin{subequations}\label{eq:wavespaces}
    \begin{align}
W_h &:= \{ w_h \in L^2(\Omega) : w_h|_K \in \mathrm{P}_k(K), \ \forall K \in \Omega_h \},\\
V_h &:= \{ ~v_h \in H(\mathrm{div},\Omega) : ~v_h|_K \in \mathrm{RT}_k(K), \ \forall K \in \Omega_h; \ ~v_h|_{\partial\Omega} = 0\},
\end{align}
\end{subequations}
where $\mathrm{P}_k(K)$ denotes the space of degree-$k$ polynomials and $\mathrm{RT}_k(K)$ the order-$k$ Raviart-Thomas space on element $K$.  The semi-discrete weak form at $(t,~\mu)$ reads: find $(\dot q_h, \dot p_h) \in W_h \times W_h$ such that
\begin{subequations}\label{eq:weakwave}
\begin{align}
(\dot{q}_h, w_h)_{\Omega_h} &= (p_h, w_h)_{\Omega_h}, &&\forall w_h \in W_h,\\
(\dot{p}_h, w_h)_{\Omega_h} &= (\nabla \cdot ~{\sigma}_h, w_h)_{\Omega_h} +  (f(~\mu), w_h)_{\Omega_h},  &&\forall w_h \in W_h,
\end{align}
\noindent where $~\sigma_h \in V_h$ is an intermediate variable satisfying
\begin{align}
    \left( ~\sigma_h, ~\xi_h \right)_{\Omega_h}
+ (q_h, \nabla \cdot ~\xi_h)_{\Omega_h}
= 0, \qquad \forall ~\xi_h \in V_h.
\end{align}
\end{subequations}
Collecting the degrees of freedom of $q_h$ and $p_h$ in vectors $~q$ and $~p$, \cref{eq:weakwave} is equivalently written in the matrix-vector form
\begin{align}\label{eq:waveFOM}
   \dot{~y}
=
\begin{bmatrix} \dot{~q} \\ \dot{~p} \end{bmatrix}
=
\begin{bmatrix}
~0 & ~I \\
- ~I & ~0
\end{bmatrix}
\begin{bmatrix}
~M_W^{-1} ~S^\intercal ~M_V^{-1} ~S & ~0 \\
~0 & ~I
\end{bmatrix}
\begin{bmatrix}
~q \\
~p
\end{bmatrix}
+
\begin{bmatrix}
    ~0\\
    ~M_W^{-1} {~f} (~\mu)
\end{bmatrix}
= 
~J\nabla^MH(~q,~p) + ~F(~\mu),
\end{align}
in terms of the canonical symplectic matrix $~J\in\mathbb{R}^{2N\times 2N}$ and the discrete Hamiltonian 
\[H(~q,~p) = \langle ~q,~M_W^{-1}~S^\intercal~M_V^{-1}~S~q\rangle_{~M_W} + |~p|^2_{~M_W}.\]
Here, $~M_W$ and $~M_V$ are the mass matrices associated with spaces $W_h$ and $V_h$, respectively, the entries of $~S$ are $S_{ji} = (\phi_i, \nabla \cdot ~\psi_j)_{\Omega_h}$, and the load vector components are $f_i(~\mu, t) = (f(~\mu; \cdot, t), \phi_i)_{\Omega_h}$.

Since both $q_h$ and $p_h$ reside in the same finite-element space $W_h$, a single $~M_W$-orthonormal basis $~U \in \mathbb{R}^{N \times n}$, obtained via the ``cotangent lift'' Proper Symplectic Decomposition (c.f. \cite{peng2016symplectic}), is used to approximate both fields, i.e., $~q \approx ~U\hat{~q}$ and $~p \approx ~U\hat{~p}.$
This is equivalent to approximating the full state $~y = [~q^\intercal\; ~p^\intercal]^\intercal$ with a block-diagonal basis
\begin{align}\label{eq:cotlift_block}
    \tilde{~y} =    \begin{bmatrix} ~U & ~0 \\ ~0 & ~U \end{bmatrix}
    \begin{bmatrix} \hat{~q} \\ \hat{~p} \end{bmatrix}
    \approx ~y.
\end{align}
The advantage of using a cotangent lift basis is equivariance with respect to the action of the symplectic group.  This means that canonical Hamiltonian FOMs such as \eqref{eq:waveFOM} directly yield canonical Hamiltonian ROMs under Galerkin projection; see \cite{peng2016symplectic,gruber2024vc,vijaywargiya2025tensor} for more details. Substituting the approximation \cref{eq:cotlift_block} into \cref{eq:waveFOM} and testing the $~q$- and $~p$-equations separately against $~U$ in the $~M_W$-weighted inner product gives the Galerkin ROM
\begin{subequations}\label{eq:waveROMfull}
\begin{align}
    ~U^\intercal ~M_W ~U\, \dot{\hat{~q}}(t) &= ~U^\intercal ~M_W ~U\, \hat{~p}(t), \\
    ~U^\intercal ~M_W ~U\, \dot{\hat{~p}}(t) &= -~U^\intercal ~S^\intercal ~M_V^{-1} ~S\, ~U\, \hat{~q}(t) + ~U^\intercal ~f(~\mu; t).
\end{align}
\end{subequations}
When $~U$ is constructed to be $~M_W$-orthonormal, $~U^\intercal ~M_W ~U = ~I$ and \cref{eq:waveROMfull} simplifies to
\begin{align}\label{eq:waveROM}
    \dot{\hat{~q}}(t) = \hat{~p}(t), \qquad
    \dot{\hat{~p}}(t) = -\hat{~A}\,\hat{~q}(t) + \hat{~f}(~\mu; t),
\end{align}
where $\hat{~A} = ~U^\intercal ~S^\intercal ~M_V^{-1} ~S\, ~U$ and $\hat{~f} = ~U^\intercal ~f$ denote the projected components of the momentum equation.  The block-diagonal structure of \cref{eq:cotlift_block} ensures that the reduced system \cref{eq:waveROM} inherits the Hamiltonian structure of the FOM:
\begin{align*}
      \dot{\hat{~y}}
=
\begin{bmatrix} \dot{\hat{~q}} \\ \dot{\hat{~p}} \end{bmatrix}
=
\begin{bmatrix}
~0 & ~I \\
- ~I & ~0
\end{bmatrix}
\begin{bmatrix}
\hat{~A} & ~0 \\
~0 & ~I
\end{bmatrix}
\begin{bmatrix}
\hat{~q} \\
\hat{~p}
\end{bmatrix}
+
\begin{bmatrix}
    ~0\\
    \hat{~f} (~\mu)
\end{bmatrix}
= ~J\nabla\hat{H}(\hat{~q},\hat{~p}) + \hat{~F}(~\mu),
\end{align*}
where $\hat{H} = H \circ {\rm blkdiag}(~U,~U)$ denotes the pullback of the original discrete Hamiltonian onto the span of the block-diagonal reduced basis.


Concretely, the reduced-order system $\eqref{eq:waveROM}$ is solved with bases constructed from the following data.  First, displacement and momentum snapshot matrices $~Q_i, ~P_i \in \mathbb{R}^{N \times T}$ , collected by solving the FOM \eqref{eq:waveFOM} at each training parameter $~\mu_i$, are concatenated along the temporal dimension to form the lifted snapshot matrix $[~Q_i \mid ~P_i] \in \mathbb{R}^{N \times 2T}$. 
Stacking these across all $P$ training parameters then yields the three-way tensor $~X \in \mathbb{R}^{N \times 2T \times P}$ of snapshot data.  From this, \Cref{alg:M_tucker_basis} is applied to $~X$ with mass matrix $~M_W$, and parameter-specific $~M_W$-orthonormal bases $~U_\mu$ are extracted from the local core matrix exactly as in the heat example.  The monolithic baseline is again obtained from the thin SVD of the (appropriately weighted) column-wise matricization of $~X_{(1)} \in \mathbb{R}^{N \times 2T P}$, retaining the leading $~M_W$-orthonormal left singular vectors.

\subsubsection{Experimental details}

The polynomial order of the finite-element spaces \cref{eq:wavespaces} is set to $k=2$, yielding $N=2160$ degrees of freedom for each of $~q$ and $~p$.  The FOM \cref{eq:waveFOM} is integrated in time over the interval $[0, 8\pi]$ with step-size $\Delta t = 8\pi/500$ using with the implicit midpoint rule, producing $T = 501$ snapshots per parameter instance.  The temporal frequency $\mu_1$ is sampled uniformly from $(0.01, 0.05)$ and the spatial center $(\mu_2, \mu_3)$ uniformly from $\Omega$, for a total of $P = 200$ samples split into 160 training and 40 testing instances.

The training snapshots form the lifted tensor $~X \in \mathbb{R}^{2160 \times 1002 \times 160}$.  Applying \Cref{alg:M_tucker_basis} with Tucker ranks $(\bar N, \bar T, \bar P) = (120, 120, 120)$ yields a relative $~M_W$-weighted representation error of $9.19 \times 10^{-3}$.  For a given test parameter $~\mu$, the generalized index vector $~e(~\mu)$ is constructed using both the MO interpolant with \rev[anthony]{$15$}{10} nearest neighbors and the RBF interpolant with shape parameter \rev[anthony]{$\varepsilon = 1$}{$\varepsilon = 1.2$}; these hyper-parameters are chosen using the LOOCV described in \Cref{app:interpolation}. A rank-$n$ basis $~U_\mu$ is then extracted as described in \Cref{alg:M_tucker_basis}.  All errors are reported in the $~M_W$-weighted Frobenius norm for the displacement field $~q$ and momentum field $~p$. 

\begin{figure}[ht]
    \centering
    \includegraphics[width=\linewidth]{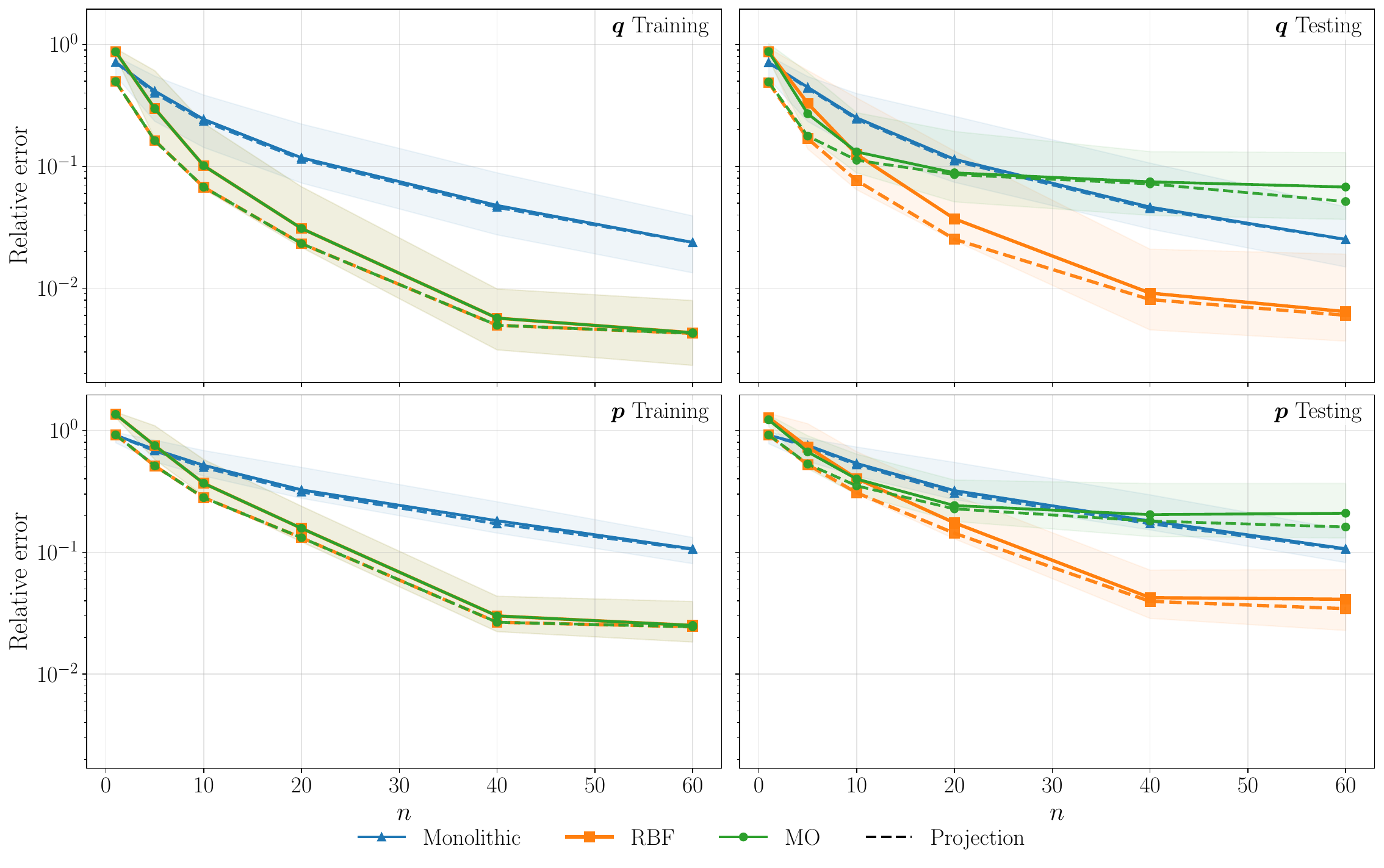}
    \caption{\rev[anthony]{}{Relative $~M_W$-weighted error in the ROM solutions to the wave system as a function of reduced basis dimension $n$, considered over all training (left) and testing (right) parameters.  Top row: displacement error; bottom row: momentum error. For each method, solid lines denote the median ROM error across parameter instances while dashed lines denote the median projection error.  Shaded bands indicate the interquartile range of the ROM error. }}\label{fig:errors_wave}
\end{figure}

\begin{figure}[ht]
  \centering
  \includegraphics[width=\linewidth]{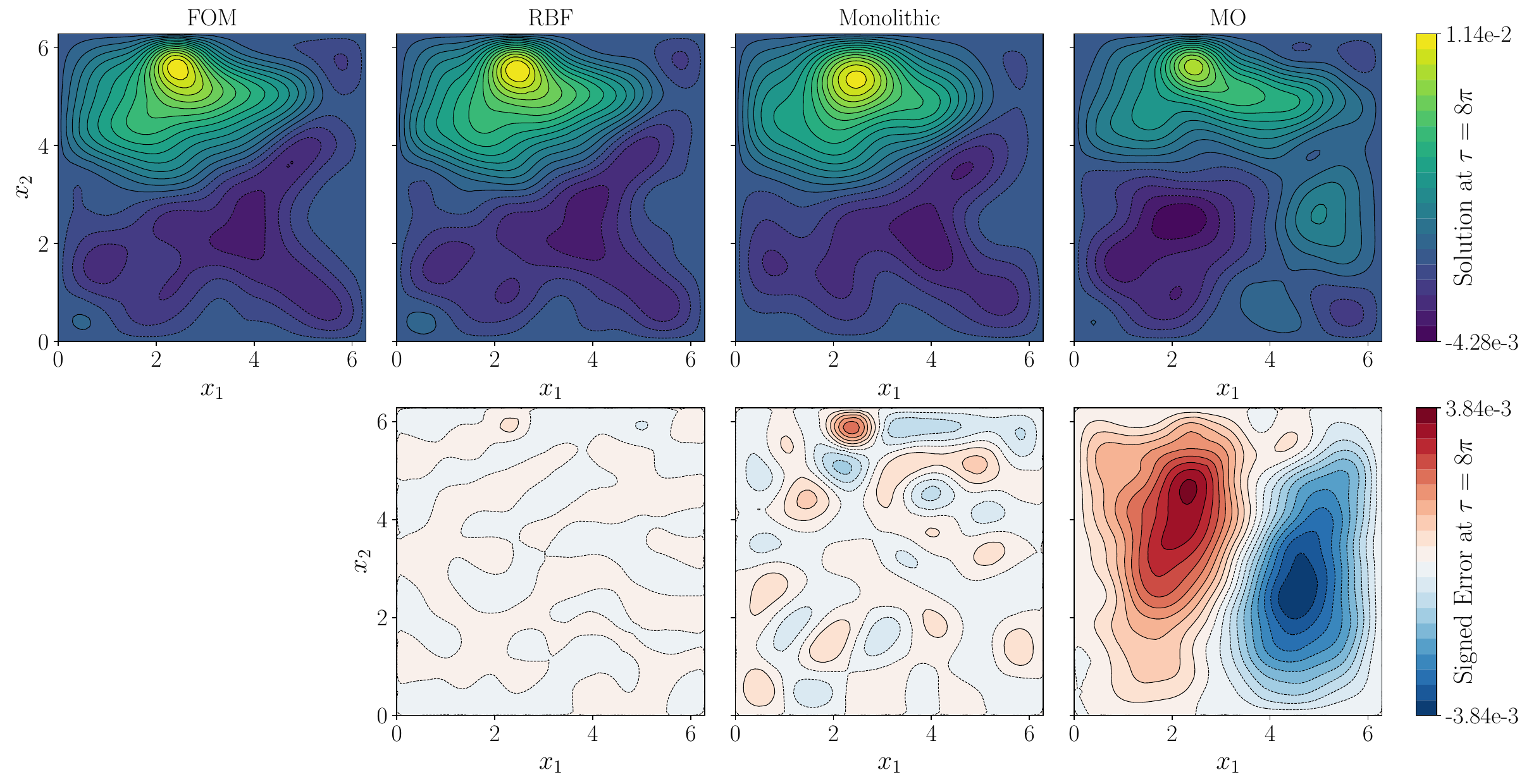}
  \caption{\rev[anthony]{}{Top row: FOM and ROM displacement solutions to the wave system at terminal time $\tau=8\pi$ for a sample test parameter with reduced state dimension $n=40$.  Bottom row: pointwise signed error in each ROM solution.  Here, the RBF solution is visually closest to the FOM solution, while the MO solution exhibits the largest discrepancy.}}\label{fig:solution_wave}
\end{figure}

\Cref{fig:errors_wave} shows the median relative $~M_W$-weighted ROM errors (solid lines) and the median projection errors (dashed lines) as functions of the reduced dimension $n$; shaded bands denote the interquartile range of the ROM errors. Similar to the case of the heat equation, the ROM errors remain close to the projection errors, indicating that the reduced Hamiltonian system \eqref{eq:waveROM} is a good approximation to the full-order dynamics.
Observe that the two HOSVD-based methods again exhibit identical performance on the training data (left column), with median error leveling off near $0.5$ resp. \rev[anthony]{$5\%$}{3\%} in position resp. momentum by $n=40$.
In contrast, the monolithic SVD approach converges more slowly, reaching around $4\%$ at $n=40$ but continuing to decrease steadily and gradually. Though all three methods reach comparable accuracy at $n=120$ (not pictured), it is evident that the tensorial strategy is much more parameter-efficient at moderate basis sizes, as expected from its nonlinear calibration to parametric information. 

The behavior of each method on the testing set (right column of \Cref{fig:errors_wave}) is also interesting.  The monolithic ROM displays a trend similar to that observed during training, with errors that steadily decrease as modes are added, and performance similar to the reproductive case.  In contrast, the tensorial ROMs constructed with RBF and MO interpolation strategies produce very different results.  The GRBF-interpolated ROM behaves similarly to its training performance and shows the fastest error reduction of all methods, falling below \rev[anthony]{$1\%$}{0.6\%} in displacement (roughly \rev[anthony]{$8\%$}{4.2\%} in momentum) by $n=40$.
The MO ROM, however, exhibits a pronounced train–test discrepancy: its median error stagnates over the range $20 \le n \le 80$ and quickly becomes larger than that of the monolithic ROM. This behavior indicates substantial interpolation error incurred by the MO approach in this case and suggests that a larger set of training parameters may be required to recover optimal performance. In contrast, the interpolation error incurred by the RBF scheme is much smaller, perhaps due to its increased regularity and suitability in sparse-data settings.  While the monolithic ROM eventually reaches the performance of the RBF version at $n=120$ (not pictured), the proposed tensorial ROM with RBF interpolation produces errors remaining markedly lower than the other two methods across the full range of $n$ before this point.

\Cref{fig:solution_wave} compares the FOM and ROM displacement fields at the terminal time $\tau = 8\pi$ for $n=40$ and a representative test parameter.  The RBF ROM reconstruction is reasonably close to the FOM solution, with a relative error of \rev[anthony]{$3.5\%$}{2.7\%} and a decent match to qualitative features such as the contour in the lower-left corner.  The monolithic ROM captures the dominant wave structure but exhibits scattered patches of localized error, particularly near the forcing peak, yielding an overall relative error of $11.9\%$.  The solution to the MO ROM shows the largest discrepancy, with large pointwise errors spread broadly across the domain and a relative error of \rev[anthony]{$37.3\%$}{16.9\%}, consistent with its elevated test-set errors observed in \Cref{fig:errors_wave} at moderate $n$.

\subsection{Maxwell's equations}\label{subsec:maxwell}

The final example considered here is a challenging 3D test case from electromagnetics. 
Consider the following initial boundary value problem for Maxwell's equations:
\begin{equation}\label{eq:maxwell}
    \left\{\begin{aligned}
        &\frac{\partial ~ E}{\partial t} = \nabla \times ~ B - ~ J(~ \mu; ~ x, t), \qquad t \in [0, \tau], \quad &&~ x \in \Omega, \\
        &\frac{\partial ~ B}{\partial t} = -\nabla \times ~ E, \qquad\qquad\qquad\,\,\,\, t \in [0, \tau], \quad &&~ x \in \Omega, \\
        & ~ E(~ x, t) \times ~ n = ~ 0, \hspace{2.3cm} t \in [0, \tau], \quad &&~ x \in \partial\Omega, \\
        & ~ E(~ x, 0) = ~ 0, \quad ~ B(~ x, 0) = ~ 0, \quad &&~ x \in \Omega,
    \end{aligned}\right.
\end{equation}
posed on the cubic domain $\Omega = [0,T]^3$ with $T=2$ and perfect electric conductor (PEC) boundary conditions, where $~n$ denotes the outward unit normal to $\partial\Omega$.  Here, $~ E$ and $~ B$ are the electric and magnetic fields, and $~ J$ is a prescribed, time-varying, parametric current source given by\footnote{Note that this $~J\in\mathbb{R}^3$ is not the same object as the canonical symplectic matrix from before.}
\begin{align}\label{eq:current}
    ~ J(~ \mu;~ x, t) = s(t)\exp\!\left({-\frac{\|~ x - ~ \mu\|^2}{2\sigma^2}}\right)~ d, \qquad s(t) = \tfrac{1}{2}\!\left(1 - \cos\frac{2\pi t}{T}\right),
\end{align}
where $~ \mu \in \mathbb R^3$ is the spatial center of the source, $\sigma = 0.1$ is its spatial width, and $~d = [1,1,1]^\intercal \in\mathbb{R}^3$ is the polarization direction.  The source thus varies parametrically only in its spatial center $~ \mu$.

The numerical solution of Maxwell's equations finds many uses in modeling antennas, waveguides, and as a component in multiphysics simulations required for prototype fusion devices. The problem posed above is designed to model the response of a closed system to an electromagnetic insult. The injected current stimulates modes at relatively low spatial/temporal frequencies, creating a resonance in the cavity (in this case, a cube). Modes are allowed to ring within the cavity without a change in energy after the cessation of the injected current.

The FOM for \cref{eq:maxwell} is constructed using a mixed-form compatible finite-element method on an unstructured tetrahedral mesh $\Omega_h$ (the formulation is motivated by~\cite{bettencourt2021empire}).  The electric field $~ E_h$ is discretized in the $H(\mathrm{curl})$-conforming N\'{e}d\'{e}lec space
\begin{align*}
    ~ W_h := \{~ w_h \in H(\mathrm{curl},\Omega) : ~ w_h|_K \in \mathcal{N}_k(K),\ \forall K \in \Omega_h;\ ~ w_h \times ~ n|_{\partial\Omega} = ~ 0\},
\end{align*}
and the magnetic field $~ B_h$ in the $H(\mathrm{div})$-conforming Raviart--Thomas space
\begin{align*}
    ~ V_h := \{~ v_h \in H(\mathrm{div},\Omega) : ~ v_h|_K \in \mathrm{RT}_{k-1}(K),\ \forall K \in \Omega_h\}.
\end{align*}
The semi-discrete weak form for each $(t,~\mu)$ reads: find $(~ E_h, ~ B_h) \in ~ W_h \times ~ V_h$ such that
\begin{subequations}\label{eq:weakmaxwell}
\begin{align}
    (\dot{~ E}_h,\, ~ w_h)_{\Omega_h} &= \left(~ B_h,\, \nabla \times ~ w_h\right)_{\Omega_h} - \left(~ J(~\mu),\, ~ w_h\right)_{\Omega_h} && \forall\, ~ w_h \in ~ W_h, \label{eq:weakmaxwell_a}\\
    (\dot{~ B}_h,\, ~ v_h)_{\Omega_h} &= -\left(\nabla \times ~ E_h,\, ~ v_h\right)_{\Omega_h} && \forall\, ~ v_h \in ~ V_h.
\end{align}
\end{subequations}
In \cref{eq:weakmaxwell_a}, the curl appears on the test function, and this is equivalent to the strong form when $~B_h$ is regular enough due to strong enforcement of the PEC condition in $~W_h$.
Denoting the basis functions of $~ W_h$ and $~ V_h$ by $\{~\phi_i\}_{i=1}^{N_E}$ and $\{~\psi_i\}_{i=1}^{N_B}$, respectively, expanding $~E_h$ and $~B_h$ in these bases, and collecting the corresponding degrees of freedom in vectors $~e \in \mathbb{R}^{N_E}$ and $~b \in \mathbb{R}^{N_B}$, \cref{eq:weakmaxwell} is equivalently written in the matrix-vector form
\begin{subequations}\label{eq:maxwellFOM}
\begin{align}
    ~M_E\, \dot{~e}(t) &= ~\delta^\intercal ~M_B\, ~b(t) - ~j(~\mu; t), \\
    ~M_B\, \dot{~b}(t) &= -~M_B ~\delta\, ~e(t), \label{eq:maxwellFOM_b}
\end{align}
\end{subequations}
where $~M_E \in \mathbb{R}^{N_E \times N_E}$ and $~M_B \in \mathbb{R}^{N_B \times N_B}$ are the mass matrices associated with $~ W_h$ and $~ V_h$, respectively.  Here, $~\delta \in \mathbb{R}^{N_B \times N_E}$ is the combinatorial curl and $~j(~\mu; t) \in \mathbb{R}^{N_E}$ is the current load vector with entries:
\begin{align*}
[~M_B ~\delta]_{ij} = (\nabla \times ~\phi_j,\, ~\psi_i)_{\Omega_h}, 
\qquad
[~j]_i = (~ J(~\mu;\cdot,t),\, ~\phi_i)_{\Omega_h}
\end{align*}
Observe that both $~\delta^\intercal ~M_B$ and its negative (Euclidean) adjoint $-~M_B ~\delta$ appear in the equations of motion, reflecting the (noncanonical) Hamiltonian structure of \cref{eq:maxwell} in the absence of the current source\footnote{Euclidean skew-symmetry is a consequence of the mass matrices which appear on the left-hand side, see e.g. \cite{vijaywargiya2025tensor}.}.  The FOM \cref{eq:maxwellFOM} is integrated in time using the velocity Verlet (leap-frog) scheme:
\begin{align*}
    ~e_{n+1/2} &= ~e_n + \tfrac{\Delta t}{2}\, ~M_E^{-1}\!\left(~\delta^\intercal ~M_B\,~b_n - ~j(~\mu;\, t_n)\right), \\
    ~b_{n+1} &= ~b_n - \Delta t\, ~\delta\, ~e_{n+1/2}, \\
    ~e_{n+1} &= ~e_{n+1/2} + \tfrac{\Delta t}{2}\, ~M_E^{-1}\!\left(~\delta^\intercal ~M_B\,~b_{n+1} - ~j(~\mu;\, t_{n+1})\right).
\end{align*}
Notice that the mass matrix $~M_B$ does not appear explicitly in the update equation for $~b$.

To design the Galerkin ROM, separate bases $~U_E \in \mathbb{R}^{N_E \times n}$ and $~U_B \in \mathbb{R}^{N_B \times n}$ are used to approximate
the electric and magnetic field degrees of freedom. Substituting the approximations $\tilde{~e} = ~U_E\hat{~e} \approx ~e$ and $\tilde{~b} = ~U_B\hat{~b} \approx ~b$ into \cref{eq:maxwellFOM} and testing the two equations against $~U_E$ and $~U_B$, respectively, gives the Galerkin ROM
\begin{subequations}\label{eq:maxwellROMfull}
\begin{align}
    ~U_E^\intercal ~M_E ~U_E\, \dot{\hat{~e}}(t) &= ~U_E^\intercal ~\delta^\intercal ~M_B ~U_B\, \hat{~b}(t) - ~U_E^\intercal ~j(~\mu; t), \\
    ~U_B^\intercal ~M_B ~U_B\, \dot{\hat{~b}}(t) &= -~U_B^\intercal ~M_B ~\delta ~U_E\, \hat{~e}(t).
\end{align}
\end{subequations}
When $~U_E$ is $~M_E$-orthonormal and $~U_B$ is $~M_B$-orthonormal, the operators on the left-hand side become the identity, simplifying \cref{eq:maxwellROMfull} to the reduced-order system
\begin{align}\label{eq:maxwellROM}
    \dot{\hat{~e}}(t) = \hat{~\delta}_1^\intercal\hat{~b}(t) - \hat{~j}(~\mu; t), \qquad
    \dot{\hat{~b}}(t) = -\hat{~\delta}_1\hat{~e}(t),
\end{align}
where $\hat{~\delta}_1 = ~U_B^\intercal ~M_B ~\delta ~U_E$ and $\hat{~j} = ~U_E^\intercal ~j$ are the reduced curl and current, respectively. Observe that the skew-symmetric coupling of the reduced operators in  \cref{eq:maxwellROM} mirrors that of \cref{eq:maxwell}.  Therefore, the noncanonical Hamiltonian structure (in the absence of the current source) is preserved at the reduced level, i.e., conservation of  $H(~e,~b) = |~e|^2_{~M_E} + |~b|^2_{~M_B}$ at the full-order level implies conservation of the pullback $\hat{H} = H \circ {\rm blkdiag}(~U_E,~U_B)$ at the reduced level (c.f. \cite{GRUBER2023116334}).  This is guaranteed for the ROM \cref{eq:maxwellROM} discretely in time though the use of the same velocity Verlet integration scheme.

It is crucial to note that discrete compatibility is not necessarily preserved at the level of the ROM. Since the bases $~U_E$ and $~U_B$ are constructed independently, it is generally the case that
$$
\mathrm{im}(~\delta^\intercal ~M_B ~U_B) \nsubseteq \mathrm{im}(~U_E), \qquad \mathrm{im}(~M_B ~\delta ~U_E) \nsubseteq \mathrm{im}(~U_B).
$$
Said differently, the image ${\rm im}(~\delta~U_E) \nsubseteq {\rm im}(~U_B)$ of the curl applied to the electric field basis may not lie in the span of the magnetic field basis.  Hence, the projected operators may incur large closure errors even as the sizes of two bases are increased, potentially leading to significant inaccuracies in the ROMs. To restore approximate closure and mitigate this issue, it is therefore desirable to enrich each of the bases $~U_E,~U_B$ with the discrete curl of the other. Orthonormalizing with respect to the corresponding mass matrices yields the enriched bases 
$$~ U_E^{\text{enr}} = \text{orth}_{~ M_E}\left([~ U_E, \delta^\intercal ~ M_B ~ U_B]\right), \quad ~ U_B^{\text{enr}} = \text{orth}_{~ M_B}\left([~ U_B, \delta ~ U_E]\right),$$
where $[\cdot,\cdot]$ denotes horizontal concatenation. \rev[ADG]{}{For the monolithic ROM, this enrichment is a purely offline operation, since the bases $~U_E,~U_B$ are fixed once and for all. This is no longer the case for the tensorial ROM: the bases $~U_{\mu,E} = ~W_E~U_{c,E}$ and $~U_{\mu,B} = ~W_B~U_{c,B}$ are only available at query time, so the enrichment must be carried out online, and doing so na\"ively would require full-order work at every parameter. Fortunately, because the discrete curl is linear and parameter-independent, the enrichment can be pre-factored offline using the spatial factors $~W_E,~W_B$ coming from the HOSVD,}
\begin{equation}\label{eq:augfactor}
\begin{aligned}
    \left[~U_{\mu,E},\ ~\delta^\intercal ~M_B ~U_{\mu,B}\right]
      &= ~W_E^{\rm aug} ~G_{\mu,E},
      &\quad ~W_E^{\rm aug} &:= \left[~W_E,\ ~\delta^\intercal ~M_B ~W_B\right],
      &\quad ~G_{\mu,E} &:= {\rm blkdiag}\left(~U_{c,E},\, ~U_{c,B}\right), \\
    \left[~U_{\mu,B},\ ~\delta ~U_{\mu,E}\right]
      &= ~W_B^{\rm aug} ~G_{\mu,B},
      &\quad ~W_B^{\rm aug} &:= \left[~W_B,\ ~\delta ~W_E\right],
      &\quad ~G_{\mu,B} &:= {\rm blkdiag}\left(~U_{c,B},\, ~U_{c,E}\right).
\end{aligned}
\end{equation}
\rev[ADG]{}{The curl is thus applied only to the fixed spatial factors $~W_E$ and $~W_B$, so that the augmented factors $~W_E^{\rm aug} \in \mathbb{R}^{N_E \times 2\bar N}$ and $~W_B^{\rm aug} \in \mathbb{R}^{N_B \times 2\bar N}$ are formed only once offline. Note that $~U_{\mu,E}$ and $~U_{\mu,B}$ are $~M$-orthonormal by construction but the concatenations on the left of \cref{eq:augfactor} are not.  While these must be orthonormalized before use, this too requires no full-order work, since}
\begin{align}\label{eq:latentchol}
    \left(~W_E^{\rm aug} ~G_{\mu,E}\right)^\intercal ~M_E \left(~W_E^{\rm aug} ~G_{\mu,E}\right)
    = ~G_{\mu,E}^\intercal \left(~W_E^{\rm aug,\intercal} ~M_E ~W_E^{\rm aug}\right) ~G_{\mu,E}
    =: ~R_{\mu,E}^\intercal ~R_{\mu,E}.
\end{align}
\rev[ADG]{}{
Therefore, precomputing the matrix $~W_E^{\rm aug,\intercal} ~M_E ~W_E^{\rm aug}$ offline yields an object at size $2\bar N \times 2\bar N$, leaving a small Cholesky factorization of size $2n \times 2n$ to be performed online. The case of the magnetic field is handled identically with $~W_B^{\rm aug}$, $~G_{\mu,B}$ and $~M_B$ in place of their electric field counterparts in \eqref{eq:latentchol}. Denoting the resulting latent Cholesky factor by $~R_{\mu,B}$, the $~M$-orthonormal enriched bases used in the tensorial ROM become}
\begin{align*}
    ~U_{\mu,E}^{\rm enr} = ~W_E^{\rm aug} ~G_{\mu,E} ~R_{\mu,E}^{-1},
    \qquad
    ~U_{\mu,B}^{\rm enr} = ~W_B^{\rm aug} ~G_{\mu,B} ~R_{\mu,B}^{-1},
\end{align*}
The rest of the discussion will assume that both bases are enriched, with superscripts dropped for simplicity. 

A final concern when constructing the ROM \cref{eq:maxwellROM} is the costly evaluation of the projected load vector $\hat{~j}(~\mu;t) = ~U_E^\intercal ~j(~\mu;t)$.  For each $~\mu$, this requires assembling the full load vector $~j(~\mu;t) \in \mathbb{R}^{N_E}$ at each time step. Owing to the separable structure of the current source \cref{eq:current} in its time and space coordinates, the load vector admits the decomposition $~j(~\mu;t) = s(t)\,~j_{\rm sp}(~\mu)$, where $~j_{\rm sp}(~\mu) \in \mathbb{R}^{N_E}$ requires only a single assembly per parameter query.  However, this assembly still requires integrating over all $N_E$ DOFs, exacerbating online costs and reducing ROM speed-up.  Thankfully, this can be mitigated by employing Q-DEIM hyper-reduction \cite{drmac2016new} to approximate $~j_{\rm sp}(~\mu)$ using evaluations at only $m \ll N_E$ selected DOFs.
In the \emph{offline} stage of Q-DEIM, $~j_{\rm sp}$ is assembled for each training parameter, forming the snapshot matrix $~J_{\rm snap} \in \mathbb{R}^{N_E \times N_{\rm train}}$.  A rank-$m$ truncated SVD then yields a low-dimensional subspace basis $~U_J \in \mathbb{R}^{N_E \times m}$ approximating the range of $~J_{\rm snap}$.  Then, the pivoted QR decomposition applied to the basis $~U_J^\intercal$ selects $m$ interpolation (row) indices and defines the selection matrix $~P = [~e_{i_1}\mid\cdots\mid~e_{i_m}] \in \mathbb{R}^{N_E \times m}$, where $~e_i$ denotes the $i$-th standard basis vector in $\mathbb{R}^{N_E}$.  It follows that multiplication by $~P^\intercal$ extracts the relevant $m$ entries of any vector in $\mathbb{R}^{N_E}$.  With this, the Q-DEIM approximation to the full spatial component of the load vector is
\begin{align*}
    ~j_{\rm sp,hr} := (~\mu) ~U_J(~P^\intercal ~U_J)^{-1}~P^\intercal ~j_{\rm sp}(~\mu) \approx ~j_{\rm sp}(~\mu),
\end{align*}
which requires only the assembly of $m$ entries of $~j_{\rm sp}$, \rev[anthony]{at each time step}{once per parameter query}.  Substituting this approximation into the reduced forcing $\hat{~j} = ~U_E^\intercal ~j$ gives the hyper-reduced forcing
\begin{align*}
    \hat{~j}_{\rm hr}(~\mu; t) = s(t)~C_J~P^\intercal ~j_{\rm sp}(~\mu) \approx  \hat{~j}(~\mu; t), \qquad ~C_J := ~U_E^\intercal ~U_J(~P^\intercal ~U_J)^{-1} \in \mathbb{R}^{n \times m},
\end{align*}
where $~C_J$ is pre-computed once \rev[ADG]{offline}{} for each query parameter.  \rev[ADG]{In the \emph{online} stage,}{} Evaluating $~P^\intercal ~j_{\rm sp}(~\mu)$ requires integrating only over the mesh elements that contain at least one interpolation index, so that the approximate reduced forcing is obtained at a cost that scales only with $m$ and $n$ and independently of $N_E$.  \rev[anthony]{}{Pseudocode is provided in \Cref{alg:qdeim}.}  This hyper-reduced forcing $\hat{~j}_{\rm hr}$ is used in all present simulations involving the ROM \eqref{eq:maxwellROM}.

\begin{algorithm}
\caption{Q-DEIM Hyper-Reduction of the Reduced Load Vector}
\label{alg:qdeim}
\begin{algorithmic}[1]
\State \textbf{Input:} Spatial load routine $~j_{\rm sp}(\cdot)\in\mathbb{R}^{N_E}$, temporal profile $s(t)$, training parameters $\{~\mu_i\}_{i=1}^{N_{\rm train}}$, $~M_E$-orthonormal basis $~U_E\in\mathbb{R}^{N_E\times r}$, hyper-reduction rank $m\ll N_E$.
\Statex \vspace{0.em} \hrulefill \ \textbf{Offline Stage} \ \hrulefill \vspace{0.2em}
\State Assemble the snapshot matrix $~J_{\rm snap}\in\mathbb{R}^{N_E\times N_{\rm train}}$ with columns $[~J_{\rm snap}]_{:,i} = ~j_{\rm sp}(~\mu_i)$.
\State Compute the rank-$m$ truncated SVD of $~J_{\rm snap}$, retaining the leading $m$ left singular vectors $~U_J\in\mathbb{R}^{N_E\times m}$.
\State Apply pivoted QR to $~U_J^\intercal$ to select $m$ interpolation indices, defining the selection matrix $~P = [~e_{i_1}\mid\cdots\mid~e_{i_m}]\in\mathbb{R}^{N_E\times m}$.
\Statex \vspace{0.2em} \hrulefill \ \textbf{Online Stage} \ \hrulefill \vspace{0.2em}
\State Precompute $~C_J = ~U_E^\intercal ~U_J(~P^\intercal ~U_J)^{-1}\in\mathbb{R}^{r\times m}$ at test parameter $~\mu\in\mathbb{R}^p$.
\State Assemble the $m$ selected entries $~P^\intercal ~j_{\rm sp}(~\mu)$, integrating only over mesh elements containing an interpolation index.
\State \textbf{Output:} Hyper-reduced forcing $\hat{~j}_{\rm hr}(~\mu;t) = s(t)~C_J~P^\intercal ~j_{\rm sp}(~\mu)\approx ~U_E^\intercal ~j(~\mu;t)$, at cost independent of $N_E$.
\end{algorithmic}
\end{algorithm}

\begin{figure}[htbp!]
    \centering
    \includegraphics[width=\linewidth]{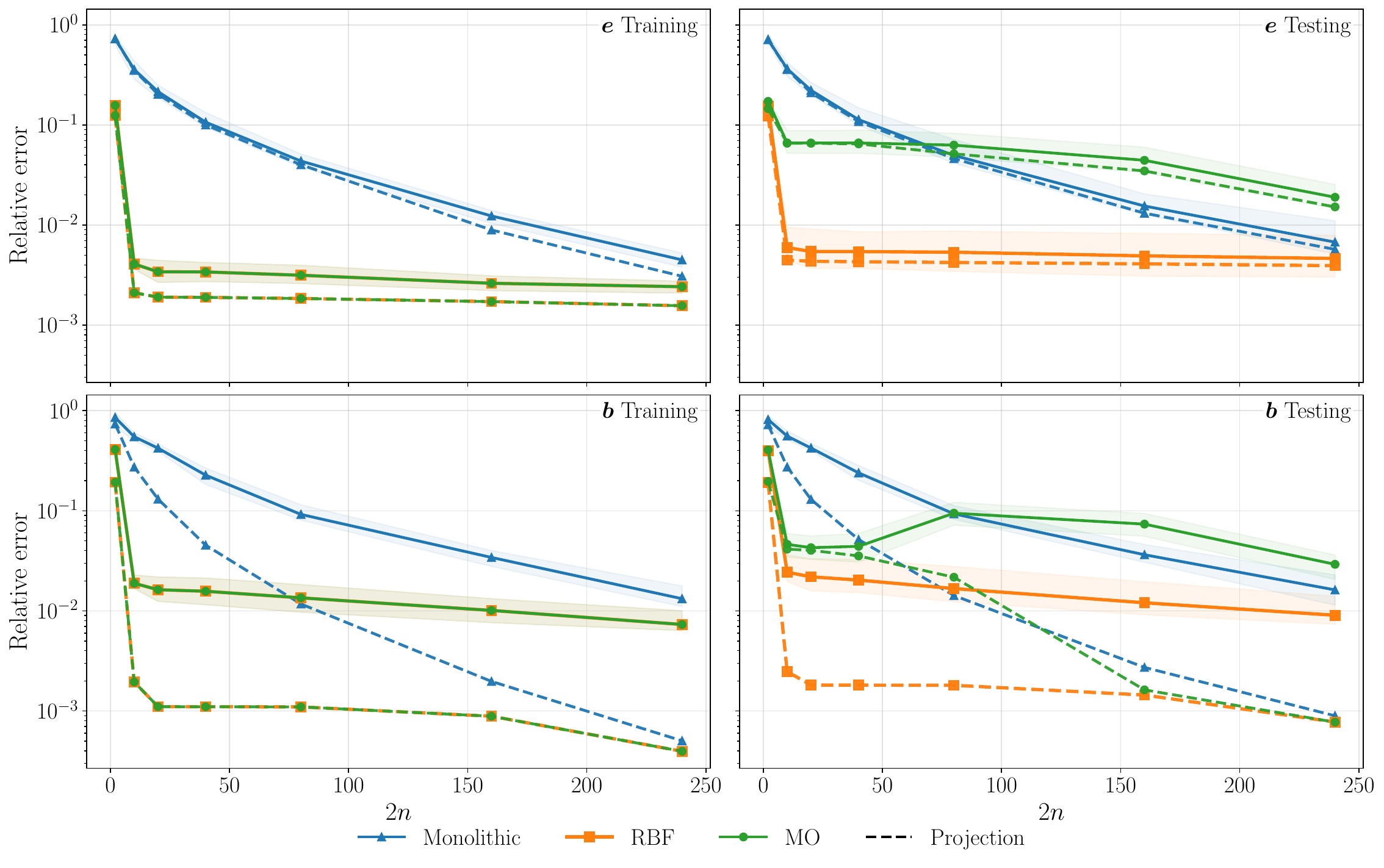}
    \caption{\rev[anthony]{}{Relative $~M_E$-weighted $L^2$ error in the ROM solutions to the Maxwell system as a function of reduced basis dimension $2n$ (curl-enriched), considered over all training (left) and testing (right) parameters.  Top row: electric field error; bottom row: magnetic field error.  For each method, solid lines denote the median ROM error and dashed lines denote the median projection error.  Shaded bands indicate the interquartile range of the ROM error. }}\label{fig:errors_maxwell}
\end{figure}

\subsubsection{Experimental details}

Choosing an unstructured tetrahedral mesh for the domain and the polynomial degree $k=1$ for the finite element spaces $~W_h$ and $~V_h$ yields $N_E = 49320$ degrees of freedom for the electric field $~{E}_h$ and $N_B = 82350$ for the magnetic field $~{B}_h$. The FOM \cref{eq:maxwellFOM} is integrated over $[0, 2.5]$ with step-size $\Delta t = 2.5/120$, producing $T = 121$ snapshots per parameter instance.  The source center $~\mu$ is sampled uniformly from $[0.5, 1.5]^3$ for a total of $P = 200$ samples split into 160 training and 40 testing instances.

\Cref{alg:M_tucker_basis} is applied separately to the electric and magnetic snapshot tensors $~X_E \in \mathbb{R}^{49320 \times 121 \times 160}$ and $~X_B \in \mathbb{R}^{82350 \times 121 \times 160}$, using $~M_E$- and $~M_B$-weighted inner products, respectively.  The choice of Tucker ranks $(\bar N, \bar T, \bar P) = (150, 120, 150)$ yields relative representation errors of $2.11 \times 10^{-3}$ for $~{E}_h$ and $1.16 \times 10^{-3}$ for $~{B}_h$.  For a given test parameter $~\mu$, rank-$n$ bases $~U_{\mu,E}$ and $~U_{\mu,B}$ are extracted using the RBF interpolant with \rev[anthony]{$\varepsilon = 1.8$}{$\varepsilon = 2$} and the MO interpolant with \rev[anthony]{15}{$7$} nearest neighbors; \rev[anthony]{}{these hyper-parameters are chosen using the LOOCV described in \Cref{app:interpolation}}. The monolithic baseline uses the leading $~M_E$- resp. $~M_B$-orthonormal left singular vectors of the matrix unfoldings $~X_{E,(1)}$ resp. $~X_{B,(1)}$. The curl enrichment described above is then applied to all three sets of reduced bases.
For Q-DEIM hyper-reduction of the load vector, rank \rev[anthony]{$m = 80$}{$m=160$} is used, capturing \rev[anthony]{$99.999\%$}{100\%} of the energy in $~J_{\rm snap}$; the resulting interpolation indices correspond to \rev[anthony]{$480$}{960} mesh elements (\rev[anthony]{$1.2\%$}{2.4\%} of the total 40500 tetrahedra).  \rev[anthony]{}{Since $m$ is set to the maximum possible value}, the median Q-DEIM reconstruction error on the training set is \rev[anthony]{$4.85 \times 10^{-3}$}{$9.38\times 10^{-15}$}.

\Cref{fig:errors_maxwell} reports the median relative $~M_E$- and $~M_B$-weighted ROM errors (solid lines) and projection errors (dashed lines) as functions of the reduced dimension $2n$ \rev[anthony]{}{(curl-enriched)}, with shaded bands indicating the interquartile range of the ROM error. On the training set, both tensorial ROMs are identical and substantially outperform the monolithic SVD ROM for the electric and magnetic fields across the full range of $2n$, with their median errors dropping rapidly and plateauing near \rev[anthony]{$0.75\%$}{0.3\%} for $2n \ge 20$ in the case of the electric field and continually decreasing to near $1\%$ in the case of the magnetic field.  In contrast, the monolithic ROM improves slowly, with its error remaining more than an order of magnitude higher at $2n = 20$ and only approaching the errors of the tensorial approaches near $2n = 240$. The ROM error of each method closely tracks the projection error in the $~e$ case but not in the $~b$ case, an artifact of the compatibility problem between the bases $~U_E$ and $~U_B$ mentioned previously.  
On the testing set, the behavior of the monolithic baseline is similar to its performance on the training data.  The proposed RBF method yields the lowest median error by a wide margin, with testing errors dropping near $0.5\%$ for $~{e}$ and $2\%$ for $~b$ at $2n=20$. 
In contrast, the MO ROM exhibits a pronounced train--test gap: the testing error for $~e$ plateaus around $7\%$ for $20 \le 2n \le 160$ and decreases slowly after, while the testing error for $~b$ is not monotone, peaking around $10\%$ at $2n=80$ and exceeding the error in the monolithic ROM for $2n\geq 80$.  Note that this behavior is independent of reduced-order incompatibility: all methods suffer a gap between their projection and ROM errors in the $~b$ case, but only the MO ROM's performance fluctuates with the addition of modes.  This suggests that the the proposed RBF strategy is also robust in the presence of imperfect ROMs, yielding greatly reduced errors with more predictable performance as modes are added.


An example visualization is provided in \Cref{fig:maxwell_E,fig:maxwell_B}, which display the magnitude along with the three components of the FOM and ROM-computed electric and magnetic fields at the terminal time $\tau = 2.5$    on the midplane slice $z = 0.5$ for a representative test parameter with $2n = 20$.  The RBF tensorial ROM approximations to $~{e}$ and $~{b}$ closely resemble the true fields, with relative errors of \rev[anthony]{$0.88\%$}{0.90\%} and $3.1\%$, respectively.  For this parameter and basis size, the MO tensorial ROM is also accurate, with relative errors of \rev[anthony]{$4.2\%$}{4.3\%} and $3.8\%$, respectively.  The monolithic ROM, in contrast, struggles severely to capture the features of either field at this basis size, with large relative errors of $36.9\%$ and $44.4\%$ in the computed $~e$ and $~b$.  This demonstrates that the solution to the FOM \eqref{eq:maxwellFOM} is highly nonlinear in the parameters $~\mu$, providing further support for the utility of the tensorial ROM in practical cases of interest.

\begin{figure}[htbp!]
    \centering
    \includegraphics[width=0.75\linewidth]{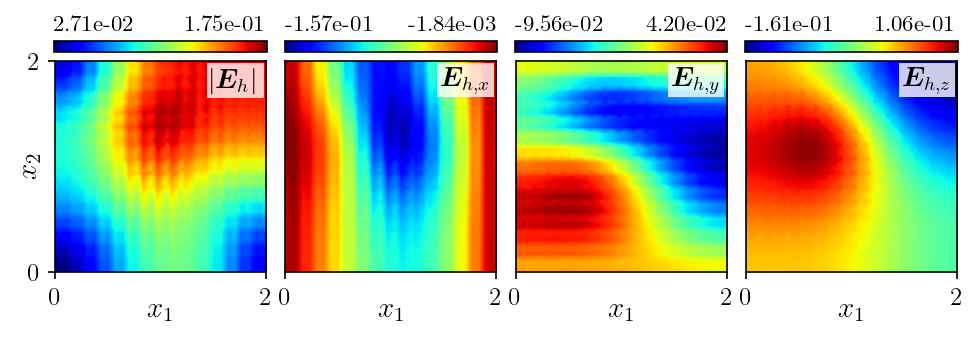}\\
    \makebox[\linewidth][c]{\textbf{(a)} FOM}
    \includegraphics[width=0.75\linewidth]{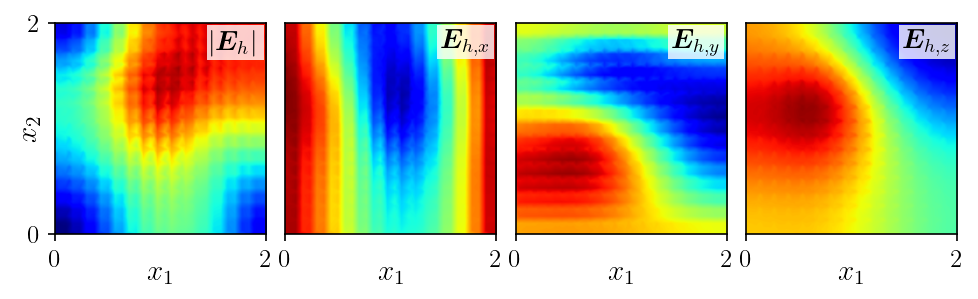}\\
    \makebox[\linewidth][c]{\textbf{(b)} RBF}
    \includegraphics[width=0.75\linewidth]{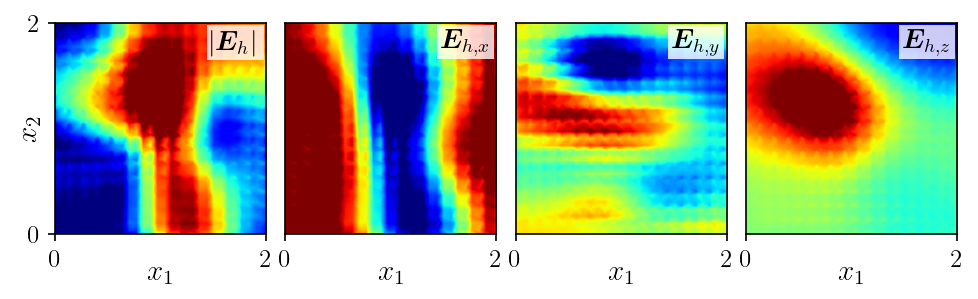}\\
    \makebox[\linewidth][c]{\textbf{(c)} Monolithic}
    \includegraphics[width=0.75\linewidth]{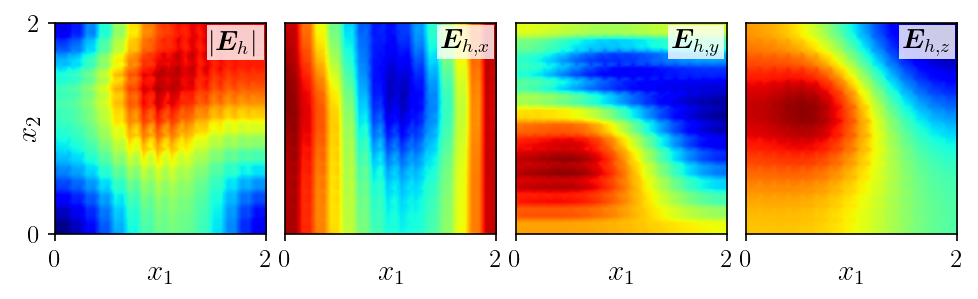}\\
    \makebox[\linewidth][c]{\textbf{(d)} MO}
    \caption{\rev[anthony]{}{Magnitude and components of $~{E_h}$ at $\tau=2.5$ on the slice $z=0.5$ obtained from the FOM and ROM solutions to the Maxwell system for a sample test parameter with ROM dimension $2n=20$. The RBF and MO ROMs yield small errors of $0.90\%$ and $4.3\%$, while the monolithic ROM produces an extremely large error of $36.9\%$.}}\label{fig:maxwell_E}
\end{figure}

\begin{figure}[htbp!]
    \centering
    \includegraphics[width=0.75\linewidth]{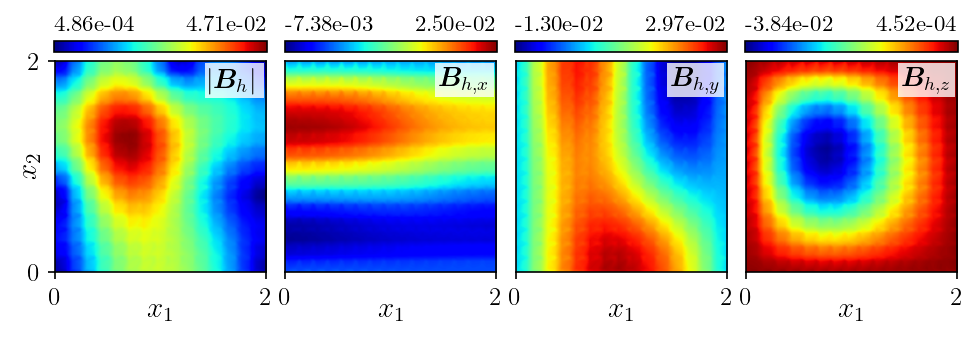}\\
    \makebox[\linewidth][c]{\textbf{(a)} FOM}
    \includegraphics[width=0.75\linewidth]{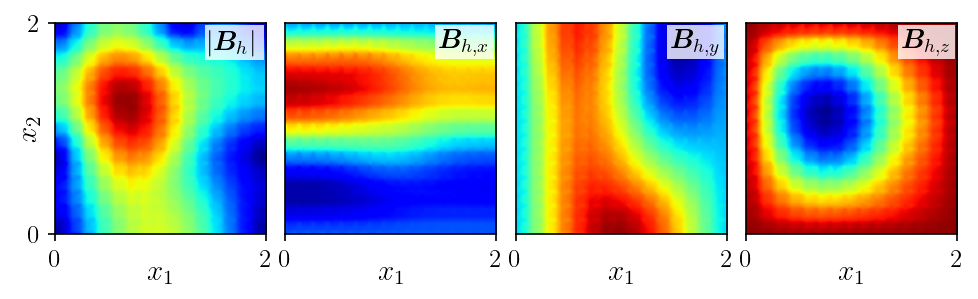}\\
    \makebox[\linewidth][c]{\textbf{(b)} RBF}
    \includegraphics[width=0.75\linewidth]{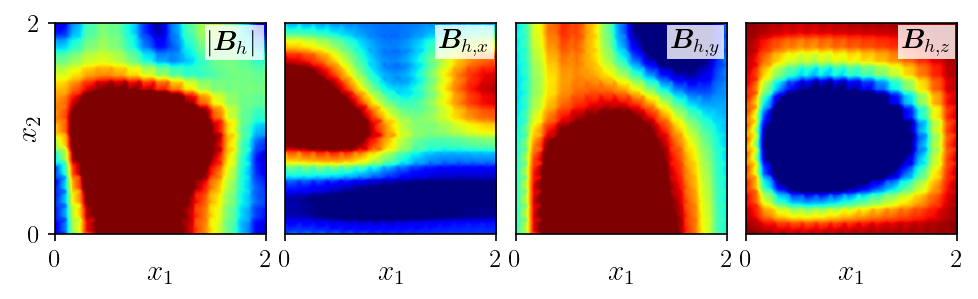}\\
    \makebox[\linewidth][c]{\textbf{(c)} Monolithic}
    \includegraphics[width=0.75\linewidth]{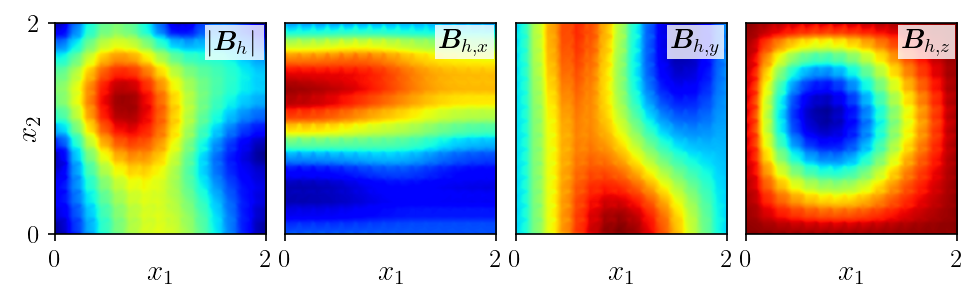}\\
    \makebox[\linewidth][c]{\textbf{(d)} MO}
    \caption{\rev[anthony]{}{Magnitude and components of $~{B_h}$ at $\tau=2.5$ on the slice $z=0.5$ obtained from the FOM and ROM solutions to the Maxwell system for the same test parameter as \Cref{fig:maxwell_E} with ROM dimension $2n=20$. The RBF and MO ROMs yield small errors of $3.1\%$ and $3.8\%$, while the monolithic ROM produces an extremely large error of $44.4\%$.}}\label{fig:maxwell_B}
\end{figure}

\subsubsection{Tensorial vs Monolithic cost comparison}

\rev[ADG]{}{
It is clear from \Cref{tab:time-complexity} and the preceeding discussion that the tensorial ROM approaches incur additional online costs that are absent in the monolithic case. However, \Cref{fig:errors_maxwell} shows that the proposed tensorial ROM achieves comparable accuracy to the monolithic ROM at much smaller basis dimensions. This raises a natural question: does this reduction in basis dimension offset the additional online overhead in terms of wall-clock time? The experiment that follows addresses this question by comparing the online costs of both methods.}

\rev[ADG]{}{
This study retains the setup of the previous experiment but recomputes the Tucker decomposition at ranks $(\bar N, \bar T, \bar P) = (150, 20, 150)$, truncating the temporal rank from $120$ to $20$. This change increases the relative representation error only slightly: from $2.11\times 10^{-3}$ to $2.13\times 10^{-3}$ for $~E_h$ and without changing $~B_h$. On the other hand, it noticeably reduces the cost of computing the SVD of $~C_\mu$ online, which scales quadratically with $\bar T$.  Both the monolithic and tensorial ROMs employ hyperreduction, reducing the cost of projecting the spatial load from $\mathcal O(Nn)$ to $\mathcal O(mn)$ where $m$ is the Q-DEIM rank. For comparison, the performance of the tensorial ROM without hyperreduction is also presented. 
To facilitate a fair comparison, errors between the tensorial ROM and monolithic baseline are roughly equalized: the monolithic baseline has basis dimensions $2n=260$, yielding a median ROM prediction error of $5.83\times 10^{-3}$ on $~E_h$. Note that explicit construction of the bases $~U_{\mu,E}$ and $~U_{\mu,B}$ is never required by the ROM.  Therefore, the cost of constructing them is excluded from consideration.}

\rev[ADG]{}{
All wall-clock times are collected using a Mac Studio with an Apple M4 Max processor, comprising 10 performance and 4 efficiency cores, with 36~GB of unified memory. The results are reported in \Cref{tab:timing}. For a given method, the reported timing is the average required for a single parameter solve computed over 40 test parameter instances. The reported uncertainty is the standard deviation across a total of 5 independent runs. The monolithic baseline takes an average wall-clock time of $1.65\pm 0.02$~ms. At each basis size, \Cref{tab:timing} also reports the median prediction error in $~E_h$ for the tensorial ROM, together with the corresponding speedups relative to the FOM and the monolithic baseline. These speedups are computed as ratios of the average wall-clock times. At a basis size of $2n=10$, the tensorial ROM achieves accuracy comparable to that of the monolithic ROM while being 56\% faster and $26\times$ smaller. As the basis dimension increases, this speedup diminishes because of increasing online overhead, while the median prediction error decreases only slowly. Relative to the FOM, the tensorial ROM achieves a speedup of 7269$\times$.
}

\begin{table}[htbp!]
\centering
\small
\renewcommand{\arraystretch}{1.15}

\resizebox{\textwidth}{!}{%
\begin{tabular}{c cc cccc}
\toprule
& \multicolumn{2}{c}{Without hyperreduction}
& \multicolumn{4}{c}{With hyperreduction} \\
\cmidrule(lr){2-3}\cmidrule(lr){4-7}

\makecell{Basis\\dimension $2n$}
& \makecell{Median\\error in $~E_h$}
& \makecell{Time\\ms}
& \makecell{Median\\error in $~E_h$}
& \makecell{Time\\ms}
& \makecell{Speedup\\against FOM}
& \makecell{Speedup against\\monolithic ROM} \\
\midrule

2
& $1.54\times 10^{-1}$
& $11.12\pm0.21$
& $1.54\times 10^{-1}$
& $0.98\pm0.01$
& $7854\times$
& $1.68\times$ \\

10
& $5.82\times 10^{-3}$
& $10.92\pm0.04$
& $5.82\times 10^{-3}$
& $1.06\pm0.03$
& $7269\times$
& $1.56\times$ \\

20
& $5.76\times 10^{-3}$
& $10.99\pm0.08$
& $5.76\times 10^{-3}$
& $1.07\pm0.02$
& $7217\times$
& $1.55\times$ \\

30
& $5.44\times 10^{-3}$
& $11.04\pm0.09$
& $5.44\times 10^{-3}$
& $1.11\pm0.06$
& $6942\times$
& $1.49\times$ \\

40
& $5.37\times 10^{-3}$
& $11.13\pm0.07$
& $5.37\times 10^{-3}$
& $1.39\pm0.04$
& $5552\times$
& $1.19\times$ \\

\bottomrule
\end{tabular}%
}

\caption{\rev[ADG]{}{Average wall-clock time taken by the tensorial RBF ROM across 40 test parameters at different basis sizes. The reported times and uncertainties are averages and standard deviations across 5 runs. The FOM takes an average wall-clock time of 7691.6~ms per parameter. The monolithic baseline has basis dimensions $2n=260$, a median prediction error of $5.828\times 10^{-3}$, and an average wall-clock time of $1.65\pm 0.02$~ms. Also displayed are the median errors in $~E_h$, reported in the $~M$-norm, along with speedups computed as ratios of the mean wall-clock times. At basis size $2n=10$, the tensorial RBF ROM achieves the same accuracy as the monolithic ROM while being 56\% faster and $26\times$ smaller.}}
\label{tab:timing}
\end{table}
\section{Conclusion}\label{sec:conclusion}
A structure-aware extension of the HOSVD-based tensorial ROM from \cite{mamonov2022interpolatory} has been proposed, analyzed, and evaluated on benchmark problems with gradient flow and Hamiltonian structure.  By integrating  arbitrary mass-orthonormal basis construction and a novel RBF snapshot interpolation strategy into the tensorial ROM, the proposed strategy was shown to be rigorously controllable and structure-preserving while maintaining the benefits of nonlinearity in the reduced basis.  Even among tensorial ROM strategies, it was shown that RBF-based snapshot interpolation yields more performant local reduced bases in various cases of interest, particularly in low-data and/or hyperbolic regimes where monolithic ROMs are known to struggle.  Evaluation of the proposed approach on 2D and 3D benchmarks has demonstrated its promise for improving model reduction over a monolithic approach with low additional overhead, enabling more effective surrogate models for challenging  problems such as the Maxwell case presented here.


\section*{Acknowledgments}

Support for this work was received through the U.S. Department of Energy, Office of Science, Office of Advanced Scientific Computing Research, Mathematical Multifaceted Integrated Capability Centers (MMICCS) program, under the Scalable, Efficient and Accelerated Causal Reasoning Operators, Graphs and Spikes for Earth and Embedded Systems (SEA-CROGS) project. Additional support was received through National Science Foundation under Cooperative Agreement 2421782 and the Simons Foundation award MPS-AI-00010515 (NSF-Simons AI Institute for Cosmic Origins - CosmicAI,
\url{https://www.cosmicai.org}/). Sandia National Laboratories is a multimission laboratory managed and operated by National Technology \& Engineering Solutions of Sandia, LLC, a wholly owned subsidiary of Honeywell International Inc., for the U.S.~Department of Energy's National Nuclear Security Administration under contract DE-NA0003525.
This paper describes objective technical results and analysis. Any subjective views or opinions that might be expressed in the paper do not necessarily represent the views of the U.S.~Department of Energy or the United States Government.

\bibliographystyle{abbrv} 
\bibliography{references}

\appendix

\section{Sensitivity Studies}\label{app:sens}

\rev[anthony]{}{
Here we collect the results of sensitivity studies that complement the results in the body.  The primary goals are to provide empirical evidence for the hyperparameter choices made in \Cref{sec:numerics}, as well as lend confidence to the conclusions drawn from the numerical experiments. 
}

\subsection{Interpolation}\label{app:interpolation}
\rev[anthony]{}{
To quantify the effect of the GRBF shape parameter $\varepsilon$ and the MO neighbor count $K$ on the results, we perform leave-one-out cross-validation on the training set: each training snapshot is removed in turn, an interpolant is constructed from the remaining samples, and the projection error of the held-out snapshot is evaluated. \Cref{tab:heat-interp,tab:wave-interp} reports the median leave-one-out projection error as a function of $\varepsilon$ and $K$ for several basis ranks $n$. Alongside the error, we track the  condition number of the RBF kernel $~K_\varepsilon$ as a function of $\varepsilon$, and the condition number of the MO data matrix $\bar{~Z}~D$ as a function of $K$. Since $~K_\varepsilon$ is assembled once from all training parameters and reused for every query, its condition number is reported directly. Since the MO matrix depends on the query point $~\mu$, we instead report its median condition number across held-out parameters. Error and conditioning are shown together, on separate vertical axes, in \Cref{fig:heat_loocv,fig:wave_loocv,fig:maxwell_loocv}.
}

\begin{figure}[!htb]
    \centering
    \includegraphics[width=\linewidth]{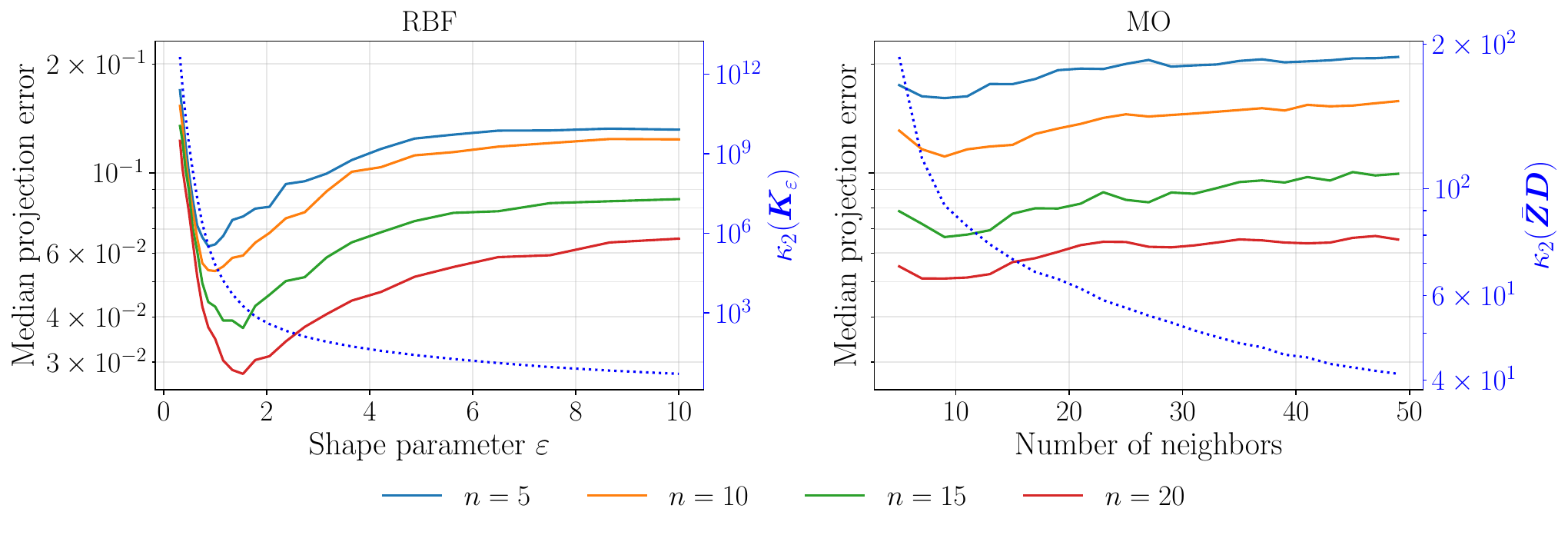}
\caption{\textbf{Heat equation leave-one-out cross-validation.} Median projection error of held-out training snapshots and condition number versus the shape parameter $\varepsilon$ (left) and the number of neighbors $K$ (right). The condition number corresponds to the RBF kernel matrix in the left panel and the MO data matrix (median) in the right panel.}\label{fig:heat_loocv}
\end{figure}

\begin{figure}[!htb]
    \centering
    \includegraphics[width=\textwidth]{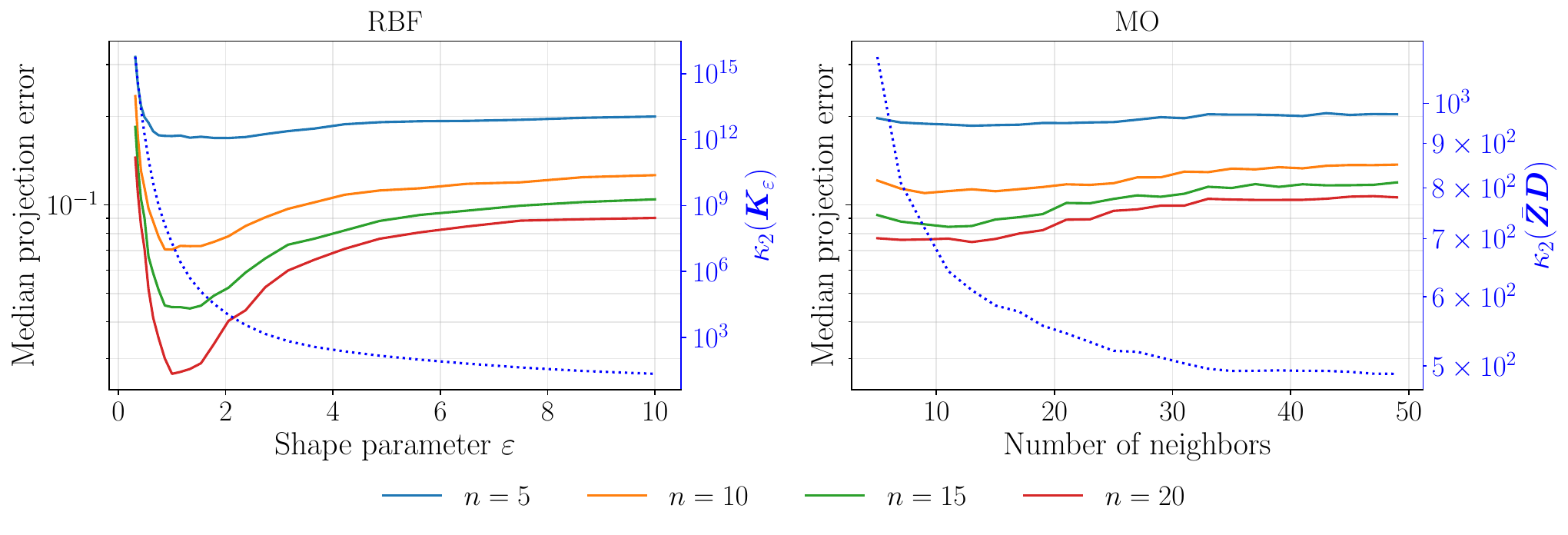}
    \caption{\textbf{Wave equation leave-one-out cross-validation.} Median projection error of held-out training snapshots and condition number versus the shape parameter $\varepsilon$ (left) and the number of neighbors $K$ (right). The condition number corresponds to the RBF kernel matrix in the left panel and the MO data matrix (median) in the right panel.}\label{fig:wave_loocv}
\end{figure}

\begin{figure}[!htb]
    \centering
    \includegraphics[width=\linewidth]{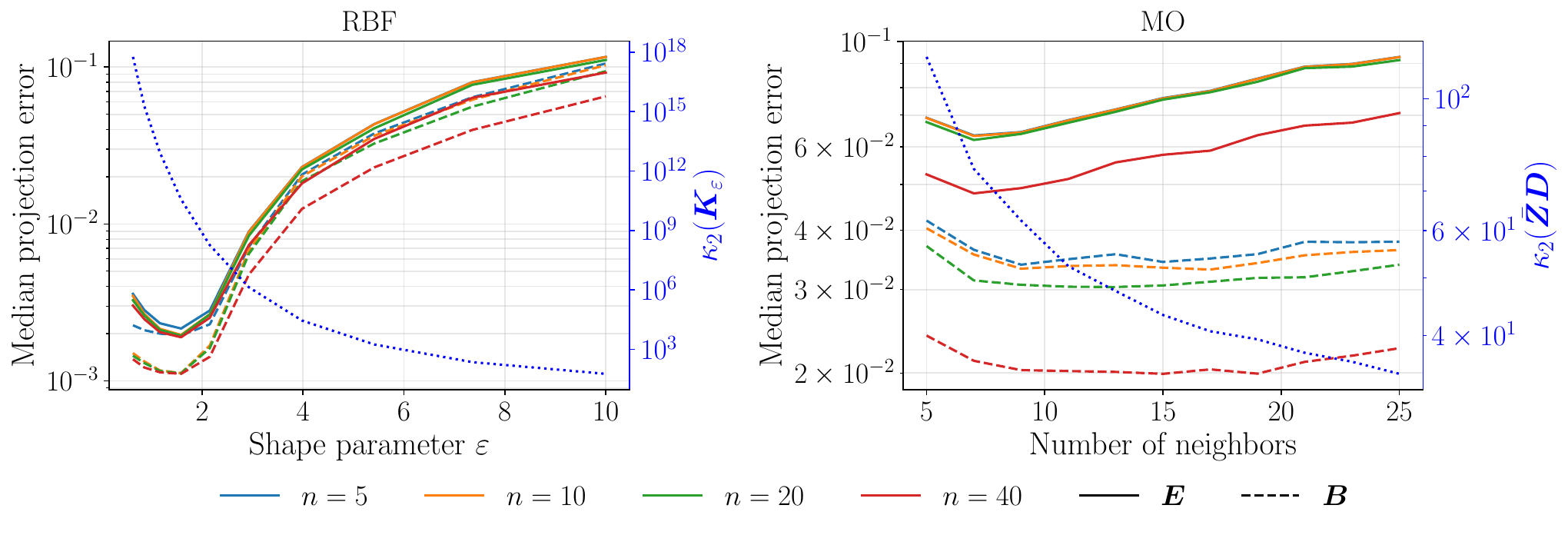}
    \caption{\textbf{Maxwell's equations leave-one-out cross-validation.} Median projection error of held-out training snapshots and condition number versus the shape parameter $\varepsilon$ (left) and the number of neighbors $K$ (right). The condition number corresponds to the RBF kernel matrix in the left panel and the MO data matrix (median) in the right panel.}\label{fig:maxwell_loocv}
\end{figure}

\rev[anthony]{}{
For the MO interpolant, $\bar{~Z}~D$ remains well conditioned as $K$ grows, while the projection error eventually increases, most noticeably in the Maxwell case. Accuracy, not conditioning, is therefore the major constraint in this case, and we recommend choosing the $K$ that minimizes projection error. Based on these figures, the choices $K=10$ for the heat and wave equations and $K=7$ for the Maxwell's equations provide a good balance across all experiments and basis sizes. 
}

\rev[anthony]{}{
Choosing the shape parameter $\varepsilon$ in the GRBF approach is more delicate, since it affects both accuracy and conditioning. Conditioning enters the \textit{a priori} error bound in \Cref{rem:GRBF-proof} through the stability constant $C_e = \sqrt{P}\,\lambda_{\min}(~K_\varepsilon)^{-1}$. As $~K_\varepsilon$ becomes ill-conditioned for small $\varepsilon$, this constant grows and the corresponding term in the bound increases, loosening the worst-case guarantee. This is consistent with \Cref{fig:heat_loocv,fig:wave_loocv,fig:maxwell_loocv}, where the median projection error is U-shaped in $\varepsilon$: it increases as $\varepsilon \to 0$, where the kernel is ill-conditioned, and also increases at large $\varepsilon$, while the conditioning improves monotonically as $\varepsilon$ grows.  We therefore recommend selecting the $\varepsilon$ that minimizes the leave-one-out median relative projection error as the first choice. If the kernel is not acceptably conditioned at that point, a trade-off between the median error and the worst-case error must be accepted. Applying this rule, the shape parameter becomes $\varepsilon=1.2$ for the heat and wave experiments and $\varepsilon=2$ for the Maxwell experiment. This balances accuracy and conditioning uniformly across the various basis sizes $n$.
}

\begin{table}[htb]
\centering
\small
\setlength{\tabcolsep}{4pt}
\renewcommand{\arraystretch}{1.15}
\begin{tabular}{c ccc ccc}
\toprule
& \multicolumn{3}{c}{RBF}
& \multicolumn{3}{c}{MO} \\
\cmidrule(lr){2-4}\cmidrule(lr){5-7}
Training samples
& \makecell{Optimal $\varepsilon$}
& \makecell{$\kappa_2(~ K_\varepsilon)$}
& \makecell{Median error}
& \makecell{Optimal $K$}
& \makecell{$\kappa_2(\bar{~Z}~ D)$}
& \makecell{Median error} \\
\midrule
20  & $1.2$   & $1.07\times 10^{1}$ & $6.708\times 10^{-1}$ & $19$ & $3.21\times 10^{1}$ & $1.189\times 10^{0}$ \\
40  & $1.2$   & $8.39\times 10^{1}$ & $5.292\times 10^{-1}$ & $7$  & $5.83\times 10^{1}$ & $7.228\times 10^{-1}$ \\
80  & $0.616$ & $4.84\times 10^{5}$ & $4.797\times 10^{-1}$ & $7$  & $8.38\times 10^{1}$ & $4.143\times 10^{-1}$ \\
160 & $1.2$   & $1.20\times 10^{4}$ & $2.090\times 10^{-1}$ & $9$  & $9.24\times 10^{1}$ & $2.720\times 10^{-1}$ \\
\bottomrule
\end{tabular}
\caption{\textbf{Heat equation:} Median relative errors in the RBF- and MO-interpolated test snapshots for varying number of training parameter samples. The interpolation is done on the associated uncompressed training snapshots. The optimal shape parameter $\varepsilon$ and number of neighbors $K$ are obtained using leave-one-out cross-validation. The condition number of the the RBF kernel $~K_\varepsilon$ and the MO data matrix $\bar{~Z}~D$ are also reported.}
\label{tab:heat-interp}
\end{table}

\begin{table}[h]
\centering
\small
\setlength{\tabcolsep}{4pt}
\renewcommand{\arraystretch}{1.15}
\resizebox{\textwidth}{!}{
\begin{tabular}{c cccc cccc}
\toprule
& \multicolumn{4}{c}{RBF}
& \multicolumn{4}{c}{MO} \\
\cmidrule(lr){2-5}\cmidrule(lr){6-9}
Training samples
& \makecell{Optimal $\varepsilon$}
& \makecell{$\kappa_2(~K_\varepsilon)$}
& \makecell{Median\\$~Q$ error}
& \makecell{Median\\$~P$ error}
& \makecell{Optimal $k$}
& \makecell{$\kappa_2(\bar{~Z}~D)$}
& \makecell{Median\\$~Q$ error}
& \makecell{Median\\$~P$ error} \\
\midrule
20  & $1.2$ & $5.52\times 10^{1}$ & $3.981\times 10^{-1}$ & $5.660\times 10^{-1}$ & $19$ & $6.08\times 10^{2}$ & $6.397\times 10^{-1}$ & $8.261\times 10^{-1}$ \\
40  & $1.2$ & $4.09\times 10^{3}$ & $1.906\times 10^{-1}$ & $3.083\times 10^{-1}$ & $7$  & $8.14\times 10^{2}$ & $4.213\times 10^{-1}$ & $7.019\times 10^{-1}$ \\
80  & $1.2$ & $1.17\times 10^{4}$ & $1.702\times 10^{-1}$ & $1.414\times 10^{-1}$ & $7$  & $6.97\times 10^{2}$ & $2.095\times 10^{-1}$ & $4.377\times 10^{-1}$ \\
160 & $1.2$ & $1.72\times 10^{6}$ & $1.140\times 10^{-1}$ & $3.800\times 10^{-2}$ & $9$  & $7.19\times 10^{2}$ & $1.274\times 10^{-1}$ & $2.711\times 10^{-1}$ \\
\bottomrule
\end{tabular}}
\caption{\textbf{Wave equation:} Median relative errors in the RBF- and MO-interpolated test snapshots for varying number of training parameter samples. The interpolation is done on the associated uncompressed training snapshots. The optimal shape parameter $\varepsilon$ and number of neighbors $k$ are obtained using leave-one-out cross-validation. The condition number of the the RBF kernel $~K_\varepsilon$ and the MO data matrix $\bar{~Z}~D$ are also reported.}
\label{tab:wave-interp}
\end{table}

\rev[anthony]{}{
\Cref{tab:heat-interp,tab:wave-interp}
provide additional ablations in the heat and wave cases.  At the cross-validated hyper-parameter choices, the condition numbers of the RBF kernel $~K_\varepsilon$ and the MO data matrix $\bar{~Z}~D$ are observed to be moderate in size.  For both interpolants, the median error decreases monotonically as $P$ increases, and the global GRBF interpolant attains lower error than MO at almost every sample size. This is consistent with the error bound in the newly added \Cref{rem:GRBF-proof}, in which the RBF interpolation error decays with the fill distance of the parameter samples. The errors reported in these tables are computed on unseen test snapshots interpolated from the full training snapshots; the relative error is measured in the $~M$-weighted Frobenius norm, with the median taken over test parameters.  The shape parameter $\varepsilon$ and the number of neighbors $K$ are selected using leave-one-out cross-validation on the training set.  Note that no Tucker compression is applied, so that the reported quantity isolates the interpolation error from the HOSVD reconstruction error.  
}

\subsection{Tucker ranks}

\rev[anthony]{}{
\Cref{tab:heat-tucker,tab:wave-tucker} consider the effect of HOSVD compression on accuracy, reporting the median relative Tucker reconstruction error for several core sizes, along with 
the resulting training and testing projection errors at the largest available basis size. For the heat equation, the reported errors are associated with the state snapshots. For the wave equations, the errors in both the displacement as well as momentum snapshots are reported. All errors are computed in the $~M$-weighted Frobenius norm.
}

\rev[anthony]{}{
Two observations follow. First, refining the Tucker core reduces the Tucker reconstruction error monotonically. Second, at each core size, the training projection error closely matches the Tucker reconstruction error, with exact agreement for the RBF and MO strategies due to the lack of interpolation. Note that the SVD error is also eliminated since the largest bases sizes are used (no truncation).
Hence, the Tucker reconstruction error is an empirical lower bound on the maximum achievable accuracy in absence of interpolation and local SVD truncation errors.
}

\begin{table}[htbp!]
\centering
\resizebox{\textwidth}{!}{
\begin{tabular}{lcccc}
\toprule
Tucker core & Tucker reconstruction error  & Method & Training projection error & Testing projection error \\
\midrule
$30\times30\times30$    & $6.9134\times10^{-2}$ & Monolithic & $6.9086\times10^{-2}$ & $1.0129\times10^{-1}$ \\
                        &                       & RBF        & $6.9105\times10^{-2}$ & $1.0139\times10^{-1}$ \\
                        &                       & MO         & $6.9105\times10^{-2}$ & $1.0139\times10^{-1}$ \\
\midrule
$60\times60\times60$    & $1.5695\times10^{-2}$ & Monolithic & $1.5852\times10^{-2}$ & $2.4694\times10^{-2}$ \\
                        &                       & RBF        & $1.5626\times10^{-2}$ & $2.5946\times10^{-2}$ \\
                        &                       & MO         & $1.5626\times10^{-2}$ & $2.5946\times10^{-2}$ \\
\midrule
$120\times120\times120$ & $9.2576\times10^{-4}$ & Monolithic & $9.0248\times10^{-4}$ & $3.7405\times10^{-3}$ \\
                        &                       & RBF        & $9.0169\times10^{-4}$ & $3.7351\times10^{-3}$ \\
                        &                       & MO         & $9.0169\times10^{-4}$ & $3.7351\times10^{-3}$ \\
\bottomrule
\end{tabular}}
\caption{\textbf{Heat equation:} median (across parameters) of the relative Tucker reconstruction error and the relative training- and testing-snapshot projection errors, all measured in the $~M$-weighted Frobenius norm, for different Tucker core shapes and the largest available basis size.}
\label{tab:heat-tucker}
\end{table}

\begin{table}[htbp!]
\centering
\resizebox{\textwidth}{!}{%
\begin{tabular}{lcclcccc}
\toprule
& \multicolumn{2}{c}{Tucker recon.\ error} & & \multicolumn{2}{c}{Projection error $~q$} & \multicolumn{2}{c}{Projection error $~p$} \\
\cmidrule(lr){2-3} \cmidrule(lr){5-6} \cmidrule(lr){7-8}
Tucker core & $~q$ & $~p$ & Method & Train & Test & Train & Test \\
\midrule
$30\times30\times30$    & $8.001\times10^{-2}$ & $2.272\times10^{-1}$ & Monolithic & $6.870\times10^{-2}$ & $7.683\times10^{-2}$ & $2.239\times10^{-1}$ & $2.414\times10^{-1}$ \\
                        &                      &                      & RBF        & $6.829\times10^{-2}$ & $7.786\times10^{-2}$ & $2.247\times10^{-1}$ & $2.431\times10^{-1}$ \\
                        &                      &                      & MO         & $6.829\times10^{-2}$ & $7.786\times10^{-2}$ & $2.247\times10^{-1}$ & $2.431\times10^{-1}$ \\
\midrule
$60\times60\times60$    & $2.757\times10^{-2}$ & $1.077\times10^{-1}$ & Monolithic & $2.366\times10^{-2}$ & $2.507\times10^{-2}$ & $1.055\times10^{-1}$ & $1.055\times10^{-1}$ \\
                        &                      &                      & RBF        & $2.359\times10^{-2}$ & $2.506\times10^{-2}$ & $1.054\times10^{-1}$ & $1.048\times10^{-1}$ \\
                        &                      &                      & MO         & $2.359\times10^{-2}$ & $2.506\times10^{-2}$ & $1.054\times10^{-1}$ & $1.048\times10^{-1}$ \\
\midrule
$120\times120\times120$ & $4.615\times10^{-3}$ & $2.453\times10^{-2}$ & Monolithic & $4.217\times10^{-3}$ & $4.742\times10^{-3}$ & $2.445\times10^{-2}$ & $2.626\times10^{-2}$ \\
                        &                      &                      & RBF        & $4.218\times10^{-3}$ & $4.721\times10^{-3}$ & $2.442\times10^{-2}$ & $2.616\times10^{-2}$ \\
                        &                      &                      & MO         & $4.218\times10^{-3}$ & $4.721\times10^{-3}$ & $2.442\times10^{-2}$ & $2.616\times10^{-2}$ \\
\bottomrule
\end{tabular}%
}
\caption{\textbf{Wave equation:} median (across parameters) of the relative Tucker reconstruction error and the relative training- and testing-snapshot projection errors, all measured in the $~M$-weighted Frobenius norm, for different Tucker core shapes and the largest available basis size.}
\label{tab:wave-tucker}
\end{table}



\section{Quadratic Basis Comparison}\label{app:QB}

\rev[anthony]{}{
Readers may wonder how the proposed tensorial ROM approach compares to other nonlinear ROM techniques, such as quadratic basis methods.  While these are designed to solve a different problem involving nonlinearity in the state $~x$, the expressiveness of the resulting bases can still be meaningfully compared.  To that end, we now 
study how the projection error of the proposed parameter-local approach compares to that of a state-of-the-art quadratic manifold}\footnote{It is unclear that this object is actually a manifold in any geometric sense, but we retain this terminology for consistency with prior work.}
\rev[anthony]{}{(QM) approach from \cite{schwerdtner2024greedy}.  As projection error is a lower bound on the ROM error, this provides a metric by which the expressiveness of the reduced bases can be judged.  Moreover, it is inherently favorable to QM methods: the additional error (and computational overhead) coming from solving the more challenging quadratic basis ROM is not considered. 
}

\rev[anthony]{}{
Precisely, we consider the state-of-the-art greedy QM of \cite{schwerdtner2024greedy}, applied to the model heat and wave problems in \Cref{subsec:heat,subsec:wave}. This QM approach augments a (potentially $~M$-orthonormal) linear basis $~V \in \mathbb{R}^{N \times r}$ with a quadratic term, approximating the state as
\begin{align*}
    \tilde{~x} = \bar{~x} + ~V\hat{~x} + ~W~h(\hat{~x}) \approx ~x, \qquad \hat{~x} = ~V^\intercal ~M (~x - \bar{~x}) \in \mathbb{R}^n,
\end{align*}
where $\bar{~x}$ denotes the row-wise mean of the global snapshot matrix, the feature map $~h(\hat{~x}) \in \mathbb{R}^{n(n+1)/2}$ collects the distinct quadratic monomials $\hat{~x}_a \hat{~x}_b$ for $1 \leq b \leq a \leq n$, and the matrix $~W \in \mathbb{R}^{N \times n(n+1)/2}$ collects the corresponding coefficients. The columns of $~V$ are selected greedily from the POD modes of the centered ensemble rather than by leading singular value. 
For a consistent comparison, the original approach of \cite{schwerdtner2024greedy} has been modified so that $~V$ is $~M$-orthonormal.  More precisely, let $~U$ denote a matrix containing candidates for the linear basis $~V$, obtained from the $~M$-weighted thin SVD of the column-wise (mode-1) matricization of a snapshot tensor $~X$ that is row-wise mean-centered.  Let $~S$ denote this object and $\mathcal I_{j-1}$ denote the index set of the columns of $~U$ selected to build $~V_{j-1} := ~U_{\mathcal I_{j-1}}$ at greedy iteration $j-1$. At iteration $j$, a candidate mode is appended to form $~V_j$ and the following minimization problem is solved for $~W_j$}
\begin{align}\label{eq:greedy}
    ~W_j = \arg\min_{~W} \big|(~I-~V_j~V_j^\intercal~M)~S - ~W h(~V_j^\intercal~M ~S)\big|_{~M}^2 + \gamma |~W|_{~M}^2,
\end{align}
\rev[anthony]{}{
where $\gamma$ is a regularization parameter. The candidate that yields the smallest objective value is retained to form $~V_j$. After $n$ steps, $~V := ~V_n = ~U_{\mathcal I_n}$ is the constructed basis with $~W$ minimizing the objective in \eqref{eq:greedy}. 
}

\rev[anthony]{}{
Constructing this manifold is considerably more expensive than computing a linear basis of the same dimension, since the greedy selection solves a regularized least-squares problem for every candidate mode at every step.
To reduce the cost of this procedure, the training snapshots are sub-sampled uniformly in time, retaining every fifth column in the heat case and every tenth in the wave case. This yields snapshot tensors $~X_\text{heat} \in \mathbb R^{1082 \times 240 \times 160}$ and $~X_\text{wave} \in \mathbb R^{2160 \times 100 \times 160}$, from which the greedy quadratic manifolds are constructed. For the heat equation 100 candidate POD modes are taken, while 200 candidate modes are taken for the wave equation. 
To choose an appropriate value for the regularization $\gamma$ at each $n$, five-fold cross-validation is performed over 33 logarithmically spaced candidates in $[10^{-12}, 10^4]$. The training parameter samples are divided into five subsets, each of which serves in turn as the withheld set; $~W$ is fit using the snapshots of the remaining four, and the relative reconstruction error is measured on the snapshots of the withheld one. Since each parameter is withheld exactly once, this yields one error per parameter. The $\gamma$ that minimizes the median reconstruction error over all 160 parameters is retained.
}

\rev[anthony]{}{
Once the QM is appropriately constructed at each $n$, the the relative reconstruction errors of the training and testing snapshots are recorded and plotted as functions of $n$.  The reported errors are also in the $~M$-weighted norm and therefore directly comparable to those of \Cref{sec:numerics}. The results are shown in the panels of ~\Cref{fig:QM_heat} (heat equation) and ~\Cref{fig:QM_wave} (wave equation). For comparison, the reconstruction errors of the monolithic basis constructed from uncentered snapshots, monolithic basis constructed from centered snapshots, and the two Tucker bases are also plotted. ~\Cref{fig:QM_heat} shows that the QM improves on the monolithic linear basis but only modestly; it nowhere closes the gap with the tensorial bases. At $n=20$, the median testing projection error is $12.6\%$ for the QM, $3.62\%$ for the tensorial RBF, and $4.7\%$ for the tensorial MO. For the training set, the median error remains larger than 7\% throughout for the QM but quickly drops to 0.1\% for the tensorial bases. The gains of the QM for the wave case are far less noticeable. ~\Cref{fig:QM_wave} shows that for both training and testing sets, the QM curves are barely below the monolithic curves for $~q$ and indistinguishable for $~p$. 
}

\rev[anthony]{}{
These results are anticipated and do not represent a flaw in the QM method: quadratic manifolds and related nonlinear-manifold techniques are designed to mitigate slow Kolmogorov $n$-width decay arising from nonlinearity in the \emph{state} $~x$, such as advection-dominated transport, and not from nonlinearity of the parameter-to-solution map.  In the present benchmarks the bottleneck is parametric: the forcing location varies over the domain, so a single global linear subspace must simultaneously resolve solutions concentrated in different regions.  Augmenting that global subspace with a quadratic correction reduces the error only modestly (heat) or not at all (wave), whereas adapting the subspace itself to each query parameter yields consistent improvements. This also emphasizes that the two ideas are complementary rather than competing: because the proposed construction produces a reduced basis and imposes no restriction on the subsequent projection step, a quadratic state approximation could be built on top of a parameter-local basis, representing an interesting direction for future work. 
}

\begin{figure}[htbp!]
    \centering
    \includegraphics[width=\linewidth]{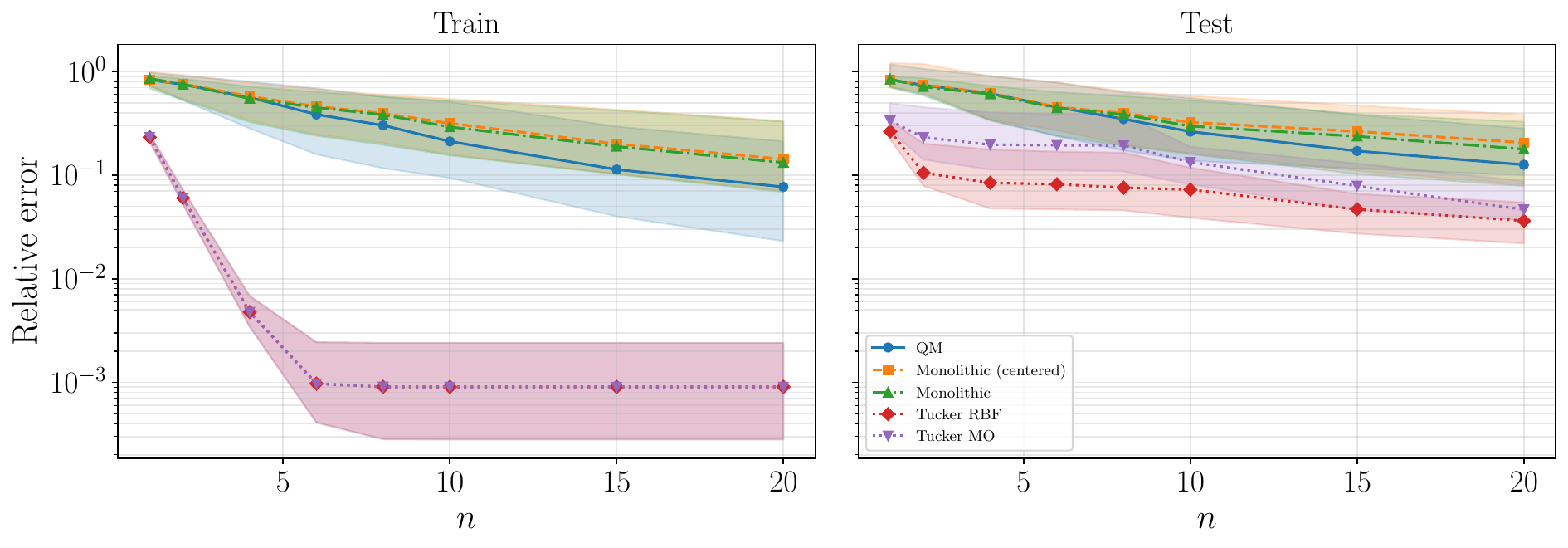}
    \caption{\textbf{Heat equation.} Relative $~M$-weighted reconstruction error of the training (left) and testing (right) snapshots versus the reduced dimension $n$, for the greedy quadratic manifold (QM), the monolithic SVD bases built from centered and uncentered snapshots, and the tensorial RBF and MO bases. Solid lines denote the median across parameter instances; shaded bands denote the interquartile range.}\label{fig:QM_heat}\label{fig:QM_heat}
\end{figure}

\begin{figure}[htbp!]
    \centering
    \includegraphics[width=\linewidth]{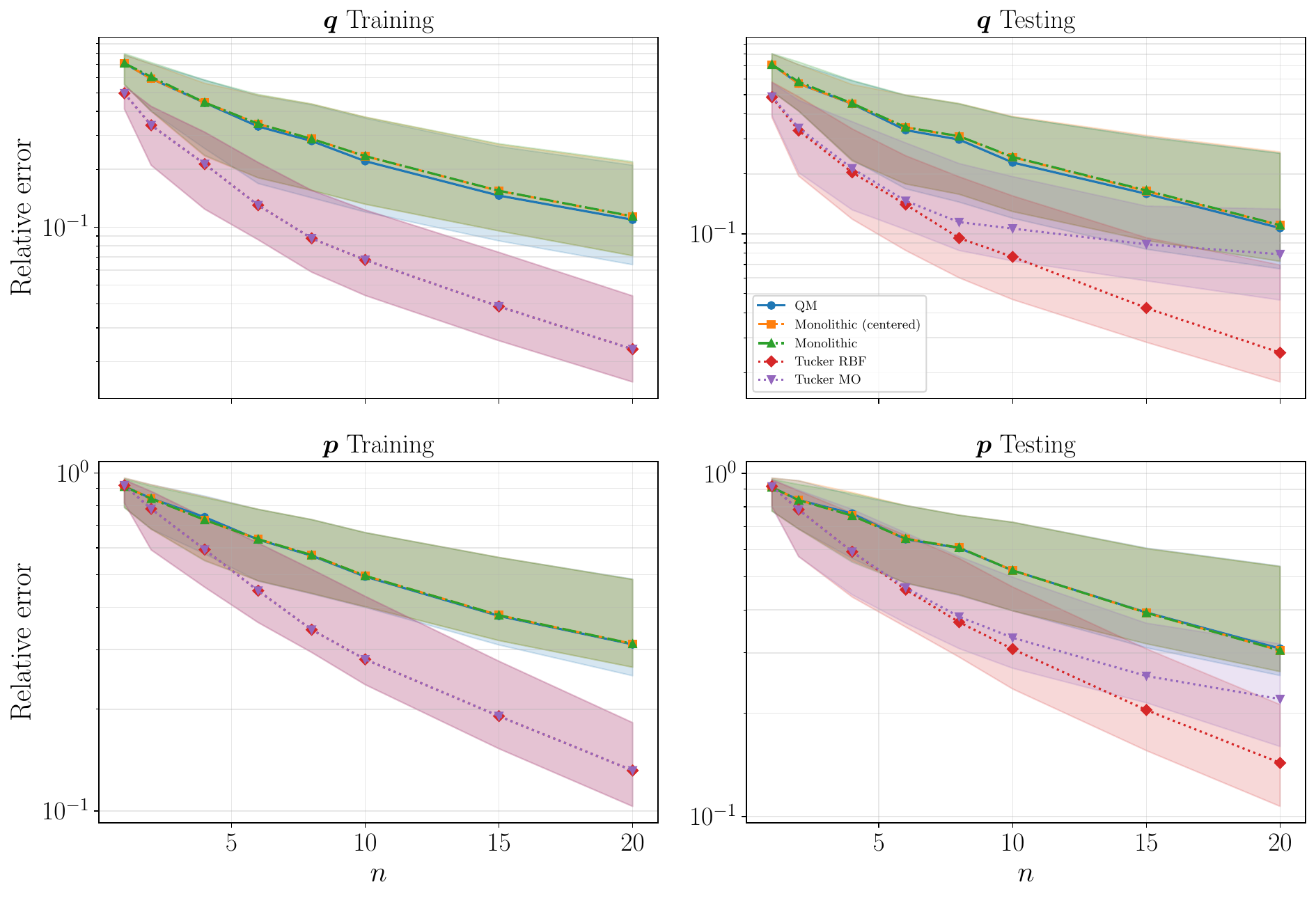}
    \caption{\textbf{Wave equation.} Relative $~M_W$-weighted reconstruction error of the training (left) and testing (right) snapshots versus the reduced dimension $n$, for the displacement $~q$ (top row) and the momentum $~p$ (bottom row). Compared are reconstructions using the greedy quadratic manifold (QM), the monolithic SVD bases built from centered and uncentered snapshots, and the tensorial RBF and MO bases. Solid lines denote the median across parameter instances; shaded bands denote the interquartile
range.}\label{fig:QM_wave}
\label{fig:QM_wave}
\end{figure}


\end{document}